\documentclass[a4paper,10pt,reqno]{amsart}
\usepackage{amsmath,amssymb,amsfonts,amsthm}
\usepackage{cite}
\usepackage{eucal}
\usepackage{enumerate}
\usepackage{array}
\usepackage{graphicx}
\usepackage{mathrsfs}
\usepackage{microtype}
\usepackage{needspace}
\usepackage{etoolbox}
\usepackage{tikz}
\usetikzlibrary{arrows.meta}
\usepackage[hypertexnames=false,pagebackref=true,colorlinks=true,linkcolor=blue,citecolor=blue,urlcolor=blue]{hyperref}
\renewcommand*{\backrefalt}[4]{%
  \ifcase #1\relax
  \or (Cited on p.~#2.)%
  \else (Cited on pp.~#2.)%
  \fi
}
\allowdisplaybreaks
\hypersetup{
  pdftitle={Stability of optimal transport on sub-Riemannian manifolds},
  pdfauthor={Bang-Xian Han and Zhuo-Nan Zhu},
  pdfsubject={Coercivity, sharp stability exponents, and logarithmic potential stability on equiregular sub-Riemannian manifolds},
  pdfkeywords={Riemannian manifolds, fat distributions, Carnot groups, optimal transport, quantitative stability,
    Kantorovich potential, horizontal semiconcavity,
    John domain}
}

\numberwithin{equation}{section}

\theoremstyle{plain}
\newtheorem{theorem}{Theorem}[section]
\newtheorem{proposition}[theorem]{Proposition}
\newtheorem{lemma}[theorem]{Lemma}
\newtheorem{corollary}[theorem]{Corollary}
\theoremstyle{definition}
\newtheorem{definition}[theorem]{Definition}
\theoremstyle{remark}
\newtheorem{remark}[theorem]{Remark}

\newcommand{\HH}{\mathbb H}

\newcommand{\RR}{\mathbb R}
\newcommand{\PP}{\mathcal P}
\newcommand{\eps}{\varepsilon}
\newcommand{\iu}{\mathrm{i}}
\newcommand{\diam}{\operatorname{diam}}
\newcommand{\dist}{\mathsf{d}}
\newcommand{\dcc}{\dist_{\mathrm{cc}}}
\newcommand{\dE}{\dist_{\mathrm E}}
\newcommand{\dR}{\dist_{\mathrm R}}
\newcommand{\Lip}{\operatorname{Lip}}
\newcommand{\osc}{\operatorname{osc}}
\newcommand{\Exp}{\operatorname{Exp}}
\newcommand{\Id}{\operatorname{Id}}
\newcommand{\spt}{\operatorname{spt}}

\newcommand{\calA}{\mathcal A}
\newcommand{\calB}{\mathcal B}
\newcommand{\calE}{\mathcal E}
\newcommand{\calD}{\mathcal D}

\newcommand{\der}{\mathrm d}
\renewcommand{\d}{\,\mathrm d}
\newcommand{\dd}{\d}
\newcommand{\restr}[1]{\lower3pt\hbox{$|_{#1}$}}

\newcommand{\dx}{{\d x}}

\newcommand{\mm}{\mathfrak m}

\newcommand{\E}{\mathbb{E}}

\newcommand{\N}{\mathbb{N}}

\providecommand{\Var}{\operatorname{Var}}
\providecommand{\Cov}{\operatorname{Cov}}
\providecommand{\one}{\mathbf 1}
\providecommand{\norm}[1]{\left\lVert #1\right\rVert}

\title[Stability of optimal transport]
{\LARGE Stability of optimal transport on sub-Riemannian manifolds}
\author{Bang-Xian Han}
\address{School of Mathematics, Shandong University, Jinan 250100, China}
\email{hanbx@sdu.edu.cn}
\author{Zhuo-Nan Zhu}
\address{School of Mathematical Sciences, University of Science and
Technology of China, Hefei 230026, China}
\email{zhuonanzhu@mail.ustc.edu.cn}
\makeatletter
\def\@settitle{%
  \begin{center}
    \baselineskip14\p@\relax
    \resizebox{\textwidth}{!}{\bfseries\@title}
  \end{center}%
}
\patchcmd{\@maketitle}{\dimen@34\p@}{\dimen@22\p@}{}%
  {\PackageError{manuscript}{Cannot adjust title spacing}{}}
\patchcmd{\@starttoc}{\addvspace{32\p@\@plus14\p@}}%
  {\addvspace{12\p@\@plus2\p@}}{}%
  {\PackageError{manuscript}{Cannot adjust contents spacing}{}}
\@namedef{subjclassname@2020}{2020 Mathematics Subject Classification}
\makeatother
\date{}

\begin{document}
\begin{abstract}
We study the stability of optimal transport for the squared
Carnot--Carath\'eodory distance. We prove the sharp $W_1^{1/4}$
stability estimate for optimal transport maps on complete Riemannian
and fat sub-Riemannian manifolds, assuming bounded source support
and an upper density bound. On John domains with
two-sided density bounds, we obtain logarithmic $L^2$ stability of
Kantorovich potentials on complete smooth sub-Riemannian manifolds
that are equiregular near the source, with arbitrary compact targets.
For potentials, we obtain every H\"older exponent $0<\alpha<1/2$
on two-generating structures, and the exponent $1/2$ in the
Riemannian and fat settings and on all step-two Carnot groups.
Examples with three target atoms show that the map and potential exponents
with respect to $W_1$ cannot exceed $1/4$ and $1/2$, respectively, on any smooth
sub-Riemannian manifold of horizontal rank at least two.
We also construct smooth transport families whose target variation
of order $t$ in $W_1$ produces map variation of order $t^{1/s}$ in an
equiregular region of step $s$. The map estimates follow from coercivity of the duality defect,
while the potential estimates follow from quantitative concavity
of the Kantorovich functional.
\end{abstract}

\maketitle
\enlargethispage{-\baselineskip}
\setcounter{tocdepth}{1}
\begingroup
\setlength{\parskip}{0pt}
\tableofcontents
\endgroup

\section{Introduction}\label{sec:introduction}
\subsection{The stability problem}
Quantitative stability in optimal transport asks how perturbations
of the target measure affect the transport map and Kantorovich
potential. Such estimates are needed, for example, when the target
is approximated by samples or a numerical discretization.
Fix a source probability measure $\rho$. For each target $\mu$,
let $T_\mu$ denote the optimal transport map, when it exists, and
let $\phi_\mu$ be a source Kantorovich potential normalized by
$\int\phi_\mu\,\d\rho=0$.
We estimate their $L^2(\rho)$ differences in terms of the
Wasserstein distance $W_1$ between targets.

For maps, we use the following $L^2$ error. Since
$(T_{\mu_0},T_{\mu_1})_\#\rho$ couples $\mu_0$ and $\mu_1$,
\[
 \begin{aligned}
 W_2(\mu_0,\mu_1)
 &\leq\norm{\dcc(T_{\mu_0},T_{\mu_1})}_{L^2(\rho)}\\
 &=\left(\int\dcc^2(T_{\mu_0}(x),T_{\mu_1}(x))\,\d\rho(x)\right)^{1/2},
 \end{aligned}
\]
where $\dcc$ is the ambient Carnot--Carath\'eodory distance.
Here $F_\#\rho$ denotes the pushforward measure,
defined by $(F_\#\rho)(A)=\rho(F^{-1}(A))$ for Borel sets $A$.
For a real-valued $f\in L^2(\rho)$, set
$\Var_\rho(f)=\int|f-\int f\,\d\rho|^2\,\d\rho$.
For potentials, this measures quadratic error modulo additive
constants. Under our zero-mean normalization,
\[
 \Var_\rho(\phi_{\mu_1}-\phi_{\mu_0})
 =\norm{\phi_{\mu_1}-\phi_{\mu_0}}_{L^2(\rho)}^2.
\]

For targets in a fixed compact set, our estimates give rates of
strong convergence of maps and potentials when the targets converge
weakly. We use the theory of optimal transport maps \cite{Brenier1991,GangboMcCann1996,McCann2001}
and quantitative results along curves of target measures
\cite{Gigli2011}.

\medskip
In sub-Riemannian geometry, admissible curves are constrained to
horizontal directions, and displacement in other directions is
generated by iterated brackets of horizontal vector fields.
The step of the structure measures the bracket order needed to
span the tangent space. We investigate how this constraint affects
transport stability.

\subsection{Background and motivation}
Kitagawa--Letrouit--M\'erigot
\cite[Theorem~1.2]{KitagawaLetrouitMerigot2025} give, to our
knowledge, the first such estimates on general Riemannian
manifolds without assuming regularity of optimal maps.
Their proofs use a regularized Kantorovich functional and variance
estimates on John domains. Both arguments also apply to metric spaces.
In \cite{HanZhu2026}, we used heat-kernel regularization to prove
potential stability on spaces satisfying the Riemannian
curvature--dimension condition ($\mathrm{RCD}$ spaces). We also
proved map stability on Alexandrov spaces.

Heisenberg groups do not satisfy the curvature--dimension condition
$\mathrm{CD}(K,N)$ for any $K\in\RR$ and finite $N$
\cite{Juillet2009}. Thus our result on $\mathrm{RCD}$ spaces does
not apply to them. Stability without curvature bounds was also
asked in \cite[Open question~4.17]{LetrouitLecture}.
We first extended our heat-kernel proof to Heisenberg groups.
It gives $L^2$ stability of source potentials with rate $W_1^{1/2}$;
see \cite[Theorem~B and Section~6]{HanZhuHeatKernelSurvey2026}.

In this paper we use M\'erigot's estimate of the map error by the
duality defect \cite{Merigot2026}. This gives map stability directly.
For potentials, we estimate the concavity gap of the Kantorovich
functional.

\subsection{Main results}\label{sec:main-results}
All geometric structures below are fixed and smooth.
A sub-Riemannian structure $(M,\mathcal H,g_{\mathcal H})$ consists of
a smooth manifold $M$, a smooth bracket-generating subbundle
$\mathcal H\subseteq TM$, and a smooth inner product $g_{\mathcal H}$
on $\mathcal H$. Bracket generation means that horizontal vector
fields and their iterated Lie brackets span each tangent space.
An absolutely continuous curve is horizontal if its velocity belongs
to $\mathcal H$ almost everywhere. The Carnot--Carath\'eodory distance $\dcc$ is the
infimum of the $g_{\mathcal H}$-lengths of absolutely continuous
horizontal curves joining the two points. Set $c=\dcc^2/2$. When $\mathcal H=TM$, $\dcc$ is the Riemannian distance and we use the same notation in both settings. 
Let $\mm$ have a smooth strictly positive density in local
coordinates. On a Carnot group, we use Haar measure normalized in
fixed exponential coordinates.
For a Borel set $A\subseteq M$, write $\PP(A)$ for the Borel probability measures on $M$ concentrated on $A$.
The distances $W_p$, $1\leq p<\infty$, always use the ambient distance. For nonnegative quantities, $f\asymp g$ means
$c_1g\leq f\leq c_2g$ for fixed constants $c_1,c_2>0$; $f\sim g$ means $f/g\to1$ in the stated limit. 
For a bounded source set $S$ and a compact target set $Y$, put
$D:=\diam_{\mathrm{cc}}(\overline S\cup Y)$.
We use $B_{\mathrm{cc}}(x,r)$ for open metric balls, and
$\lambda B=B_{\mathrm{cc}}(x,\lambda r)$ when
$B=B_{\mathrm{cc}}(x,r)$. We write $S\Subset O$ when
$\overline S$ is a compact subset of the open set $O$,
$\mm\restr S$ for the restriction of $\mm$ to $S$, and
$\#A$ for the cardinality of $A$. For a real-valued function $f$,
$\osc_A f=\sup_A f-\inf_A f$.
For a smooth function $f$, its horizontal gradient
$\nabla_{\mathcal H}f$ is characterized by
$g_{\mathcal H}(\nabla_{\mathcal H}f,X)=Xf$ for horizontal $X$.
For locally Lipschitz functions with respect to $\dcc$, we use
the almost everywhere horizontal derivatives, equivalently their
weak horizontal derivatives. The letter $C$ denotes a positive
constant whose value may change.

\medskip
A \emph{Carnot group of step $s$} is a connected, simply connected
nilpotent Lie group whose Lie algebra has a stratification
$\mathfrak g=V_1\oplus\cdots\oplus V_s$, with
$[V_1,V_j]=V_{j+1}$, $V_s\ne\{0\}$, and $V_{s+1}=\{0\}$.
An inner product on $V_1$ defines the horizontal metric by left
translation. A Carnot group of step two has Lie algebra $V_1\oplus V_2$,
with horizontal layer $V_1$ and central layer $V_2=[V_1,V_1]$.
It is \emph{fat} if $[v,V_1]=V_2$ for every nonzero $v\in V_1$. More generally, the Levi form of a horizontal distribution
$\mathcal H$ is
$\mathcal L_x(v,w)=[\widetilde v,\widetilde w](x)\bmod\mathcal H_x$,
where $\widetilde v,\widetilde w$ are local horizontal extensions
of $v,w\in\mathcal H_x$. The distribution is \emph{two-generating} if
\begin{equation}\label{eq:two-generating}
	\mathcal H_x+[\mathcal H,\mathcal H]_x=T_xM,
	\qquad x\in M.
\end{equation}
It is \emph{fat} if
$\mathcal L_x(v,\mathcal H_x)=T_xM/\mathcal H_x$
for every $x\in M$ and every nonzero $v\in\mathcal H_x$. 
The structure is \emph{equiregular} if the ranks of the spaces
$\mathcal H^j$ spanned by brackets of length at most $j$ are
constant. Its step is the least $s$ with $\mathcal H^s=TM$,
and its homogeneous dimension is
$Q=\sum_{j=1}^s j(\dim\mathcal H^j-\dim\mathcal H^{j-1})$,
where $\mathcal H^0=\{0\}$. On a Carnot group this is
$Q=\sum_j j\dim V_j$.

Write $C(Y)$ for the real-valued continuous functions on $Y$.
For $\psi\in C(Y)$, its $c$-transform is
$\psi^c(x)=\min_{y\in Y}\{c(x,y)-\psi(y)\}$.
Choose a Borel map $T:S\to Y$ such that $T(x)$ attains this
minimum at each $x$. Put $\phi=\psi^c$ and
$R(x,y)=c(x,y)-\phi(x)-\psi(y)\geq0$.
For a probability plan $\pi$ on $S\times Y$ with first marginal
$\rho$, write
\[
 \calD(\pi)=\int R\,\d\pi,\qquad
 \calE(\pi)=\int\dcc^2(T(x),y)\,\d\pi(x,y).
\]
These are the duality defect and the squared transport error.

\begin{theorem}[Sharp coercivity estimate]\label{thm:main-coercivity}
Let $M$ be one of the following spaces:
\begin{enumerate}[(i)]
\item a complete connected smooth Riemannian manifold;
\item a complete connected smooth fat sub-Riemannian manifold.
\end{enumerate}
Let $S\subseteq M$ be a bounded Borel set, $Y\subseteq M$ be compact, and $\rho\in\PP(S)$ satisfy
$\rho\leq a_2\mm\restr S$ for some $a_2>0$.
For $\psi\in C(Y)$, put $\phi=\psi^c$ and
$R(x,y)=c(x,y)-\phi(x)-\psi(y)$. Let $T:S\to Y$ be any Borel
map satisfying $R(x,T(x))=0$ for every $x\in S$. Then every probability plan $\pi$ on $S\times Y$ with first marginal
$\rho$ satisfies
\begin{equation}\label{eq:main-coercivity}
 \left(\int_{S\times Y}\dcc^2(T(x),y)\,\d\pi(x,y)\right)^2
 \leq C\int_{S\times Y}R(x,y)\,\d\pi(x,y).
\end{equation}
The constant $C$ depends only on the fixed structure, $\mm$, a
compact set $K$ containing $\overline S\cup Y$, and $a_2$.
On a fat Carnot group with Haar measure one has $C=C(G,a_2,D)$.
In particular, $C$ is independent
of $\psi$, $T$, and the second marginal of $\pi$.
When the horizontal rank is at least two, the power $2$ on the
left cannot be decreased. 
\end{theorem}

The geometric assertions are proved in
Section~\ref{sec:realizations}; Proposition~\ref{prop:fixedsharp} proves the
sharpness assertion for fixed marginals.

\begin{corollary}[Sharp stability of optimal transport maps]
\label{cor:map-stability}
Under the hypotheses of Theorem~\ref{thm:main-coercivity}, every
$\mu\in\PP(Y)$ has a unique optimal coupling with $\rho$, induced
by a map $T_\mu$, and
\begin{equation}\label{eq:main-map-stability}
 \norm{\dcc(T_{\mu_0},T_{\mu_1})}_{L^2(\rho)}
 \leq C W_1^{1/4}(\mu_0,\mu_1),
 \qquad \mu_0,\mu_1\in\PP(Y).
\end{equation}
The constant has the same dependencies as in
Theorem~\ref{thm:main-coercivity}. The exponent $1/4$ is optimal
locally on each such space whose horizontal rank is at least two.
\end{corollary}

Optimal maps exist uniquely on complete two-generating structures
for compactly supported absolutely continuous sources and arbitrary
compactly supported targets
\cite[Theorem~3.2 and Proposition~4.2]{FigalliRifford10}.
Our map estimate needs fatness to control geodesic endpoints by
their initial normal covectors. For $\mathcal H\ne TM$, fatness
implies step two. Proposition~\ref{prop:smooth-defect-obstruction}
shows that \eqref{eq:main-coercivity} fails in every equiregular
region of step greater than two. The potential estimates use
optimal plans and do not require optimal maps.

For Kantorovich potentials, the John condition allows us to combine
local variance estimates over the source domain.
A bounded open set $S\subseteq M$ is an \emph{$\eta_0$-John domain}
if it has a point $x_0$ with the following property. Every $x\in S$
can be joined to $x_0$ by a rectifiable curve
$\gamma:[0,\ell]\to S$, parametrized by ambient arc length, with
\[
 \gamma(0)=x,\qquad \gamma(\ell)=x_0,\qquad
 \dcc(\gamma(s),M\setminus S)\geq\eta_0s,
 \qquad 0\leq s\leq\ell.
\]
For compact $M$ we also allow $S=M$, using a finite connected cover
to extend the local estimates to all of $M$.

\begin{theorem}[Stability of Kantorovich potentials]
\label{thm:potential-stability}
Let $M$ be a complete connected smooth sub-Riemannian manifold,
including the Riemannian case.
Let $S\subseteq M$ be a bounded $\eta_0$-John domain and assume
that the structure is equiregular, of homogeneous dimension $Q$,
on a neighborhood $U$ of $\overline S$.
Let $Y\subseteq M$ be compact, and let $\rho\in\PP(S)$ satisfy
$a_1\mm\restr S\leq\rho\leq a_2\mm\restr S$ for some $a_1,a_2>0$.
For every $\mu\in\PP(Y)$ the source Kantorovich potential
$\phi_\mu$ is unique after normalization by
$\int_S\phi_\mu\,\d\rho=0$. For $\mu_0,\mu_1\in\PP(Y)$:
\begin{enumerate}[(i)]
\item Without further restrictions on the step or target measures,
\begin{equation}\label{eq:main-log-potential-stability}
 \norm{\phi_{\mu_1}-\phi_{\mu_0}}_{L^2(\rho)}
 \leq C\left[\log\left(e+\frac{D}{W_1(\mu_0,\mu_1)}\right)
       \right]^{-1/[2(Q+2)]},
\end{equation}
with right-hand side zero when $\mu_0=\mu_1$.
\item If the distribution on $M$ is two-generating, then, for
every $0<\alpha<1/2$,
\begin{equation}\label{eq:subcritical-potential-stability}
 \norm{\phi_{\mu_1}-\phi_{\mu_0}}_{L^2(\rho)}
 \leq C_\alpha W_1^\alpha(\mu_0,\mu_1).
\end{equation}
\item If $M$ is Riemannian, fat, or a step-two Carnot group, then
\begin{equation}\label{eq:main-potential-stability}
 \norm{\phi_{\mu_0}-\phi_{\mu_1}}_{L^2(\rho)}
 \leq C W_1^{1/2}(\mu_0,\mu_1).
\end{equation}
\item If $\min\{\#\spt\mu_0,\#\spt\mu_1\}\leq m$ for an integer
$m\geq1$, then, at arbitrary step,
\begin{equation}\label{eq:main-finite-target-stability}
 \norm{\phi_{\mu_1}-\phi_{\mu_0}}_{L^2(\rho)}
 \leq C\sqrt m\,W_1^{1/2}(\mu_0,\mu_1).
\end{equation}
\end{enumerate}
The constants depend only on the fixed structure, $\mm$, the source
neighborhood $U$, a compact $K$ containing $\overline S\cup Y$,
and $a_1,a_2,\eta_0$; in \textup{(ii)} they also depend on $\alpha$.
On a Carnot group, the dependence on the structure, $U$, and $K$
reduces to dependence on $G$ and $D$.
The exponent in \textup{(iii)} is optimal locally on every such
space whose horizontal rank is at least two.
\end{theorem}

Part~\textup{(i)} applies to every Carnot group, including the
Engel group. Equiregularity is required only near the source.
The targets and the minimizing curves may leave that neighborhood.
The logarithmic exponent need not be optimal.
The Riemannian case of \textup{(iii)} is due to
Kitagawa--Letrouit--M\'erigot
\cite[Theorem~1.2]{KitagawaLetrouitMerigot2025}.
Section~\ref{sec:potential-stability} proves all four parts.

Part~\textup{(iv)} needs only a horizontal $L^1$ Poincar\'e
inequality on the source and Lipschitz bounds on the cost.
Proposition~\ref{prop:one-finite-target} gives the precise statement.
In \textup{(ii)}, our constants may diverge as $\alpha\uparrow1/2$.
Proposition~\ref{prop:horizontal-obstruction} shows that horizontal
semiconcavity can fail away from the diagonal. It also shows that
the positive parts of directional second-difference quotients can
have unbounded $L^1$ norms along an optimal map. These facts
prevent a direct use of our proof of \textup{(iii)}.
Whether the exponent $1/2$ holds for all two-generating structures
remains open.

Part~\textup{(iii)} also applies to nonfat groups, for example
$\HH^1\times\RR$. The proof uses horizontal semiconcavity.
Full semiconcavity in source coordinates fails on this group,
even away from the diagonal; see Remark~\ref{rem:nonfat-obstruction}.
Our proof of map stability requires this stronger property.
Map stability on all Carnot groups of step two remains open.

\bigskip
The following local construction proves sharpness with the metric,
source domain, source measure, and compact target set all fixed.

\begin{theorem}[Upper bounds on stability exponents]\label{thm:simultaneous-sharpness}
Let $(M,\mathcal H,g_{\mathcal H})$ be a connected smooth
sub-Riemannian manifold without boundary, whose horizontal subbundle
$\mathcal H$ is smooth, bracket-generating, and of rank at least two.
Fix a smooth positive density $\omega$ on $M$.
Let $O\subseteq M$ be nonempty and open. There exist a John domain
$S\Subset O$, a compact set $Y\subseteq O$, and a constant $b>0$
with the following property. Fix $\rho=\omega\restr{S}/\omega(S)$.
For every sufficiently small $\varepsilon>0$, there are probability
measures $\mu_\varepsilon,\nu_\varepsilon\in\PP(Y)$, each supported
on exactly three points, such that
\begin{align}
 W_1(\mu_\varepsilon,\nu_\varepsilon)&=b\varepsilon^2,
 \label{eq:universal-W1}\\
 \|\phi_{\mu_\varepsilon}-\phi_{\nu_\varepsilon}\|_{L^2(\rho)}
 &\asymp\varepsilon,
 \label{eq:universal-potentials}\\
 \|\dcc(T_{\mu_\varepsilon},T_{\nu_\varepsilon})\|_{L^2(\rho)}
 &\asymp\varepsilon^{1/2}.
 \label{eq:universal-maps}
\end{align}
For both targets, the optimal coupling with $\rho$ is unique and
induced by a map, and the source Kantorovich potential
is unique under the normalization $\int_S\phi\,\d\rho=0$.
The cost $c$ is smooth near $\overline S\times Y$.
The comparison constants are positive and independent of $\varepsilon$.
Thus no potential estimate with $W_1$ exponent greater than $1/2$,
or map estimate with exponent greater than $1/4$, holds uniformly
over targets even for these fixed $S,Y,\rho$.
\end{theorem}

The rank condition is automatic for smooth bracket-generating
subbundles $\mathcal H\ne TM$;
in the Riemannian case it is equivalent to $\dim M\geq2$.
Under this condition, the examples prove local sharpness in
Corollary~\ref{cor:map-stability} and
Theorem~\ref{thm:potential-stability}\textup{(iii)}. On general smooth
sub-Riemannian structures, they give upper bounds on the possible
stability exponents. 

The same targets distinguish the sharpness questions for $W_1$
and $W_2$. Step~1 of the proof gives
\[
 W_2^2(\mu_\varepsilon,\nu_\varepsilon)=b\varepsilon^3,
 \qquad
 \norm{\dcc(T_{\mu_\varepsilon},T_{\nu_\varepsilon})}_{L^2(\rho)}
 \asymp W_2^{1/3}(\mu_\varepsilon,\nu_\varepsilon).
\]
Thus these examples exclude $W_2$ map exponents greater than
$1/3$, whereas they exclude $W_1$ exponents greater than $1/4$.
More generally, the proof computes the corresponding $W_p$
scaling for every $1\leq p<\infty$. These restrictions do not
assert matching positive estimates for $p>1$.

\medskip

\Needspace{8\baselineskip}
\begin{theorem}[Smooth transport families]\label{thm:smooth-response}
Let $M$ and $\omega$ be as in
Theorem~\ref{thm:simultaneous-sharpness}, and let $O\subseteq M$
be a nonempty open set on which the structure is equiregular of
step $s$. Fix an auxiliary smooth Riemannian metric and denote its distance by $\dR$.
There exist $t_0>0$, a fixed probability measure $\rho=f\omega$,
with $0\leq f\in C_c^\infty(O)$, a compact set $Y\subseteq O$,
and probability measures $\mu_t=g_t\omega$, where
\[
 g\in C^\infty((-t_0,t_0)\times M),\qquad
 g_t=g(t,\cdot)\geq0,\qquad \spt g_t\subseteq Y.
\]
The unique optimal plans from $\rho$ to $\mu_t$ are induced by
maps $T_t$ satisfying
\begin{align}
 W_1(\mu_t,\mu_0)&\asymp
 \norm{\dR(T_t,T_0)}_{L^2(\rho)}\asymp |t|,
 \label{eq:smooth-response-linear}\\
 \norm{\dcc(T_t,T_0)}_{L^2(\rho)}&\asymp |t|^{1/s}.
 \label{eq:smooth-response-weight}
\end{align}
The maps depend smoothly on $(t,x)$ near $\spt\rho$.
The cost is smooth near $\spt\rho\times Y$, and each pair in
this product is joined by a unique normal minimizing geodesic
with invertible exponential differential.
\end{theorem}

Theorems~\ref{thm:simultaneous-sharpness} and~\ref{thm:smooth-response}
imply
\[
 \alpha\leq\min\{1/4,1/s\}
\]
for any map estimate with exponent $\alpha$ that holds for every
bounded source with an upper density bound and every compact
target in such an open set. The two examples may use different
fixed sources.

\medskip
There is also a restriction on the coercivity estimate itself.
For the smooth family above, Proposition~\ref{prop:smooth-defect-obstruction}
proves, with the dual potentials and the map $T_0$ fixed at $t=0$, that
\[
 \calD(\pi_t)\asymp |t|^2,\qquad
 \calE(\pi_t)\asymp |t|^{2/s},\qquad
 \pi_t=(\Id,T_t)_\#\rho.
\]
Here $\Id$ is the identity map. Thus $\calE^p\leq C\calD$
requires $p\geq s$. In particular, \eqref{eq:main-coercivity}
fails when $s>2$, even where the cost is smooth and the minimizing
geodesics are normal.

For steps three and four, \eqref{eq:main-coercivity} fails, but
these examples leave the $W_1$ map exponent $1/4$ open. For $s\geq5$,
the smooth examples also rule out that exponent.
Whether these bounds are attained in higher step remains open.

\subsection{Coercivity and proof strategy}\label{sec:proof-strategy}
We estimate the map error by the duality defect and the potential
variance by the concavity gap. Kantorovich duality then bounds
both defects by $W_1$.

\medskip
\noindent\textbf{Coercivity for maps.}
Under the hypotheses of Theorem~\ref{thm:main-coercivity}, every
probability plan with first marginal $\rho$ satisfies
\begin{equation}\tag{C}\label{eq:intro-coercivity}
 \calE(\pi)^2\leq C\calD(\pi).
\end{equation}
The constant is independent of the target function $\psi$ and the second
marginal of $\pi$. If the second marginal is
$T_\#\rho$, then $\calD$ is the excess transport cost.
The power $2$ is sharp even with both marginals fixed when the
horizontal rank is at least two; see
Proposition~\ref{prop:fixedsharp}.

\medskip
For target functions $\theta,\theta_0,\theta_1\in C(Y)$, set
\[
 \begin{aligned}
 K_c(\theta)&=\int_S\theta^c\,\d\rho,\\
 J_c(\theta_0,\theta_1)
 &=K_c\!\left(\frac{\theta_0+\theta_1}{2}\right)
   -\frac{K_c(\theta_0)+K_c(\theta_1)}2,\qquad
 V_c(\theta_0,\theta_1)=\Var_\rho(\theta_1^c-\theta_0^c).
 \end{aligned}
\]
Concavity of $K_c$ gives $J_c\geq0$.

\Needspace{6\baselineskip}
\medskip
\noindent\textbf{Coercivity for potentials.}
Under the hypotheses of Theorem~\ref{thm:potential-stability},
there is a nondecreasing modulus $\Omega$, with $\Omega(0)=0$,
such that every $\theta_0,\theta_1\in C(Y)$ with
$\osc_Y(\theta_1-\theta_0)\leq2D^2$ satisfies
\begin{equation}\tag{P}\label{eq:intro-potential-coercivity}
 V_c(\theta_0,\theta_1)
 \leq\Omega\!\left(J_c(\theta_0,\theta_1)\right).
\end{equation}
The modulus depends only on the structure, $\mm$, $U$, $K$,
and $a_1,a_2,\eta_0$ from Theorem~\ref{thm:potential-stability}.
For a fixed constant $J_0>0$ and $0<t\leq1$, it can be chosen so that
\begin{equation}\label{eq:intro-potential-hierarchy}
 \Omega(J_0t)\leq
 \begin{cases}
 C[\log(e+1/t)]^{-1/(Q+2)},&\text{equiregular near the source},\\
 C_\alpha t^{2\alpha},\quad 0<\alpha<1/2,
       &\text{two-generating},\\
 Ct,&\text{Riemannian, fat, or a Carnot group of step two}.
 \end{cases}
\end{equation}
Here $C_\alpha$ may also depend on $\alpha$.
In particular, $J_c=0$ implies that the two source transforms
differ by a constant on $S$.

Both estimates involve the same quantity when the target varies.
For $i=0,1$, choose a $c$-conjugate Kantorovich pair
$(\phi_i,\psi_i)$ for $(\rho,\mu_i)$ and set
\[
 \Delta:=\int_Y(\psi_1-\psi_0)\,\d(\mu_1-\mu_0).
\]
Optimality gives $\Delta\geq0$. Each target potential $\psi_i$
is $D$-Lipschitz on $Y$, so Kantorovich--Rubinstein duality gives
\[
 0\leq\Delta\leq2D\,W_1(\mu_0,\mu_1).
\]
For maps, $\Delta$ is the sum of the two
comparison defects in \eqref{eq:symmetricgap}; for potentials,
\eqref{eq:optimal-jensen-gap} gives
$J_c(\psi_0,\psi_1)\leq\Delta/4$.
Since $\calE$ is the squared $L^2$ map error,
\eqref{eq:intro-coercivity} gives the exponent $1/4$. For normalized
potentials, the Lipschitz bound also gives the oscillation bound
required in \eqref{eq:intro-potential-coercivity}, and hence
\[
 \norm{\phi_{\mu_1}-\phi_{\mu_0}}_{L^2(\rho)}^2
 \leq\Omega\!\left(\frac D2W_1(\mu_0,\mu_1)\right).
\]
This is Proposition~\ref{prop:potential-conversion}.
The three bounds in \eqref{eq:intro-potential-hierarchy} give
Theorem~\ref{thm:potential-stability}\textup{(i)--(iii)}. We prove them in
\eqref{eq:equiregular-log-jensen}, \eqref{eq:subcritical-jensen},
and Theorem~\ref{thm:directional-jensen}, respectively.
The estimate in part~\textup{(iv)}, where one target has finite
support, follows separately from an $L^1$-Poincar\'e inequality
and Lipschitz bounds on the cost.

On the fixed target set $Y$, the elementary comparisons are
\[
 W_1(\mu_0,\mu_1)\leq W_2(\mu_0,\mu_1)
 \leq\bigl(DW_1(\mu_0,\mu_1)\bigr)^{1/2}.
\]
The first follows from Cauchy--Schwarz, and the second from
$\dcc^2\leq D\dcc$ on $Y\times Y$.
Thus each of our estimates also holds with $W_2$ in place of
$W_1$. Conversely, a bound by $W_2^\alpha$ implies a bound by
$W_1^{\alpha/2}$. These comparisons do not determine the sharp
exponent for $W_2$. In particular, the $W_1^{1/4}$ map estimate
does not imply a $W_2^{1/3}$ estimate.

\medskip
For maps, we follow M\'erigot's proof
\cite[Theorems~3.1 and~3.4 and Appendix~A]{Merigot2026} and use the
subdifferential estimates of Carlier--Delalande--M\'erigot
\cite{CarlierDelalandeMerigot2025}.
On Riemannian manifolds, smooth upper supports give semiconcavity
in source coordinates. We then compare subgradients of a convex
function and use the exponential map to estimate target distances,
as in \cite{KitagawaLetrouitMerigot2025}.
On fat manifolds, we apply this argument on annuli of transport
lengths. If one length lies in $[r,2r]$ and the other is at most
$\theta r$, we use Lemma~\ref{lem:mixed}. Here $0<\theta<1/4$ is
fixed. Summing over the annuli gives \eqref{eq:intro-coercivity}.
The required geometric estimates are uniform after rescaling in
privileged coordinates; see Proposition~\ref{prop:fat-annuli}.

\medskip
For potentials, we use the concavity method of
M\'erigot--Delalande--Chazal and Delalande--M\'erigot
\cite{MerigotDelalandeChazal2020,DelalandeMerigot2023}.
Kitagawa--Letrouit--M\'erigot developed this method on manifolds
\cite{KitagawaLetrouitMerigot2025}.
Entropic regularization and integration along horizontal paths
give local variance estimates. We sum them over a John domain,
following \cite{LetrouitMerigot2026,KitagawaLetrouitMerigot2025}.
This proves $V_c\leq CJ_c$ on Riemannian and fat manifolds and
Carnot groups of step two.
For a general two-generating structure, we approximate the cost
by penalized Riemannian costs \cite{Ge1993,AmbrosioRigot2004}.
We bound the cost error and the directional second derivatives.
Proposition~\ref{prop:cost-comparison} then gives every potential
exponent below $1/2$.

For arbitrary step, we approximate both source potentials by
functions affine in horizontal coordinates. Dorronsoro's theorem
on Carnot groups \cite[Theorem~34]{Ryoo2026} gives a scale at which
both errors are small. Two contact inequalities then bound the
difference of the affine gradients; see
Lemma~\ref{lem:common-approximate-contact}.
Smoothing and a Poincar\'e inequality give the logarithmic estimate.
On an equiregular manifold, we compare the rescaled structure with
its tangent Carnot groups and bound the additional errors.
This uses only a neighborhood of the source. The targets may lie
outside that neighborhood, and the minimizing curves may be abnormal.

\medskip
For sharpness, we perturb measures supported on three points,
as in the unstable transport examples of Letrouit
\cite[Section~3.1]{LetrouitUnstable2026}.
We use transport from a source with a density to a measure of
finite support, called semi-discrete optimal transport
\cite{KitagawaMerigotThibert2019,deGournayKahnLebrat2019,BansilKitagawa2022}.
The source sets assigned to individual targets are called Laguerre
cells. They are defined by the minima of $c(x,y)-\psi(y)$.
The source is fixed on a John domain, and the cost is smooth at all
transported pairs. We also construct families with smooth densities.
For these families, the ball--box theorem gives the dependence
of the map error on the step
\cite{Bellaiche1996,NagelSteinWainger1985}.

\subsection{Comparison with related work}
Early Euclidean estimates include Berman's dimension-dependent
bound \cite[Theorem~1.4 and Lemma~4.2]{Berman2021} and
M\'erigot--Delalande--Chazal's dimension-independent estimate
\cite[Theorem~3.1]{MerigotDelalandeChazal2020}.
For compact convex sources with two-sided density bounds,
Delalande--M\'erigot obtained potential exponent $1/2$ and map
exponent $1/6$ for arbitrary compact targets
\cite[Corollary~3.2 and Theorem~4.2]{DelalandeMerigot2023}.
Letrouit--M\'erigot treat John domains, strongly log-concave sources,
and densities that degenerate at the boundary, and give a non-John source
with no uniform H\"older potential estimate
\cite[Theorems~1.4, 1.7, 1.9, and~1.10]{LetrouitMerigot2026}.
Letrouit rules out map exponents above $1/4$ for bounded source
densities \cite[Theorem~2 and Section~3.1]{LetrouitUnstable2026};
M\'erigot proves the matching bound \cite[Theorem~3.1]{Merigot2026}.
More generally, Letrouit's Theorem~2 excludes $W_p$ map exponents
above $p/[2(p+1)]$ for a suitable bounded Euclidean source.

\medskip
A concurrent preprint of OpenAI
\cite[Theorem~1.1 and Proposition~2.2]{OpenAI2026Brenier}
proves the sharp $W_2^{1/3}$ map estimate for the uniform source
on a compact convex body in $\RR^d$, $d\geq2$, uniformly over
targets in a fixed compact set. Its example with three target atoms and a
fixed cube as source disproves the square-root conjecture
in~\cite[Conjecture~3]{LetrouitUnstable2026}.
In that example, $W_1\asymp a^2$, $W_2\asymp a^{3/2}$, and
the map error is of order $a^{1/2}$, so it also excludes
$W_1$ map exponents above $1/4$.
Our local examples give the same two map exponent restrictions
on every smooth sub-Riemannian manifold of horizontal rank at
least two, with a fixed John source; they do not require that
source to be convex. We do not prove a matching $W_2^{1/3}$
upper estimate in the geometric settings considered here.

\medskip
On Riemannian John domains, Kitagawa--Letrouit--M\'erigot
\cite[Theorem~1.2]{KitagawaLetrouitMerigot2025} require two-sided
density bounds and, for maps, finite boundary measure. Earlier manifold results
required special geometries or regularity of the maps, as
discussed in their introduction. They derive map stability
from potential stability by a reverse Poincar\'e inequality and
integral geometry \cite[Section~4]{KitagawaLetrouitMerigot2025}.
Our defect estimate gives the sharp Riemannian map exponent $1/4$
for a bounded Borel source with only an upper density bound.

\medskip
Ambrosio--Rigot proved existence and uniqueness of optimal maps
on Heisenberg groups \cite[Theorem~5.1]{AmbrosioRigot2004}.
Agrachev--Lee treated complete two-generating structures
\cite[Theorems~4.1 and~6.3]{AgrachevLee2009};
Figalli--Rifford's general theory also covers this setting
\cite[Theorem~3.2 and Proposition~4.2]{FigalliRifford10}.
The unrestricted higher-step problem remains open, although
Badreddine's theorem covers complete analytic rank-two structures
in dimension four with $\dim(\mathcal H+[\mathcal H,\mathcal H])=3$,
including the Engel group \cite[Theorem~1]{Badreddine2019}.
These existence theorems do not give quantitative stability.
For geometric background see
\cite{Mitchell1985,Bellaiche1996,Montgomery2002,CapognaDanielliPaulsTyson2007,
AgrachevBarilariBoscain2020,Rifford2014}.

\medskip
Abnormal minimizers can prevent full off-diagonal semiconcavity
\cite{MontanariMorbidelli2016}; even fat groups lack full diagonal
semiconcavity. Curvature--dimension theories based on diffusion
\cite{BaudoinGarofalo2017} or transport
\cite{BarilariMondinoRizzi2026} do not directly give the concavity
estimates used here.
We use horizontal semiconcavity on step-two groups
\cite[Theorem~1.1]{DragoniLiuZhang2025} and quantitative first-order
approximation at arbitrary step.

\medskip
Semi-discrete results often fix the target sites.
Bansil--Kitagawa control differences between Laguerre cells by the
$\ell^1$ distance of mass vectors independently of the masses
\cite[Theorem~1.3]{BansilKitagawa2022}, but comparison with $W_1$
is not uniform when sites collide. Their weight and Hausdorff
estimates require minimum cell masses or further geometric data
\cite[Theorem~1.10 and Corollary~1.11]{BansilKitagawa2022}.
Kitagawa--M\'erigot--Thibert's bound on the modulus of strong concavity depends on
the sites and a positive lower bound on cell masses
\cite[Theorem~5.1]{KitagawaMerigotThibert2019}.
The differentiability results of de~Gournay--Kahn--Lebrat require
smoothness and nondegeneracy \cite[Theorem~1]{deGournayKahnLebrat2019}.
Berman treats structured mesh approximations under stronger
Euclidean regularity assumptions
\cite[Section~1.1.1 and Theorem~1.1]{Berman2021}.
In Theorem~\ref{thm:potential-stability}\textup{(iv)}, one target
is arbitrary and the bound is independent of atom separation
and masses. The source and target region remain fixed while the
sites and masses may vary; cardinality enters only through the
explicit factor $\sqrt m$.
Only cost Lipschitz bounds and an $L^1$-Poincar\'e
inequality enter, so the argument applies to sub-Riemannian costs.

\medskip
Our examples prove sharpness in each fixed Euclidean dimension
at least two (Corollary~\ref{cor:carnot-euclidean-sharpness}).
M\'erigot also describes the potential exponent as optimal
\cite[Remark~1.2]{Merigot2026}. Our sharpness assertion concerns the class of John domains and
leaves open the optimal exponent for convex domains or balls.
The unit-ball examples in
\cite[Remark~3.4]{KitagawaLetrouitMerigot2025} exclude exponents
above $(d+2)/(2d)$, giving asymptotic sharpness as $d\to\infty$.
For strongly log-concave sources, logarithmic losses appear in
both estimates and Gaussian lower bounds
\cite[Theorem~1.4 and Proposition~6.10]{LetrouitMerigot2026}.
Our examples with three target points have a fixed source on a John domain
and unique optimal plans, including in the limit when the two
nearby target points coincide.

\medskip
Gallou\"et--M\'erigot--Thibert prove stability under strong
$c$-concavity of Kantorovich potentials \cite{GallouetMerigotThibert2026}.
For estimates with regular densities on uniformly convex domains,
see \cite{CajaLopezDelgadinoKitagawa2026}. Stability for varying
Euclidean marginals is studied in \cite{CazellesPauwelsPortales2026}.
Our targets may be atomic, and our cost need not be globally smooth.
Related work concerns barycenters
\cite{CarlierDelalandeMerigot2024Barycenters,CarlierGuerinMerigot2026,HanZhuBarycenters2026},
including our estimates on Alexandrov spaces; pushforwards and plans
\cite{CarlierDelalandeMerigot2025,LiNochetto2021}; Euclidean power
costs, with potential exponent $1/2$ for $p\geq2$
\cite[Theorem~1.1]{MischlerTrevisan2026}; and entropic approximation
\cite{Leonard2012,NutzWiesel2022,DivolNilesWeedPooladian2025}.
Ford treats Euclidean plan stability and coercivity under
bounds on the source mass of metric balls, including perturbations of both marginals
\cite{Ford2026Plans}, and quantitative uniqueness of potentials
\cite{Ford2026Potentials}.
Sub-Riemannian work also studies transport densities, Beckmann
and congestion problems \cite{CircelliClop2025,CircelliCitti2026},
and small-noise Schr\"odinger approximations
\cite{AduElamvazhuthiGharesifard2026}.

\subsection{Organization of the paper}
Section~\ref{sec:coercivity} proves the duality defect estimate
under hypotheses stated in source coordinates or on annuli of
source and target pairs.
Section~\ref{sec:realizations} verifies their geometric assumptions
and establishes the horizontal cost bounds needed for potential
stability.
Section~\ref{sec:potential-stability} proves the potential estimates
using directional bounds, cost approximation, and quantitative
affine approximation at arbitrary step, together with a separate
first-order argument when one target is finitely supported.
Section~\ref{sec:sharpness} constructs local sharpness examples,
proves sharpness of the defect estimate with fixed marginals,
and studies smooth transport families in higher step. 

\section{Coercivity for optimal transport maps}\label{sec:coercivity}
We adopt the direct defect estimate of M\'erigot
\cite[Theorem~3.4 and Appendix~A]{Merigot2026} and the quantitative
subdifferential estimates of Carlier--Delalande--M\'erigot
\cite{CarlierDelalandeMerigot2025}.
We state the required hypotheses in source coordinates and on
annuli of source and target pairs, then verify them in the
geometric settings of Section~\ref{sec:realizations}.

\subsection{Kantorovich duality and the duality defect}
Throughout this section, $(X,\dist)$ denotes a metric
space; in the applications $\dist=\dcc$. Source sets $S$ are
Borel with compact closure, and target sets $Y$ are compact.
We write $\Pi(\rho,\mu)$ for the probability measures on the product
with marginals $\rho$ and $\mu$, and put
$\operatorname{OT}_c(\rho,\mu)=\inf_{\pi\in\Pi(\rho,\mu)}\int c\,\d\pi$.
In particular, $W_1(\mu,\nu)=\inf_{\pi\in\Pi(\mu,\nu)}\int\dist\,\d\pi$.
More generally, for $1\leq p<\infty$, we use
$W_p^p(\mu,\nu)=\inf_{\pi\in\Pi(\mu,\nu)}\int\dist^p\,\d\pi$;
the transport maps and potentials still correspond to $c=\dist^2/2$.
Also, $\Id$ denotes the identity map, $\delta_y$ the Dirac mass at
$y$, $\spt\sigma$ the support of a measure $\sigma$, and
$\one_A$ the indicator function of $A$.
For an integrable
function $f$ and a finite signed measure $\sigma$, write
$\E_\sigma f=\int f\,\d\sigma$. If $\sigma$ is a probability measure,
set $\Var_\sigma(f)=\E_\sigma[(f-\E_\sigma f)^2]$ and
$\Cov_\sigma(f,g)=\E_\sigma[(f-\E_\sigma f)(g-\E_\sigma g)]$.
For a function on $A$, $\Lip(f;A)$ denotes its Lipschitz constant
for the ambient distance, and $\Lip(A,\dist)$ is the space of
real-valued Lipschitz functions on $A$.
We write $\osc_A f=\sup_A f-\inf_A f$. 
For $0<\mm(A)<\infty$, the mean is
$f_{A,\mm}=\mm(A)^{-1}\int_A f\,\d\mm$.
The Kantorovich--Rubinstein formula \cite{Villani2009} is
\begin{equation}\label{eq:KR-duality}
 W_1(\mu,\nu)=\sup_{\Lip(u;X)\leq1}\int_Xu\,\d(\mu-\nu).
\end{equation}

\medskip
For a continuous cost on $\overline S\times Y$, Kantorovich duality
\cite{Villani2009} states that
\[
 \operatorname{OT}_c(\rho,\mu)
 =\max_{\substack{\phi\in C(\overline S),\ \psi\in C(Y)\\
                  \phi(x)+\psi(y)\leq c(x,y)}}
       \left\{\int_S\phi\,\d\rho+\int_Y\psi\,\d\mu\right\}.
\]
The two functions in a maximizing pair are called
\emph{Kantorovich potentials}: $\phi$ is the source potential and
$\psi$ is the target potential.
They may be chosen as mutual $c$-transforms:
\begin{equation}\label{eq:mutual-c-transforms}
 \phi=\psi^c,\quad
 \psi^c(x)=\min_{y\in Y}\{c(x,y)-\psi(y)\},\qquad
 \psi=\phi^c,\quad
 \phi^c(y)=\inf_{x\in S}\{c(x,y)-\phi(x)\}.
\end{equation}
We call functions satisfying \eqref{eq:mutual-c-transforms}
$c$-conjugate. We always choose Kantorovich potentials in this form.
The same target infimum extends $\phi$ to source neighborhoods.
For any $\phi=\psi^c$, the set
$\partial^c\phi(x)=\{y\in Y:\phi(x)+\psi(y)=c(x,y)\}$
is called the contact set. When $\phi,\psi$ are Kantorovich potentials,
a coupling $\pi\in\Pi(\rho,\mu)$ is optimal if and only if
$\phi(x)+\psi(y)=c(x,y)$ for $\pi$-almost every $(x,y)$. The following result is classical. 

\begin{lemma}[Lipschitz bounds]\label{lem:transport-foundations}
For $c=\dist^2/2$ and
$D=\diam(\overline S\cup Y)$, every conjugate pair satisfies
\begin{equation}\label{eq:potential-Lipschitz}
 \Lip(\phi;S)\leq D,\qquad \Lip(\psi;Y)\leq D.
\end{equation}
\end{lemma}
\begin{proof}
For $x,x'\in\overline S$ and $y\in Y$, the triangle inequality gives
\[
 |c(x,y)-c(x',y)|
 \leq\tfrac12\bigl(\dist(x,y)+\dist(x',y)\bigr)\dist(x,x')
 \leq D\dist(x,x').
\]
Taking infima in the $c$-transform preserves this Lipschitz bound.
The same argument in the target variable proves the bound for $\psi$.
\end{proof}

\medskip
For an arbitrary continuous target function $\psi$, put $\phi=\psi^c$
and select $T(x)\in\partial^c\phi(x)$ measurably. Compactness and
continuity give such a Borel map
\cite[Theorem~35.46(ii)]{Kechris1995}. For a probability plan $\pi$
whose first marginal is $\rho$, define $ R(x,y):=c(x,y)-\phi(x)-\psi(y)$ and 
\begin{equation}\label{eq:defs}
 \begin{aligned}
 \calD(\pi):=\int_{S\times Y}R\,\d\pi,\qquad
 \calE(\pi):=\int_{S\times Y}\dist^2(T(x),y)\,\d\pi(x,y).
 \end{aligned}
\end{equation}
Then $R\geq0$ and $R(x,T(x))=0$. Our goal is the coercivity
estimate \eqref{eq:intro-coercivity}, with a constant independent
of $\psi$ and of the second marginal of $\pi$.

\medskip
All convexity, semiconcavity, and subdifferentials in coordinate
domains are Euclidean; we use the standard calculus of
\cite{CannarsaSinestrari2004}.
We write $\mathcal L^d$ for Lebesgue measure, $B_{\mathrm E}(x,r)$
for a Euclidean ball, and $\diam_{\mathrm E}$ for Euclidean diameter.
For a Euclidean measurable set $A$, $|A|=\mathcal L^d(A)$.
A function $f$ on a convex set is $\Lambda$-semiconcave if
$f(z)-\Lambda|z|^2/2$ is concave.
Its superdifferential $D^+f(z)$ consists of the covectors $p$ such that
$f(z+h)\leq f(z)+p\cdot h+o(|h|)$ as $h\to0$.
A smooth upper support at $z$ is a smooth function that is at least
$f$ near $z$ and equals $f$ at $z$.
For a convex function $u$, $\partial u(z)$ consists of the vectors
$q$ such that $u(w)\geq u(z)+q\cdot(w-z)$ throughout its domain.
For a set $E$ in that domain, put
$\partial u(E):=\bigcup_{z\in E}\partial u(z)$.
The oscillation $x\mapsto\diam_{\mathrm E}\partial u(B_{\mathrm E}(x,s))$
is lower semicontinuous: each fixed pair of subgradients counted
at $x$ is still counted at every sufficiently nearby center.
It is therefore measurable.
\Needspace{5\baselineskip}
\subsection{Convex subdifferential estimates}
We use a local form of the Br{\o}ndsted--Rockafellar
theorem \cite{BrondstedRockafellar1965}.
It replaces an approximate subgradient, defined by a local
inequality, with an exact subgradient at a nearby point.
The local domains in the statement allow us to apply it within
source charts.

 \begin{lemma}[Localized Br{\o}ndsted--Rockafellar theorem]\label{lem:localized-brondsted-rockafellar}
 	Let \(U\subseteq \RR^d\) be open and convex and let \(u:U\to\RR\) be convex. Let $s>0$, assume that $\overline{B_{\mathrm E}(x,2s)}\subseteq U$, and suppose there exist \(q\in\RR^d\), \(\varepsilon\ge0\), such that
 	\begin{equation}
 		u(z)\ge u(x)+q\cdot(z-x)-\varepsilon,
 		\qquad 
 		\forall z\in \overline{B_{\mathrm E}(x,2s)}.
 	\end{equation}
 	Then there exist
 	\begin{equation}
 		x_s\in \overline{B_{\mathrm E}(x,s)},
 		\quad
 		q_s\in \partial u(x_s), \quad \text{such that} \quad 	|q_s-q|\le \frac{\varepsilon}{s}. 
 	\end{equation}
 \end{lemma}
 
 \begin{proof}
Set $F(z):=u(z)-q\cdot z$. If $\varepsilon=0$, then $x$ is a
local minimizer of the convex function $F$, so $0\in\partial F(x)$;
take $x_s=x$ and $q_s=q$.
For $\varepsilon>0$, put $\alpha:=\varepsilon/s$ and minimize
$F(z)+\alpha|z-x|$ over $\overline{B_{\mathrm E}(x,2s)}$.
A minimizer $x_s$ exists by continuity and compactness. The assumption
and minimality yield 
\[
 F(x)-\varepsilon+\alpha|x_s-x|
 \leq F(x_s)+\alpha|x_s-x|\leq F(x),
\]
hence $|x_s-x|\leq s$. Thus $x_s$ is an interior point of the
constraint ball. The convex optimality condition gives
$r\in\partial F(x_s)$ with $|r|\leq\alpha$.
Taking $q_s:=q+r\in\partial u(x_s)$ proves the claim.
 \end{proof}

\medskip
The following result is a localized integrated form of the
subdifferential estimate of Carlier--Delalande--M\'erigot
\cite[Lemma~2.3 and the proof of Theorem~2.1]{CarlierDelalandeMerigot2025}.
Their argument controls subgradient oscillation through the
distributional Laplacian of a convex function. We reproduce it to
specify the domains and constants needed below. We adopt M\'erigot's use
of these oscillation bounds to control duality defects
\cite[Theorems~2.3 and~3.4]{Merigot2026}.
\begin{lemma}[Oscillation of convex subdifferentials]\label{lem:subdifferential-oscillation}
	Let \(K\Subset U\subseteq\RR^d\), where \(K\) is compact and
	\(U\) is open, bounded and convex, and let \(u:U\to\RR\) be convex and
	Lipschitz. There exists \(s_0=s_0(K,U)>0\) such that for any \(0<s<s_0\),
	\begin{equation}
		\label{eq:integrated-subgradient-oscillation}\int_K \diam_{\mathrm E}\bigl(\partial u(B_{\mathrm E}(x,s))\bigr)\,\dx
		\le C(d,K,U) \Lip(u;U)s.
	\end{equation}
\end{lemma}
\begin{proof}
\textbf{Step 1: Normalize by an affine function.} Choose an open set \(U'\) such that \(K\Subset U'\Subset U\), and choose
\(s_0>0\) small enough that $B_{\mathrm E}(x,4s)\subseteq U'$ for all $x\in K$ and $0<s<s_0$.
Fix \(x\in K\) and \(0<s<s_0\). Define 
	     \begin{equation*}
	     	A_{x,s}:={\rm osc}_{ \overline{B_{\mathrm E}(x,2s)}}(\tilde u), \quad\tilde u(z):=u(z)-a_{x,s}\cdot z, \quad a_{x,s}:=\frac{1}{|B_{\mathrm E}(x,4s)|}\int_{B_{\mathrm E}(x,4s)} \nabla u(z)\,\d z. 
	     \end{equation*}
		Subtracting an affine function only translates the subdifferential,
		so it does not change its diameter.
	
	For any  \(y\in B_{\mathrm E}(x,s)\) and \(p\in \partial \tilde u(y)\), if \(p\ne0\), then $y+s\frac{p}{|p|}\in B_{\mathrm E}(x,2s)$. Since \(p\in\partial \tilde u(y)\), the subgradient inequality yields 
	    \begin{equation}\label{eq:subgradient-upper}
	    	\tilde u\left(y+s\frac{p}{|p|}\right)
	    	\ge
	    	\tilde u(y)+s|p|,  \quad \text{and hence} \quad	|p|\le \frac{A_{x,s}}{s}.
	    \end{equation}
	    The same estimate holds for $p=0$.
	    If $A_{x,s}=0$, \eqref{eq:subgradient-upper} makes the local
	    subdifferential diameter zero, so its contribution to the desired
	    estimate vanishes. We now assume $A_{x,s}>0$. 
	    
	    On the other hand, let \(y_0,y_1\in \overline{B_{\mathrm E}(x,2s)}\) be such that $ A_{x,s}=\tilde u(y_1)-\tilde u(y_0)$. If $A_{x,s}>0$, then for \(p_1\in \partial \tilde u(y_1)\), the subgradient inequality yields 
	    \begin{equation}\label{eq:subgradient-lower}
	    	\tilde u(y_0)\ge \tilde u(y_1)+p_1\cdot(y_0-y_1),  \quad \text{and hence} \quad|p_1|\ge \frac{A_{x,s}}{4s}.
	    \end{equation}
	   
		   	\Needspace{4\baselineskip}

\medskip
\textbf{Step 2: A lower bound for the gradient on a subset.}
		   	\par\smallskip\noindent Define
	   	$E:=B_{\mathrm E}\left(y_1+s\frac{p_1}{|p_1|},\frac{s}{4}\right)
		    	\subseteq B_{\mathrm E}(x,4s)$. For any $z\in E$, by \eqref{eq:subgradient-lower}, we have 
	   	\begin{equation*}
		   		p_1\cdot(z-y_1)\ge \frac{3}{4} s |p_1| \ge \frac{3}{16}A_{x,s}, \quad\text{and hence}\quad 	\tilde u(z)-\tilde u(y_1)
	   		\ge\frac{3}{16}A_{x,s}. 
	   	\end{equation*}
		  	Convexity then yields 
	   \begin{equation*}
	   	\tilde u(z)-\tilde u\left(\frac{z+y_1}{2}\right)
	   	\ge
	   	\frac12\bigl(\tilde u(z)-\tilde u(y_1)\bigr)
	   	\ge
	   	\frac{3}{32}A_{x,s}. 
	   \end{equation*}
	   	
Let \(z_t=y_1+t(z-y_1)\), \(t\in[0,1]\).
The restriction of $\tilde u$ to each segment is Lipschitz and
hence absolutely continuous. For each $t\in[1/2,1]$, the affine map $z\mapsto z_t$
preserves null sets. Fubini's theorem therefore shows that, for almost
every $z\in E$, the derivative equals
$\nabla\tilde u(z_t)\cdot(z-y_1)$ for almost every $t\in[1/2,1]$.
Consequently, for almost every $z\in E$,
	   	\begin{equation*}
	   		\frac{3}{32}A_{x,s}
	   		\le
	   		\int_{\frac{1}{2}}^1 \nabla \tilde u(z_t)\cdot (z-y_1)\,\d t
	   		\le
	   		\frac{5}{4}s\int_{\frac{1}{2}}^1 |\nabla \tilde u(z_t)|\,\d t .
	   	\end{equation*}
	   	Integrating with respect to $z\in E$ implies 
	   	\begin{equation*}
	   		A_{x,s} |E|
	   		\le
	   		\frac{40}{3} s \int_E\int_{\frac{1}{2}}^1
	   		|\nabla \tilde u(z_t)|\,\d t\,\d z
	   		\le
	   	C(d) s \int_{B_{\mathrm E}(x,4s)}
	   		|\nabla\tilde u(w)|\,\d w. 
	   	\end{equation*}
	   	Therefore, we have  
	   	\begin{equation*}
	   	A_{x,s} \leq C(d)s^{1-d}\int_{B_{\mathrm E}(x,4s)}
	   		|\nabla\tilde u(w)|\,\d w .
	   	\end{equation*}

	  	\medskip
	  	\textbf{Step 3: Integrate the gradient lower bound.}
	  	Note that by \eqref{eq:subgradient-upper},
	   	\begin{equation*}
	   	\diam_{\mathrm E}\bigl(\partial u(B_{\mathrm E}(x,s))\bigr)=\diam_{\mathrm E}\bigl(\partial \tilde u(B_{\mathrm E}(x,s))\bigr) \leq  \frac{2A_{x,s}}{s}\leq C(d)s^{-d}\int_{B_{\mathrm E}(x,4s)}
	   			|\nabla\tilde u(w)|\,\d w .
	   	\end{equation*}
	   		   	
	We use the convexity argument of
\cite[Lemma~2.3 and the proof of Theorem~2.1]{CarlierDelalandeMerigot2025}.
Since $u$ is convex, $\nabla u\in BV_{\mathrm{loc}}(U)$ and $D^2u$
is a positive semidefinite matrix-valued Radon measure; see also
\cite[Section~6.3]{EvansGariepy2015}.
	Here $BV_{\mathrm{loc}}$ denotes local bounded variation,
	$D^2u$ is the distributional Hessian, and $|D^2u|$ is its total
	variation measure. The trace $\Delta u$ is the distributional Laplacian.
	Consequently,
	$|D^2u|\le C(d)\operatorname{tr}(D^2u)=C(d)\Delta u$. The standard $BV$--Poincar\'e inequality then yields 

	\begin{equation*}
		\int_{B_{\mathrm E}(x,4s)}|\nabla\tilde u(w)|\,\d w=\int_{B_{\mathrm E}(x,4s)}
		|\nabla u(w)-a_{x,s}|\,\d w\leq C(d)s|D^2u|(B_{\mathrm E}(x,4s)). 
	\end{equation*}
		Combining these estimates gives
	\begin{equation*}
			\diam_{\mathrm E}\bigl(\partial u(B_{\mathrm E}(x,s))\bigr)\leq C(d)s^{1-d}(\Delta u)(B_{\mathrm E}(x,4s)).
	\end{equation*}
	 Integrating over $K$ and applying Fubini's theorem, we obtain 
	\begin{equation*}
			\begin{aligned}
				\int_K \diam_{\mathrm E}\bigl(\partial u(B_{\mathrm E}(x,s))\bigr)\,\dx
				\leq C(d)s^{1-d}
				\int_{U'}\bigl|K\cap B_{\mathrm E}(y,4s)\bigr|\,\d\Delta u(y)
				\leq C(d)s(\Delta u)(U').
			\end{aligned}
	\end{equation*}
	
	Finally, choose $\chi\in C_c^\infty(U)$ with $0\leq\chi\leq1$ and
	$\chi=1$ on $U'$.  Distributional integration by parts then yields 
	\begin{equation*}
		(\Delta u)(U')
		\leq\int_U\chi\,\d\Delta u
		=-\int_U\nabla\chi\cdot\nabla u\,\d x
		\leq\norm{\nabla\chi}_{L^1(U)}\Lip(u;U),
	\end{equation*}
	which proves the claim. 
\end{proof}

\subsection{Coercivity in source coordinates}
We adapt M\'erigot's defect argument
\cite[Theorem~3.4]{Merigot2026} to costs that are semiconcave
in source coordinates. Convexification allows us to apply
the preceding subdifferential estimates. We also require
an estimate controlling the squared distance between target
points by differences of selected cost supergradients
at the same source point. 

\begin{definition}[Semiconcavity and endpoint estimates]\label{def:regular}
Let $\chi:V\to U$ be a source chart with open, bounded,
and convex image $U\subseteq\RR^d$, and let $Y$ be a
compact target set.
Assume that the continuous cost $c:V\times Y\to\RR$
satisfies the following conditions with constants
$L,\Lambda,B$: 
\begin{enumerate}[(R1)]
\item For every $y\in Y$, the function $f_y(z)=c(\chi^{-1}z,y)$ is $L$-Lipschitz and $\Lambda$-semiconcave on the bounded convex set $U$. 
\item There exist nonempty sets $P(z,y)\subseteq D^+f_y(z)$
for $(z,y)\in U\times Y$ such that
\begin{equation}\label{eq:recovery}
	\dist^2(y,w)\leq B|p-p'|
\end{equation}
for all $z\in U$, $y,w\in Y$, $p\in P(z,y)$,
and $p'\in P(z,w)$. 
\end{enumerate}
The elements of $P(z,y)$ are supergradients of $f_y$ at $z$.
The differential of a smooth upper support belongs to $D^+f_y(z)$.
We do not require $P(z,y)$ to contain every element of $D^+f_y(z)$.
\end{definition}

\begin{theorem}[Coercivity in source coordinates]\label{thm:regular}
	Let $M$ be a smooth $d$-dimensional manifold. Let $S\subseteq M$ be a bounded Borel set, $Y\subseteq M$ be compact, and $\rho\in\PP(S)$. Suppose there are finitely many source charts
	$\chi_j:V_j\to U_j$, where $U_j\subseteq\RR^d$ is open,
	bounded, and convex, and compact sets $K_j\Subset U_j$ such that $\overline S\subseteq\bigcup_j\chi_j^{-1}(K_j)$. 
	Assume that $c$ is continuous on each $V_j\times Y$
	and satisfies Definition~\ref{def:regular} in these charts
	with common constants $L,\Lambda,B$, and that
	\[
	(\chi_j)_\#\bigl(\rho|_{S\cap\chi_j^{-1}(K_j)}\bigr)
	\leq A_j\mathcal L^d.
	\]
	
	For every $\psi\in C(Y)$, set $\phi=\psi^c$.
	Then \eqref{eq:intro-coercivity} holds for every Borel map
	$T:S\to Y$ satisfying $\phi(x)+\psi(T(x))=c(x,T(x))$
	for all $x\in S$, and every probability plan $\pi$
	on $S\times Y$ with first marginal $\rho$.
	The constant depends only on the charts and domains
	$(\chi_j,K_j,U_j)$, $L,\Lambda,B$, $\sum_jA_j$,
	and $\diam(Y)$, and is independent of $\psi$, $T$, and $\pi$.
\end{theorem}

\begin{proof}
	Choose a Borel partition $S=\bigsqcup_j S_j$ with
	$S_j\subseteq\chi_j^{-1}(K_j)$.
	Using the $c$-transform extension of $\phi$, set 
	\begin{equation*}
		\begin{aligned}
			u_j(z)=\tfrac{\Lambda}{2}|z|^2-\phi(\chi_j^{-1}z),\qquad 
			a_{j,y}(z)=\tfrac{\Lambda}{2}|z|^2
			-c(\chi_j^{-1}z,y)+\psi(y).
		\end{aligned}
	\end{equation*}
	By (R1), each $a_{j,y}$ is convex and $u_j=\sup_y a_{j,y}$.
	Their Lipschitz constants have a common bound independent of $\psi$.
	
	Fix $x\in S_j$, $y\in Y$, and write $z=\chi_j(x)$,
	$w=T(x)$.
	Choose $p_y\in P(z,y)$ and $p_w\in P(z,w)$.
	By convexity and the contact condition, 
	\[
	q_y:=\Lambda z-p_y\in\partial a_{j,y}(z),
	\qquad
	q_w:=\Lambda z-p_w\in\partial u_j(z).
	\]
	Since $u_j(z)-a_{j,y}(z)=R(x,y)$, $q_y$ is an
	$R(x,y)$-subgradient of $u_j$.
	For $0<s<s_0$, chosen sufficiently small
	uniformly over the charts,
	Lemma~\ref{lem:localized-brondsted-rockafellar} at scale $s/2$ then yields 
	\[
	q_s\in\partial u_j(B_{\mathrm E}(z,s)),
	\qquad |q_s-q_y|\leq 2R(x,y)/s.
	\]
	Put $\omega_{j,s}(z):=
	\diam_{\mathrm E}\partial u_j(B_{\mathrm E}(z,s))$. By (R2), we have  
	\begin{equation}\label{eq:pointwise}
		\dist^2(y,T(x))
		\leq B|q_y-q_w|
		\leq B\left(\frac{2R(x,y)}s+\omega_{j,s}(z)\right).
	\end{equation}
	
	Using the first marginal of $\pi$, the density bound, and
	Lemma~\ref{lem:subdifferential-oscillation}, after decreasing $s_0$,
	we obtain
	\[
	\int_{S_j\times Y}\omega_{j,s}(\chi_j(x))\,\d\pi
	\leq A_j\int_{K_j}\omega_{j,s}(z)\,\d z
	\leq C A_j s.
	\]
	Integrating \eqref{eq:pointwise} and summing over $j$ yields
	\begin{equation}\label{eq:balance}
		\calE(\pi)\leq C\left(\frac{\calD(\pi)}s+s\right),
		\qquad 0<s<s_0.
	\end{equation}
	For $0<\calD(\pi)<s_0^2$, take $s=\sqrt{\calD(\pi)}$.
	For $\calD(\pi)=0$, let $s\downarrow0$.
	For $\calD(\pi)\geq s_0^2$, use
	$\calE(\pi)\leq\diam(Y)^2$.
	This proves \eqref{eq:intro-coercivity} with the stated dependence
	of the constant.
\end{proof} 

\begin{remark}\label{rem:twist}
	Semiconcavity alone does not imply \eqref{eq:intro-coercivity}.
	Take $c\equiv0$ and $\psi\equiv0$, so that $\phi\equiv0$.
	For distinct points $y_0,y_1\in Y$, the map $T\equiv y_0$ and the plan $\pi=\rho\otimes\delta_{y_1}$ satisfy
	\[
	\calD(\pi)=0,\qquad
	\calE(\pi)=\dist^2(y_0,y_1)>0.
	\]
	In the geometric applications, (R2) is verified using
	the initial covectors of normal minimizing geodesics
	and estimates for the normal exponential map.
\end{remark} 

\subsection{Consequences for optimal transport}
To deduce stability under target perturbations from the defect estimate,
we use M\'erigot's argument based on the symmetrized duality gap
\cite[proof of Theorem~3.1]{Merigot2026}. We apply it to the squared metric cost and to plans with
arbitrary second marginal.

\begin{corollary}\label{cor:coercivity-consequences}
Let $c=\dist^2/2$ and assume that \eqref{eq:intro-coercivity} holds
with a common constant $C$ for every $\psi\in C(Y)$,
every Borel map $T:S\to Y$ satisfying
$\phi(x)+\psi(T(x))=c(x,T(x))$ with $\phi=\psi^c$, and every probability plan $\pi$
on $S\times Y$ with first marginal $\rho$. 
Then the optimal coupling from $\rho$ to every
$\mu\in\PP(Y)$ is unique and induced by a map $T_\mu$, and
\[
 \norm{\dist(T_{\mu_0},T_{\mu_1})}_{L^2(\rho)}
 \leq C'W_1^{1/4}(\mu_0,\mu_1).
\]
For $\mu=T_\#\rho$ and every $\pi\in\Pi(\rho,\mu)$,
\begin{equation}\label{eq:excess}
 \begin{aligned}
 \calD(\pi)=\int c\,\d\pi-\operatorname{OT}_c(\rho,\mu),\qquad 
 \calE(\pi)^2\leq C
 \left(\int c\,\d\pi-\operatorname{OT}_c(\rho,\mu)\right).
 \end{aligned}
\end{equation}
\end{corollary}
\begin{proof}
	By compactness, the optimal transport plan exists. 
	For each $\mu\in\PP(Y)$, choose a $c$-conjugate Kantorovich
	pair $(\phi_\mu,\psi_\mu)$ and a Borel map $T_\mu$ satisfying
	$\phi_\mu(x)+\psi_\mu(T_\mu(x))=c(x,T_\mu(x))$.
	Every optimal plan $\gamma\in\Pi(\rho,\mu)$ has zero defect,
	so \eqref{eq:intro-coercivity} implies $\calE(\gamma)=0$.
	Since its first marginal is $\rho$,  it follows that 
	$\gamma=(\Id,T_\mu)_\#\rho$.
	In particular, $(T_\mu)_\#\rho=\mu$. 
	Thus the optimal transport plan is unique and induced by a map,
	which we denote by $T_\mu$. For $i=0,1$, denote the corresponding potentials by
	$(\phi_i,\psi_i)$, and set $T_i=T_{\mu_i}$ and
	$R_i=c-\phi_i-\psi_i$.
	Apply \eqref{eq:intro-coercivity} with $(\phi_1,\psi_1)$ and
	$\pi=(\Id,T_0)_\#\rho$.
	Since $R_0\geq0$ and
	$\phi_i(x)=c(x,T_i(x))-\psi_i(T_i(x))$, we obtain
	\begin{equation}\label{eq:symmetricgap}
		\begin{aligned}
			\left(\int\dist^2(T_0,T_1)\,\d\rho\right)^2
			&\leq C\int
			\bigl[R_1(x,T_0(x))+R_0(x,T_1(x))\bigr]\,\d\rho(x)\\
			&=C\int_Y(\psi_1-\psi_0)\,\d(\mu_1-\mu_0)\leq 2CDW_1(\mu_0,\mu_1),
		\end{aligned}
	\end{equation}
	where $D=\diam(\overline S\cup Y)$ and the last inequality follows from
	Lemma~\ref{lem:transport-foundations} and the Kantorovich--Rubinstein formula
	\eqref{eq:KR-duality}. 
	
	Finally, fix $\psi\in C(Y)$, set $\phi=\psi^c$, and let $T$
	be a Borel map satisfying $\phi(x)+\psi(T(x))=c(x,T(x))$.
	Put $\mu=T_\#\rho$. By dual feasibility and equality on the graph of $T$, 
	\[
	\operatorname{OT}_c(\rho,\mu)
	=\int_S c(x,T(x))\,\d\rho(x)
	=\int_S\phi\,\d\rho+\int_Y\psi\,\d\mu.
	\]
	Consequently, for every $\pi\in\Pi(\rho,\mu)$,
	\[
	\calD(\pi)
	=\int c\,\d\pi-\int_S\phi\,\d\rho-\int_Y\psi\,\d\mu
	=\int c\,\d\pi-\operatorname{OT}_c(\rho,\mu).
	\]
	Combining this identity with \eqref{eq:intro-coercivity}
	proves \eqref{eq:excess}.
\end{proof} 

\subsection{Estimates on annuli}\label{sec:annular}
Uniform semiconcavity estimates may fail near the diagonal, so we
localize the argument to annuli. At each transport scale, we again
use M\'erigot's estimates for the defect and the oscillation of
subdifferentials \cite[Theorem~3.4]{Merigot2026}.

Let $(X,\dist)$ be a proper geodesic metric space: closed bounded
sets are compact, and every two points can be joined by a minimizing geodesic. Let $\mm$ be a locally finite Borel reference measure, and fix a Borel source set $S$ with compact closure and a compact target set $Y$. Set $D=\diam(\overline S\cup Y)>0$ and assume $\rho\leq a_2\mm$. Fix $0<\theta<1/4$ and use $c=\dist^2/2$.

\begin{definition}[Assumptions on annuli]\label{def:annular}
Assume that the following conditions hold on the fixed source and
target regions, with constants independent of $r\in(0,D]$.

\smallskip\noindent\textup{(A) Coordinate estimates on annuli.}
There exists a finite Borel partition $S=\bigsqcup_j S_{j,r}$ and charts $F_{j,r}$, defined on neighborhoods of these sets, taking $S_{j,r}$ into a fixed compact $K_{\mathrm E}\Subset U_{\mathrm E}\subseteq\RR^d$ with $U_{\mathrm E}$ bounded and convex. There are compact target subsets $Y_{j,r}\subseteq Y$ such that
\[
 (x,y)\in S_{j,r}\times Y,\quad \theta r\leq \dist(x,y)\leq2r
 \quad\Longrightarrow\quad y\in Y_{j,r}.
\]
The normalized costs
\[
 f_{j,r,y}(z)=r^{-2}c(F_{j,r}^{-1}z,y),\qquad y\in Y_{j,r},
\]
have uniformly bounded Lipschitz and semiconcavity constants on $U_{\mathrm E}$.
For every $z\in U_{\mathrm E}$ and $y,w\in Y_{j,r}$,
there exist $p_y\in D^+f_{j,r,y}(z)$ and
$p_w\in D^+f_{j,r,w}(z)$ such that
\begin{equation}\label{eq:scaledrecovery}
 r^{-2}\dist^2(y,w)\leq B|p_y-p_w|.
\end{equation}
	Moreover, there are $a_{j,r}\geq0$ such that
\begin{equation}\label{eq:densitybudget}
 (F_{j,r})_\#(\mm|_{S_{j,r}})\leq a_{j,r}\mathcal L^d,
 \qquad \sum_j a_{j,r}\leq B.
\end{equation}
One may also use finitely many fixed pairs of coordinate domains
$(K_{\mathrm E},U_{\mathrm E})$, with uniform constants over those pairs.
In Lemma~\ref{lem:cone}, $B$ denotes a ball for the ambient metric.
If $Y_{j,r}$ is empty, no pair with source in $S_{j,r}$ satisfies
$\theta r\leq\dist(x,y)\leq2r$, so that part of the partition
contributes nothing to the estimates.

\smallskip\noindent\textup{(F) Flow decomposition.}
\textup{(F1)} There are finitely many bounded Borel sets
$V_j\subseteq\RR^{m_j}$, bounded intervals $I_j$, and Borel injections
$\Theta_j:V_j\times I_j\to X$, with Borel images $Z_j$ and Borel
inverses on $Z_j$, such that
\begin{equation}\label{eq:directional-measure}
 (\Theta_j^{-1})_\#(\mm|_{Z_j})
 \leq C_j(\mathcal L^{m_j}\otimes\mathcal L^1)|_{V_j\times I_j}.
\end{equation}
Write
$\Phi_j^t\Theta_j(\xi,s)=\Theta_j(\xi,s+t)$ whenever $s,s+t\in I_j$. 

\textup{(F2)} For each pair $(x,y)\in S\times Y$ with
$\ell=\dist(x,y)>0$, there exists a Borel choice of an index $j$ and a
constant-speed minimizing geodesic
$\Gamma_{x,y}\in C([0,1];X)$ from $x$ to $y$, with $x\in Z_j$.
Here the path space carries the uniform topology. Set
$\gamma_{x,y}(t)=\Gamma_{x,y}(t/\ell)$, so that $\gamma_{x,y}$ has unit
speed. For some fixed $0<\sigma\leq1$, the chosen index satisfies 
\begin{equation}\label{eq:flowtrack}
 \dist(\Phi_j^t x,\gamma_{x,y}(t))\leq\tfrac\alpha2 t,
 \quad 0<t<\sigma\ell,\qquad
 \alpha=\frac{1-4\theta}{4(1+\theta)}.
\end{equation}
All shifts in \eqref{eq:flowtrack} are defined, and the
sets, intervals, maps, and constants in (F1), together with
$\sigma$, are independent of $r$. 
\end{definition}

Condition \textup{(A)} gives the coordinate estimates when both
transport lengths lie in $[\theta r,2r]$. When one length lies in
$[r,2r]$ and the other is at most $\theta r$, Lemma~\ref{lem:cone}
excludes certain balls along a minimizing geodesic. We combine
this exclusion with the flow coordinates in \textup{(F)}.
Summing these estimates over all scales gives the following result.

\begin{theorem}[Annular criterion for coercivity]\label{thm:annular}
	Assume the conditions of Definition~\ref{def:annular},
	and let $\rho\in\PP(S)$ satisfy $\rho\leq a_2\mm$.
	There is a constant $C>0$ depending only on $a_2$, $D$,
	$\theta$, and the constants and coordinate domains in that definition.
	With this constant, \eqref{eq:intro-coercivity} holds for every
	$\psi\in C(Y)$, every Borel map $T:S\to Y$ satisfying
	$\phi(x)+\psi(T(x))=c(x,T(x))$ with $\phi=\psi^c$, and every
	probability plan $\pi$ on $S\times Y$ with first
	marginal $\rho$.
	Consequently, the conclusions of
	Corollary~\ref{cor:coercivity-consequences} hold.
\end{theorem} 

\Needspace{13\baselineskip}
\begin{lemma}[Transport lengths in one annulus]\label{lem:comparable}
Set
\[
 \begin{aligned}
 L(x,y)&=\max\{\dist(x,T(x)),\dist(x,y)\},\\
 l(x,y)&=\min\{\dist(x,T(x)),\dist(x,y)\}.
 \end{aligned}
\]
For
\[
 A_r=\{r\leq L<2r,\ l\geq\theta r\},
\]
and all $0<\tau<\tau_0$, hypothesis (A) yields 
\begin{equation}\label{eq:comparable}
 \int_{A_r}\dist^2(T(x),y)\dd\pi
 \leq C\left(\tau^{-1}\int_{A_r}R\dd\pi+\tau r^2\right).
\end{equation}
\end{lemma}
\begin{proof}
	
Let $\Lambda$ be a uniform semiconcavity constant in (A).
For $z\in U_{\mathrm E}$, set
\[
u_{j,r}(z)=\frac{\Lambda}{2}|z|^2
-\min_{w\in Y_{j,r}}
\bigl\{f_{j,r,w}(z)-r^{-2}\psi(w)\bigr\}.
\]
By (A), the functions $u_{j,r}$ are convex and uniformly
Lipschitz on $U_{\mathrm E}$, independently of $j$, $r$,
and $\psi$. Since $K_{\mathrm E}\Subset U_{\mathrm E}$,
the radius $\tau_0$ can also be chosen uniformly. For $(x,y)\in A_r$ with $x\in S_{j,r}$, both $T(x)$ and
$y$ belong to $Y_{j,r}$. Hence the minimum above equals
$r^{-2}\phi(x)$ at $z=F_{j,r}x$, and the normalized defect
is $r^{-2}R(x,y)$.
Applying the argument of \eqref{eq:pointwise} with
coordinate radius $\tau$ and multiplying by $r^2$ yields 
\[
 \dist^2(T(x),y)\leq C\left(\tau^{-1}R(x,y)+r^2\diam_{\mathrm E}\partial u_{j,r}(B_{\mathrm E}(F_{j,r}x,\tau))\right).
\]
The restriction of $\pi$ to
$A_r\cap(S_{j,r}\times Y)$ has first marginal bounded by
$\rho|_{S_{j,r}}$. Hence by 
Lemma~\ref{lem:subdifferential-oscillation} and \eqref{eq:densitybudget}, we obtain 
\[
r^2\int_{A_r\cap(S_{j,r}\times Y)}
\diam_{\mathrm E}\partial u_{j,r}
\bigl(B_{\mathrm E}(F_{j,r}x,\tau)\bigr)\,\d\pi(x,y)
\leq Ca_2a_{j,r}\tau r^2.
\]
Integrating the pointwise estimate and summing over $j$,
using $\sum_j a_{j,r}\leq B$, proves \eqref{eq:comparable}. 
\end{proof}

\begin{lemma}[Exclusion of balls along minimizing geodesics]\label{lem:cone}
Suppose $(x,y_x),(z,y_z)\in S\times Y$ satisfy
$R(x,y_x)+R(z,y_z)\leq\eps$ for some $\eps\geq0$, and
$\dist(x,y_x)=\ell\in[r,2r]$, $\dist(z,y_z)\leq\theta r$. Every unit-speed minimizer $\gamma$ from $x$ to $y_x$ satisfies
\begin{equation}\label{eq:cone}
 z\notin B(\gamma(t),\alpha t),\qquad
 \frac{8\eps}{3r}<t<\frac{3r}{16}.
\end{equation}
\end{lemma}
\begin{proof}
By dual feasibility on the cross pairs, we have  
\[
 \dist^2(x,y_x)+\dist^2(z,y_z)\leq \dist^2(x,y_z)+\dist^2(z,y_x)+2\eps.
\]
If $z$ belonged to the indicated ball, with $a=\dist(z,y_z)\leq\theta r$, the triangle inequalities would give
\[
 \dist(x,y_z)<a+(1+\alpha)t,\qquad
 \dist(z,y_x)<\ell-(1-\alpha)t.
\]
Substitution implies
\[
 0\leq2t\{(1+\alpha)a-(1-\alpha)\ell\}+2(1+\alpha^2)t^2+2\eps
 \leq-\tfrac32rt+4t^2+2\eps<0.
\]
Here $(1+\alpha)\theta-(1-\alpha)=-3/4$ and the bounds on $t$ were used.
\end{proof}

\begin{lemma}[One short and one long transport length]\label{lem:mixed}
Define
\begin{align*}
 E^+_{r,\eps}&=\{x:r\leq \dist(x,T(x))\leq2r,\ \exists y\in Y:
 \dist(x,y)\leq\theta r,\ R(x,y)\leq\eps\},\\
 E^-_{r,\eps}&=\{x:\dist(x,T(x))\leq\theta r,\ \exists y\in Y:
 r\leq \dist(x,y)\leq2r,\ R(x,y)\leq\eps\}.
\end{align*}
Under hypothesis (F), there are $\kappa,C>0$ such that
\begin{equation}\label{eq:mixed}
 \rho(E^+_{r,\eps})+\rho(E^-_{r,\eps})\leq C\eps/r^2,
 \qquad0<\eps\leq\kappa r^2.
\end{equation}
\end{lemma}
\begin{proof}
For each $x$, the targets satisfying the inequalities in the
definition of $E^\pm_{r,\eps}$ form a compact subset of $Y$.
The sets of pairs $(x,y)$ satisfying these conditions are Borel,
because $c$ and $\phi$ are continuous and $T$ is Borel.
The Arsenin--Kunugui theorem
\cite[Theorem~35.46(ii)]{Kechris1995} makes the projections Borel
and gives Borel choices of these targets. Use (F) to partition each exceptional set into finitely many Borel
sets. On each set, the index $j$ associated with the longer transport is fixed.

For $x,z\in E^+_{r,\eps}$ in the same part, use $(x,T(x))$ and choose
$(z,y_z)$ with $\dist(z,y_z)\leq\theta r$ and $R(z,y_z)\leq\eps$.
For $x,z\in E^-_{r,\eps}$ in the same part, choose $(x,y_x)$ with
$r\leq\dist(x,y_x)\leq2r$ and $R(x,y_x)\leq\eps$, and use
$(z,T(z))$. Lemma~\ref{lem:cone} and \eqref{eq:flowtrack} imply, after decreasing $\sigma\leq3/16$,
\[
 x\text{ in such a set}\quad\Longrightarrow\quad
 \Phi_j^t x\text{ not in that set},
 \qquad 8\eps/(3r)<t<\sigma r.
\]
Choose $\kappa$ so that the lower endpoint is smaller than half the upper endpoint.
Put $a=8\eps/(3r)$ and $b=\sigma r$. On a slice with fixed
$\xi\in V_j$, no two parameters in the exceptional set differ by
a number in $(a,b)$. Up to finitely many endpoints, partition $I_j$
into at most $1+|I_j|/b$ half-open intervals of length at most $b$.
Within each interval,
the slice has diameter at most $a$, and hence measure at most $a$.
Thus the measure of the whole slice is bounded by
\[
 \frac{8\eps}{3r}
 \left(1+\frac{|I_j|}{\sigma r}\right)
 \leq C\frac{\eps}{r^2},
\]
since $r\leq D$. For each sign, denote the partition pieces by $E_j^\pm$.
By the preceding slice estimate, \eqref{eq:directional-measure},
and $\rho\leq a_2\mm$, we obtain 
\[
\begin{aligned}
	\rho(E^\pm_{r,\eps})
	\leq a_2\sum_j C_j
	\int_{V_j}\int_{I_j}
	\one_{E_j^\pm}(\Theta_j(\xi,s))\,\d s\,\d\xi\leq C\frac{\eps}{r^2}, 
\end{aligned}
\]
which proves the claim. 
\end{proof}

\begin{proof}[Proof of Theorem~\ref{thm:annular}]
Take $r_k=2^{-k}D$, $k\geq0$. The bands $r_k\leq L<2r_k$
partition all positive lengths under consideration. The error is
zero when $L=0$. Set $\tau_*=\min\{\tau_0,1\}$ and fix $0<\tau<\tau_*$.  Apply
Lemma~\ref{lem:comparable} when $l\geq\theta r_k$.
For the other pairs, split at $R(x,y)=\kappa\tau r_k^2$. 

If $R(x,y)\geq\kappa\tau r_k^2$, then $\dist^2(T(x),y)\leq16r_k^2\leq C R(x,y)/\tau$. Otherwise,
$x\in E^+_{r_k,\kappa\tau r_k^2}
\cup E^-_{r_k,\kappa\tau r_k^2}$.
Since $\pi$ has first marginal $\rho$,
Lemma~\ref{lem:mixed} shows that these pairs have
$\pi$-mass at most $C\tau$. Their contribution to
$\calE(\pi)$ is therefore at most $C\tau r_k^2$.
Combining these bounds with Lemma~\ref{lem:comparable}
and summing over the disjoint bands, by monotone convergence, we obtain 
\[
 \calE(\pi)\leq C\left(\frac{\calD(\pi)}\tau+
 \tau\sum_{k=0}^{\infty}r_k^2\right)
 \leq C\left(\frac{\calD(\pi)}\tau+D^2\tau\right).
\]
For sufficiently small positive $\calD(\pi)$, choose
$\tau=\sqrt{\calD(\pi)}/D$ to obtain \eqref{eq:intro-coercivity}.
The case $\calD(\pi)=0$ follows by letting $\tau\downarrow0$,
while the bound $\calE(\pi)\leq\diam(Y)^2$ handles larger
defects after increasing $C$.
The remaining conclusions follow from
Corollary~\ref{cor:coercivity-consequences}. 
\end{proof}

\section{Geometric estimates}\label{sec:realizations}
We now verify the geometric hypotheses of
Theorems~\ref{thm:regular} and~\ref{thm:annular} in the
Riemannian and fat sub-Riemannian settings, respectively.
We also prove the horizontal estimates needed for potential
stability on arbitrary step-two Carnot groups. 

We first recall the control notation. In a local orthonormal
horizontal frame $F_1,\ldots,F_k$, a control
$a\in L^2([0,1];\RR^k)$ determines a curve by
$\dot\gamma_a=\sum_i a_iF_i(\gamma_a)$, with $\gamma_a(0)=x$.
The \emph{endpoint map} is $E_x(a)=\gamma_a(1)$, on the controls for
which this curve exists. The curve is \emph{regular} if $DE_x(a)$
is surjective. Otherwise it is \emph{singular}, or \emph{abnormal}.
The normal Hamiltonian is
$H(x,p)=\tfrac12|p|_{\mathcal H_x}^2
=\tfrac12\sum_i p(F_i(x))^2$, where $|p|_{\mathcal H_x}$
is the dual norm of $p|_{\mathcal H_x}$.
A \emph{normal geodesic} is the projection of a trajectory of
this Hamiltonian; the corresponding covector curve is its
\emph{normal lift}. We denote its endpoint at time one by
$\Exp_x(p)$, whenever the trajectory with initial covector
$p\in T_x^*M$ is defined up to that time. The differential
$D_p\Exp_x(p)$ is taken with respect to the initial covector.
When it is invertible, the endpoint is called nonconjugate to $x$
along that geodesic.

On a compact neighborhood we may replace the local frame by finitely
many smooth fields $F_i$ satisfying $\sum_iF_i\otimes F_i=g_{\mathcal H}^{-1}$,
where $g_{\mathcal H}^{-1}$ is the horizontal quadratic form on
covectors. These fields are obtained from local orthonormal frames
and a partition whose squares sum to one. They may be redundant.
For a horizontal curve its minimum-norm control is
$a_i=g_{\mathcal H}(\dot\gamma,F_i)$; this has the smallest norm
among controls representing its velocity, and its energy is
$\norm{a}_{L^2}^2/2$.

\subsection{Riemannian manifolds}\label{sec:riemannian}
We write $\exp_x$ for the Riemannian exponential map and
$\operatorname{inj}(x)$ for the injectivity radius at $x$.
The metric identifies vectors and covectors by
$v^\flat=g(v,\cdot)$; its inverse is denoted by $p\mapsto p^\sharp$.

\begin{lemma}[Uniform upper supports on compact sets]\label{lem:riemgeometry}
	Let $(M,g)$ be a complete connected smooth Riemannian manifold, and
	suppose that $\overline S$ and $Y$ are compact subsets of $M$.
	In finitely many source charts covering $\overline S$, the cost
	$c=\dcc^2/2$ satisfies Definition~\ref{def:regular} with uniform
	constants.
	
	Let $x$ belong to one of these source charts and $y\in Y$.
	Let $\gamma:[0,1]\to M$ be a constant-speed minimizing geodesic
	from $x$ to $y$, with initial velocity $v_y$. It defines a smooth
	upper support for $c(\cdot,y)$ at $x$, with differential
	\[
	p_y=-v_y^\flat,\qquad y=\exp_x(-p_y^\sharp).
	\]
	The estimates hold also on the diagonal and at cut points.
\end{lemma}

\begin{proof}
	We construct smooth upper supports as in \cite{Calabi1958}. Choose the source charts inside a fixed compact neighborhood of $\overline S$. All minimizing geodesics from this neighborhood
	to $Y$ have uniformly bounded lengths and, by completeness,
	lie in a common compact set. The injectivity radius has a positive lower bound on this compact
	set. Choose $t\in(0,1)$ and $r>0$ using these two bounds. For every constant-speed minimizer $\gamma:[0,1]\to M$
	from $x$ to $y$, the function 
	\[
	S_{x,y}(z)
	=\frac{\dcc^2(z,\gamma(t))}{2t}
	+\frac{1-t}{2}\dcc^2(x,y)
	\]
	is smooth on $B(x,r)$, with uniform coordinate $C^2$ bounds.
	The triangle inequality through $\gamma(t)$ and the weighted
	Cauchy--Schwarz inequality give
	$\dcc^2(z,y)\leq\dcc^2(z,\gamma(t))/t+(1-t)\dcc^2(x,y)$;
	compare \cite[Claim~2.4]{CorderoErausquinMcCannSchmuckenschlager2001}.
	Thus
	\[
	c(z,y)\leq S_{x,y}(z),\qquad c(x,y)=S_{x,y}(x).
	\]
	The first variation formula yields 
	\[
	p_y:=dS_{x,y}(x)=-v_y^\flat,
	\qquad y=\exp_x(-p_y^\sharp);
	\]
	see also \cite[Proposition~6]{McCann2001}. In finitely many relatively compact source charts with convex
	images, the uniform $C^2$ bounds on these supports give a common
	semiconcavity constant. The triangle inequality gives a common
	Lipschitz constant, proving \textup{(R1)}. 
	
	For \textup{(R2)}, fix a source chart $\chi$ and write
	$z=\chi(x)$. Let $P(z,y)$ be the set of coordinate
	representatives of $-v_y^\flat$ as the minimizing geodesic
	from $x$ to $y$ varies. These sets are nonempty and contained
	in $D^+f_y(z)$. For $p\in P(z,y)$ and $p'\in P(z,w)$, let $v,v'$ be the
	corresponding initial velocities. Since these are uniformly
	bounded, integrating the derivative of $\exp_x$ along the
	segment joining them yields 
	\[
	\dcc(y,w)
	=\dcc(\exp_x v,\exp_x v')
	\leq C|p-p'|.
	\]
	Since $\dcc(y,w)\leq\diam(Y)$, this proves \eqref{eq:recovery}. 
\end{proof} 

\begin{proof}[Proof of Theorem~\ref{thm:main-coercivity} on Riemannian manifolds]
	Lemma~\ref{lem:riemgeometry} gives source charts satisfying Definition~\ref{def:regular}.
	On compact subsets contained in the chart domains, the smoothness of $\mm$ and
	$\rho\leq a_2\mm$ give the required coordinate density bounds.
	Then Theorem~\ref{thm:regular} proves \eqref{eq:intro-coercivity}, and Corollary~\ref{cor:coercivity-consequences} yields uniqueness and \eqref{eq:main-map-stability}, with constants depending only
	on the fixed metric, $\mm$, $K$, and $a_2$.
\end{proof} 

\begin{remark}\label{rem:riemannian-l1-defect} 
	In the Riemannian case, the same argument also gives
	\[
	\left(\int_{S\times Y}\dcc(T(x),y)\,\d\pi(x,y)\right)^2
	\leq C\calD(\pi).
	\]
	Indeed, use the linear endpoint bound from
	Lemma~\ref{lem:riemgeometry} in the proof of
	Theorem~\ref{thm:regular}. This estimate fails in general for the Carnot--Carath\'eodory
	distance, even with fixed marginals; see
	Remark~\ref{rem:heisenberg-l1-defect}. 
\end{remark} 

\subsection{Local geometry of two-generating structures}\label{sec:two-step-local}
The following local estimates apply to all two-generating
structures. They will be used for map stability on fat manifolds
and potential stability on general two-generating manifolds.
Fix a compact region of a smooth two-generating structure.
Let $h=\operatorname{rank}\mathcal H$ and $k=\dim M-h$, and put
$Q=h+2k$. Choose local orthonormal fields $X_1,\ldots,X_h$ and
smooth complementary fields $Z_1,\ldots,Z_k$. On smaller compact
charts, choose privileged exponential coordinates of the first kind 
\cite[Proposition~6.2]{ChoiPonge2019II}, depending smoothly on the
center $p$: the point with coordinates $(z,w)$ is
\[
 \kappa_p(z,w)=\exp\left(\sum_i z_iX_i+\sum_a w_aZ_a\right)(p).
\]
Here $\exp(V)(p)$ denotes the time-one flow of the vector field $V$.
In these coordinates, the functions $z_i$ have horizontal order one,
and the functions $w_a$ have horizontal order two: all first horizontal
derivatives of $w_a$ vanish at $p$, while some second horizontal
derivative is nonzero. Their respective weights are one and two.
For $\delta_r(z,w)=(rz,r^2w)$, write
\[
 X_i^{p,r}=r(\delta_{1/r})_*(\kappa_p^{-1})_*X_i.
\]
The rescaled fields  admit the nilpotent approximation 
\begin{equation}\label{eq:rescaled-fields}
	X_i^{p,r}=X_i^{p,0}+O_{C^3}(r),\qquad
	X_i^{p,0}=\partial_{z_i}
	+\frac12\sum_{j,a}b_{ji}^a(p)z_j\partial_{w_a},
\end{equation}
uniformly on fixed bounded sets of rescaled coordinates. 
The coefficients $b_{ji}^a$ represent $[X_j,X_i]$ modulo
$\mathcal H$ in the chosen complementary frame. The remainder and
its derivatives through order three are bounded by $Cr$.
These uniform estimates on compact sets follow from the weighted Taylor expansion
in privileged coordinates; see
\cite{Bellaiche1996,ChoiPonge2019I,Mitchell1985,AgrachevBarilariBoscain2020}.
The limiting fields define the metric tangent Carnot group. 

Two-generation and compactness give a uniform positive lower
bound for the maximum absolute value of the determinants formed
from the horizontal fields and their first brackets. Consequently, the constants in the ball--box estimates for the
rescaled fields can be chosen uniformly. Let $\dist_{p,r}$ denote their control distance.
The ball--box bounds \cite{NagelSteinWainger1985} and smooth positive
density give
\begin{equation}\label{eq:fat-box-density}
 \dist_{p,r}^2(q,q')\leq C|q-q'|_{\mathrm E},\qquad
 (\kappa_p\delta_r)^*\mm=r^Qw_{p,r}(q)\,\d q,
 \qquad 0<c\leq w_{p,r}\leq C.
\end{equation}
The distance estimate holds uniformly for $q,q'$ in the fixed
bounded sets of rescaled coordinates under consideration. 
In particular, $\mm(B_{\mathrm{cc}}(p,r))\asymp r^Q$ uniformly. 

The uniform ball--box bounds place the images of these coordinate
sets in
$B_{\mathrm{cc}}(p,Cr)$, while $B_{\mathrm{cc}}(p,4Cr)$ lies in the
coordinate chart for sufficiently small $r$, uniformly in $p$.
By \cite[Proposition~2.15]{DonMagnani2023} and the scaling of
horizontal lengths, 
\begin{equation}\label{eq:fat-distance-rescaling}
	\dcc(\kappa_p\delta_rq,\kappa_p\delta_rq')
	=r\dist_{p,r}(q,q').
\end{equation}

Fix $L>0$. Every curve starting in these coordinate sets and
having $\dist_{p,r}$-length at most $L$ stays in one larger box.
This follows from the triangle inequality and the uniform
ball--box bounds. The box is independent of $p$ and small $r$.
These coordinate, distance, and volume estimates use only two-generation. 

\subsection{Fat manifolds}\label{sec:fat-manifolds}

On complete smooth fat structures, we need estimates that are
uniform over the rescaled systems of Section~\ref{sec:two-step-local}
and their nilpotent limits. We first prove such estimates for
regular minimizers, then verify the annular and flow conditions. 

\begin{lemma}[Uniform upper supports for regular minimizers]\label{lem:varying-controls}
	Consider smooth control systems forming a compact set in local
	$C^3$ on a common neighborhood, with energy $\norm{u}_{L^2}^2/2$
	and endpoint coordinates in a fixed finite atlas. Assume that nearby endpoints admit joining curves whose lengths
	tend to zero uniformly over these systems. Suppose that the
	minimizers under consideration, including those for limiting
	systems and endpoint pairs, exist, stay in a common compact
	subset of this neighborhood, and have lengths in
	$[\ell_-,\ell_+]$, where $\ell_->0$. 
	
	If every limiting minimizer is regular, then each minimizer from
	$x$ to $y$ admits a $C^2$ upper support $S_{x,y}$ for $c(\cdot,y)$
	on coordinate balls of a fixed positive radius, with uniform $C^2$ bounds. The normal lift \(p\) of the minimizer oriented from \(x\) to \(y\) is uniformly bounded and satisfies 
	\[
	\der S_{x,y}(x)=-p(0),\qquad y=\Exp_x(p(0)).
	\]
	The assertion includes cut points.
\end{lemma} 

\begin{proof}
	Consider convergent systems and endpoint pairs $(x_j,y_j)\to(x,y)$. For each pair $(x_j,y_j)$, choose a constant-speed minimizer from $y_j$ to $x_j$, and let $u_j$ be its minimum-norm control. Then $\norm{u_j}_{L^2}\leq \ell_+$.
	Since the paths lie in a common compact set, where the fields are
	uniformly bounded, the control equations make them
	equicontinuous. After passing to a subsequence,
	$u_j\rightharpoonup u$ in $L^2$ and the paths converge uniformly. Moreover, the integral control equation shows that $u$ joins $y$ to $x$. Let $m_j=\norm{u_j}_{L^2}^2/2$, and let $m$ be the minimum energy
	for the limiting system. Let $\bar u$ be a constant-speed minimizing control from $y$ to
	$x$ in the limiting system. Starting at $y_j$, the same control
	in the $j$th system ends at a point $z_j\to x$. Since $x_j\to x$,
	uniform local controllability joins $z_j$ to $x_j$ by a curve of
	length $o(1)$. Joining these two curves and reparametrizing at constant speed yields a competitor of energy $m+o(1)$, so 
	$\limsup_j m_j\leq m$. Therefore, 
	\[
	m\leq\tfrac12\norm{u}_{L^2}^2
	\leq\liminf_j m_j\leq\limsup_j m_j\leq m.
	\]
	The control norms therefore converge, and $u_j\to u$ strongly in $L^2$.

	\medskip
	For any minimizer under consideration, set $L=DE_y(u)$. Let $L^*$ be its adjoint for the $L^2$ and
	Euclidean inner products, and let $\lambda_{\min}$ denote the
	smallest eigenvalue of a symmetric matrix.
	If no uniform positive lower bound existed for $\lambda_{\min}(LL^*)$, we could choose
	$L_j=DE_{y_j}(u_j)$ with $\lambda_{\min}(L_jL_j^*)\to0$.
	After passing to a subsequence in a fixed endpoint chart,
	the strong compactness above and the first variational equation
	give $L_j\to L$ in operator norm, where $L$ is the endpoint
	differential of a limiting minimizer.
	Thus $\lambda_{\min}(LL^*)=0$, contradicting its regularity.
	Consequently, $\lambda_{\min}(LL^*)\geq\delta>0$ uniformly. Hence 
	\[
	Q=L^*(LL^*)^{-1},\qquad LQ=\Id,\qquad
	\norm Q\leq\delta^{-1/2}.
	\]

	\medskip
	The paths stay in a common compact subset of the open neighborhood.
	Gronwall's inequality therefore gives a common $L^2$ radius for
	control perturbations on which the endpoint maps are defined.
	The first two variational equations give uniform $C^2$ bounds
	for $a\mapsto E_y(u+Qa)$ near $a=0$. Its derivative there is
	$LQ=\Id$, so the quantitative inverse function theorem yields 
	an inverse $a=a(z)$ on a coordinate ball of fixed radius, with uniform
	$C^2$ bounds; compare
	\cite[proof of Theorem~5.9]{FigalliRifford10} (see also \cite[Theorem 3.14]{Rifford2014}). Hence
	\[
	S_{x,y}(z)=\tfrac12\norm{u+Qa(z)}_{L^2}^2
	\]
	is an upper support for $c(z,y)$: the perturbed control joins
	$y$ to $z$, and equality holds at $z=x$. 
	
	Let $q$ be the normal lift of the minimizer from $y$ to $x$.
	The first-order condition $u=L^*q(1)$, together with $LQ=\Id$ and $Da(x)=\Id$, yields 
	\[
	\der S_{x,y}(x)=Q^*u=q(1).
	\]
	The bounds on $Q$ and $u$, followed by the adjoint equation,
	bound $q$ uniformly. Reversing the lift by $p(t)=-q(1-t)$ yields
	\[
	\der S_{x,y}(x)=-p(0),\qquad y=\Exp_x(p(0)).
	\]
	Since the construction treats each minimizer separately, it remains
	valid at cut points.
\end{proof} 

To apply Lemma~\ref{lem:varying-controls} uniformly as $r\to0$,
we first check fatness of the nilpotent limits. Their step-two
brackets are the Levi forms $\mathcal L_p$ of $\mathcal H$ at the
base points \cite[Definition~6.9 and Proposition~6.12]{ChoiPonge2019I}.
Here $v\in\mathcal H_p$ and
$\zeta\in(T_pM/\mathcal H_p)^*$ are measured in the fixed local frames.
On each fixed compact base region, fatness and continuity yield the uniform bound
\[
\inf_{p,\,|v|=|\zeta|=1}
\norm{\mathcal L_p(v,\cdot)^*\zeta}>0.
\]
Thus the nilpotent limits are uniformly fat. By the $C^1$
convergence in \eqref{eq:rescaled-fields}, so are the rescaled
fields on the fixed boxes for sufficiently small $r$. 

\begin{proposition}[Uniform estimates on annuli in fat manifolds]\label{prop:fat-annuli}
	Let $(M,\mathcal H,g_{\mathcal H})$ be a complete connected smooth
	fat sub-Riemannian manifold, and suppose that $\overline S$ and $Y$
	are compact subsets of $M$. Then $c=\dcc^2/2$ satisfies condition
	\textup{(A)} of Definition~\ref{def:annular} with uniform constants. 
	
	Fix a compact set of pairs of rescaled coordinates separated from
	the diagonal. Each minimizing geodesic defines a $C^2$ upper
	support for $\dist_{p,r}^2/2$ on a coordinate ball of fixed radius.
	The $C^2$ norms and the normal lifts are bounded uniformly in
	$p$ and sufficiently small $r\geq0$.
	This includes cut points. 
\end{proposition} 

\begin{proof}
	Fix $r_0>0$ so that the rescaled estimates of
	Section~\ref{sec:two-step-local} hold for $0<r\leq r_0$.
	The nilpotent limits are fat, so their nonconstant minimizers are
regular \cite[Example~1.15]{Rifford2014}. The minimizers under
consideration for $0<r\leq r_0$ exist in $M$ by completeness; confinement
and \eqref{eq:fat-distance-rescaling} identify them with minimizers
of the rescaled systems. At $r=0$, they exist in the complete tangent
Carnot groups. On these fixed sets of pairs separated from the diagonal,
the minimizing lengths are bounded above and away from zero.
Section~\ref{sec:two-step-local} shows that bounded length curves
remain in a common compact box, while \eqref{eq:fat-box-density}
gives uniformly short joining curves between nearby endpoints.
Lemma~\ref{lem:varying-controls}
therefore implies uniform $C^2$ upper supports and bounded normal
lifts. On sufficiently small convex source boxes, the supports
give the asserted Lipschitz and semiconcavity bounds. 
	
	To prove \eqref{eq:scaledrecovery}, write
	$x=\kappa_p\delta_r z$, $y=\kappa_p\delta_r\tilde y$, and
	$w=\kappa_p\delta_r\tilde w$. Choose minimizing geodesics from
	$z$ to $\tilde y,\tilde w$, and let $p_y,p_w$ be the gradients of
	their upper supports for the normalized costs. Then $\tilde y=\Exp_z^{p,r}(-p_y)$ and
	$\tilde w=\Exp_z^{p,r}(-p_w)$, where $\Exp^{p,r}$ denotes  the normal exponential map of the rescaled fields. 
	Set $\lambda_s=-(1-s)p_y-sp_w$, $0\leq s\leq1$.
	These initial covectors are uniformly bounded.
	Hamiltonian conservation bounds their horizontal speeds, and
	Section~\ref{sec:two-step-local} confines the base curves to a
	common larger box. In this box the adjoint equation has the form
	$\dot\lambda=-\sum_i u_i(DX_i^{p,r})^*\lambda$, with uniformly
	bounded controls $u_i$ and coefficients. Gronwall's inequality
	therefore bounds the covectors up to time one and prevents finite-time
	escape. The variational equation then bounds
	$D_\lambda\Exp_z^{p,r}$ uniformly along the segment.
	Thus, by \eqref{eq:fat-distance-rescaling} and \eqref{eq:fat-box-density}, 
	\begin{equation*}
		\begin{aligned}
			r^{-2}\dcc^2(y,w)
			=\dist_{p,r}^2(\tilde y,\tilde w)
			&\leq C|\tilde y-\tilde w|_{\mathrm E}\\
			&\leq C\int_0^1
			\|D_\lambda\Exp_z^{p,r}(\lambda_s)\|
			|p_y-p_w|\,\d s
			\leq C|p_y-p_w|.
		\end{aligned}
	\end{equation*}
	
For small $r$, choose a maximal $\eta r$-separated set of source
centers and partition the source into Borel cells, each contained
in the $\eta r$-ball around its center. For a cell centered at $p_j$,
	set $F_{j,r}=(\kappa_{p_j}\delta_r)^{-1}$ and use nested convex boxes whose diameter for $\dist_{p_j,r}$ is less
	than $\theta/4$. Set
	$Y_{j,r}=\{y\in Y:\theta r/2\leq\dcc(p_j,y)\leq3r\}$.
	If $\eta$ is small enough, these sets
	contain every pair required in \textup{(A)}, while the targets
	remain uniformly separated from the outer source boxes. Uniform
	ball volumes give at most $Cr^{-Q}$ centers; by
	\eqref{eq:fat-box-density}, each partition element has coordinate
	density at most $Cr^Q$. Consequently
	\[
	\sum_j a_{j,r}\leq Cr^Q\cdot Cr^{-Q}\leq C.
	\] 
	For $r\in[r_0,D]$, use finitely many fixed source charts of
	metric diameter less than $\theta r_0/4$ and define $Y_{j,r}$ by the same distance bounds. The relevant endpoint pairs form a compact set away
	from the diagonal. Fix such a chart $F=F_{j,r}$, write
	$x=F^{-1}(z)$, and for each $v\in\{y,w\}$ choose a minimizing
	geodesic from $x$ to $v$. Let $q_v$ be the differential at $x$ of
	its upper support for $c(\cdot,v)$. The corresponding
	normalized coordinate gradient is $p_v=r^{-2}(dF^{-1}|_z)^*q_v$. 
	By Lemma~\ref{lem:varying-controls} and the uniform
	estimates for the exponential map and the ball--box estimates, 
	\[
	r^{-2}\dcc^2(y,w)
	\leq Cr^{-2}|q_y-q_w|
	\leq C|p_y-p_w|.
	\]
	This proves \eqref{eq:scaledrecovery}. The remaining normalized
	bounds follow from $r^{-2}\leq r_0^{-2}$, and the smooth density
	of $\mm$ then yields \eqref{eq:densitybudget}. 
\end{proof} 

To verify \textup{(F)}, we compare short minimizers with flows
generated by constant horizontal controls. For a smooth vector
field $X$, write $\Phi_X^t$ for its local flow. Given a local
orthonormal horizontal frame $X_1,\ldots,X_h$ and a constant
coefficient vector $v\in\RR^h$, set
$\Phi_v^t=\Phi_{\sum_i v_iX_i}^t$. 

\begin{lemma}[Approximation by horizontal flows]\label{lem:fat-flow}
Let $\gamma:[0,R]\to M$ be a unit-speed minimizer in the fixed
compact region, with $0<R\leq D$ and initial unit control $v$.
For fixed sufficiently small $\sigma>0$,
\begin{equation}\label{eq:fat-flow-approximation}
 \dcc(\gamma(t),\Phi_v^t x)\leq Ct\sqrt{t/R},
 \qquad 0<t<\sigma R,\quad x=\gamma(0).
\end{equation}
For unit vectors $v,w$ and uniformly small $t>0$,
\begin{equation}\label{eq:fat-flow-net}
 \dcc(\Phi_v^t x,\Phi_w^t x)\leq Ct\sqrt{|v-w|}.
\end{equation}
\end{lemma}

\begin{proof} 
	After decreasing $\sigma$, let $u$ be the control of
	$\gamma|_{[0,\sigma R]}$ in the initial local frame, and put
	$\tau=s/R$. The normal lifts are uniformly bounded after rescaling. For small
	$R$, this follows from Proposition~\ref{prop:fat-annuli}. For
	$R$ bounded below, apply Lemma~\ref{lem:varying-controls} to the
	full minimizer on $[0,1]$, using the fields $F_i$ defined above.
	The Hamiltonian equations then give $\left|\frac{\der}{\der\tau}u(R\tau)\right|\leq C$. 
	Since $v=u(0)$, it follows that 
	\[
	|u(s)-v|\leq Cs/R,\qquad 0\leq s\leq\sigma R.
	\] 
	In these privileged exponential coordinates centered at $x$, write
	$\gamma(s)=(z(s),\eta(s))$ and set $\Delta z(s)=z(s)-sv$.
	Since $\Phi_v^s x=(sv,0)$, comparison of the control equations yields 
	\[
	|\dot{\Delta z}(s)|
	\leq C\bigl(s/R+|\Delta z(s)|+|\eta(s)|\bigr),\qquad
	|\dot\eta(s)|
	\leq C\bigl(s^2/R+|\Delta z(s)|+|\eta(s)|\bigr).
	\]
	By Gronwall's inequality, we obtain 
	\[
	|\Delta z(t)|\leq C t^2/R,\qquad
	|\eta(t)|\leq C t^3/R. 
	\]
	In coordinates centered at $\Phi_v^t x$, the vertical coordinate
	difference is at most
	$C\bigl(|\eta(t)|+t|\Delta z(t)|\bigr)\leq Ct^3/R$. Hence, by the ball--box estimate, 
	\[
	\dcc(\gamma(t),\Phi_v^t x)
	\leq C\left(t^2/R+\sqrt{t^3/R}\right)
	\leq Ct\sqrt{t/R}.
	\]
	For two constant controls, the recentered horizontal and vertical
	errors are $O(t|v-w|)$ and $O(t^2|v-w|)$, respectively. Therefore, 
	\[
	\dcc(\Phi_v^t x,\Phi_w^t x)
	\leq Ct\bigl(|v-w|+\sqrt{|v-w|}\bigr)
	\leq Ct\sqrt{|v-w|}, 
	\] 
	which proves the claim. 
\end{proof} 

\begin{proof}[Proof of Theorem~\ref{thm:main-coercivity} on fat sub-Riemannian manifolds]
Proposition~\ref{prop:fat-annuli} proves \textup{(A)}.
Choose $\sigma$ and finite nets of unit controls so that the two
errors in Lemma~\ref{lem:fat-flow} are each at most $\alpha t/4$.
The resulting nonvanishing fields admit finitely many flow boxes
covering the compact source region, with bounded coordinate
densities. Decreasing $\sigma$ so that $\sigma D$ fits inside the
larger boxes implies \textup{(F1)} and \eqref{eq:flowtrack}.

On each positive distance annulus, the bounds on normal lifts and the
closed endpoint and energy conditions yield a Borel graph of
minimizing initial covectors with nonempty compact fibers.
The measurable selection theorem
\cite[Theorem~35.46]{Kechris1995}, followed by the first suitable
box and net direction, gives the Borel choices in \textup{(F2)}
across countably many annuli. Theorem~\ref{thm:annular} and
Corollary~\ref{cor:coercivity-consequences} finish the proof.
For a fat Carnot group, left invariance reduces the geometry to
a ball controlled by $D$, giving $C=C(G,a_2,D)$. 
\end{proof} 

On fat sub-Riemannian manifolds, the squared distance is locally
semiconcave away from the diagonal \cite[Proposition~4.1]{FigalliRifford10}. The following proposition gives the uniform bound along horizontal flows, including at the diagonal, which will be used for the stability of Kantorovich potentials. 

\begin{proposition}[Directional semiconcavity along horizontal flows]
	\label{prop:fat-directional}
	Let $(M,\mathcal H,g_{\mathcal H})$ be a complete connected smooth
	fat sub-Riemannian manifold, and let $\overline S,Y\subseteq M$
	be compact. Let $X$ range over finitely many smooth unit horizontal
	fields, each defined near a compact source set, with these sets
	covering $\overline S$. For $c=\dcc^2/2$, there exist $L_1,L_2<\infty$
	and $\sigma>0$ such that, for $|s|<\sigma$,
	\[
	|\partial_s c(\Phi_X^s x,y)|\leq L_1,\qquad
	\partial_s^2 c(\Phi_X^s x,y)\leq L_2,
	\]
	uniformly for $x$ in the corresponding source set and $y\in Y$.
	The first inequality holds almost everywhere in $s$ and the
	second holds in distributions, including across cut and diagonal
	points. The same assertion holds on complete connected smooth Riemannian
	manifolds for smooth unit vector fields.
\end{proposition}

\begin{proof}
	By finiteness and compactness, decrease $\sigma$ so that these
	flow segments remain in fixed compact subsets of the respective
	field domains.
	Set $f(s)=c(\Phi_X^s x,y)$. Since $X$ has unit length,
	$s\mapsto\dcc(\Phi_X^s x,y)$ is $1$-Lipschitz. The distances are
	uniformly bounded on the compact regions in use, proving the
	first inequality.
	
	Fix $s_0$ with $z=\Phi_X^{s_0}x\ne y$, and put $r=\dcc(z,y)$.
	If $r$ is small, set $F=(\kappa_z\delta_r)^{-1}$, $q(\tau)=F(\Phi_X^{r\tau}z)$, and  $\tilde y=F(y)$. Writing $X$ as a smooth combination of the local frame fields, \eqref{eq:rescaled-fields} gives uniform $C^1$ bounds for the rescaled field generating $q$.
	By Proposition~\ref{prop:fat-annuli}, the normalized cost
	$\dist_{z,r}^2(\cdot,\tilde y)/2$ has a $C^2$ upper support $S$
	at $q(0)=0$, with uniform bounds. Hence $q(h/r)$ stays in the coordinate neighborhood on which $S$ is an upper support whenever $|h|\leq c_0r$, for some
	$c_0>0$ independent of $z$ and $r$. By \eqref{eq:fat-distance-rescaling},
	$g(h):=r^2S(q(h/r))$ is an upper support for $f(s_0+h)$.
	Writing $\tau=h/r$, we have
	\[
	g''(h)
	=D^2S(q(\tau))[q'(\tau),q'(\tau)]
	+DS(q(\tau))\,q''(\tau)\leq C.
	\]
	The bound is uniform in $r$ by the $C^2$ bounds on $S$ and
	the $C^1$ bounds on the rescaled field. For $r$ bounded below,
	fatness and Lemma~\ref{lem:varying-controls} yield the same bound
	on compact off-diagonal endpoint sets. The characterization of semiconcavity by smooth upper supports therefore yields $f''\leq L_2$
	in distributions away from the diagonal, including at cut points. 
	
	On the other hand, if $\Phi_X^{s_*}x=y$, then $0\leq f(s_*+h)\leq h^2/2$. Hence both one-sided derivatives at $s_*$ are zero. 
	Near $s_*$, the flow meets $y$ only at $s_*$.
	Since $f'_-(s_*)=f'_+(s_*)=0$, the two restrictions of
	$f(s)-L_2s^2/2$ are concave and have matching one-sided
	derivatives. They therefore form a concave function across
	$s_*$, giving $f''\leq L_2$ in distributions. In the Riemannian case, compose the
	uniform upper supports of Lemma~\ref{lem:riemgeometry} with
	the flows.
\end{proof}

\begin{remark}[Contact and Sasakian manifolds]
A distribution on a $(2n+1)$-dimensional manifold is contact if
locally $\mathcal H=\ker\alpha$, where $\alpha$ is a smooth
differential one-form and $\alpha\wedge(\der\alpha)^n$ is nowhere zero.
Contact distributions are fat of corank one, so the preceding
argument includes all complete smooth contact sub-Riemannian structures. In Darboux
coordinates the distribution is the standard Heisenberg distribution:
only the horizontal metric varies under rescaling. 
In particular, the map and potential estimates apply to complete
smooth Sasakian structures under the respective source hypotheses.
Here a Riemannian manifold $(M,g)$ is Sasakian if its metric cone
$((0,\infty)\times M,(\der r)^2+r^2g)$ is K\"ahler; the induced
contact distribution carries the restriction of $g$.
\end{remark}

\subsection{Step-two Carnot groups}\label{sec:preliminaries}
Let $G$ be a connected and simply connected Lie group with stratified
Lie algebra $\mathfrak g=V_1\oplus V_2$, where $V_2\ne\{0\}$,
$[V_1,V_1]=V_2$ and $[V_1,V_2]=\{0\}$. We fix an inner product on
$V_1$, choose orthonormal coordinates on $V_1$ and linear coordinates
on $V_2$, and use the resulting exponential coordinates on $G$. Write
$d_1:=\dim V_1$, $d_2:=\dim V_2$, $d:=d_1+d_2$, and identify
$G$ with $\RR^d\simeq V_1\times V_2$. Points are denoted by
$x=[\zeta,t]$, with $\zeta\in V_1$ and $t\in V_2$. The group law is
\begin{equation}\label{eq:step-two-group-law}
 [\zeta,t]\cdot[\zeta',t']
 =\left[\zeta+\zeta',t+t'+\frac12[\zeta,\zeta']\right].
\end{equation}
Here $[\zeta,\zeta']\in V_2$ denotes the Lie bracket. The identity is
$e=[0,0]$ and $[\zeta,t]^{-1}=[-\zeta,-t]$. The dilations
$\delta_r[\zeta,t]=[r\zeta,r^2t]$ are group automorphisms.
We adopt the conventions of
\cite{FollandStein1982,AgrachevBarilariBoscain2020,Rifford2014}.

We normalize Haar measure by $\mm:=\mathcal L^d$ in the fixed
exponential coordinates. Left and right translations have Jacobian
one, so $\mm$ is bi-invariant, and
$\mm(\delta_r E)=r^Q\mm(E)$ for Borel sets $E$, where
$Q=d_1+2d_2$ is the homogeneous dimension.
For $v\in V_1$, let $V_v$ be the left-invariant horizontal vector field
defined by
\[
 V_vf(x)=\left.\frac{\der}{\der s}\right|_{s=0}
 f\bigl(x\cdot[sv,0]\bigr).
\]
If $e_1,\ldots,e_{d_1}$ is an orthonormal basis of $V_1$, the fields
$X_j:=V_{e_j}$ form the horizontal orthonormal frame. We write
$H_xG=(\der L_x)_e V_1$ for the horizontal space at $x$,
where $L_x(y)=x\cdot y$ is left translation.
The Carnot--Carath\'eodory distance has the definition given in
Section~\ref{sec:introduction}.
It is left invariant and homogeneous:
\[
 \dcc(z\cdot x,z\cdot y)=\dcc(x,y),\qquad
 \dcc(\delta_rx,\delta_ry)=r\dcc(x,y).
\]
The metric space $(G,\dcc)$ is proper and geodesic, and the exact formula for the measure of a ball is
\begin{equation}\label{eq:Haar-ball-volume}
 \mm(B_{\mathrm{cc}}(x,r))=\omega_G r^Q,
 \qquad \omega_G:=\mm(B_{\mathrm{cc}}(e,1)).
\end{equation}
In particular, $\mm$ is doubling: $\mm(B_{\mathrm{cc}}(x,2r))
\leq2^Q\mm(B_{\mathrm{cc}}(x,r))$. We write $\dE$ for Euclidean
distance in these coordinates. The distance from the identity satisfies
\begin{equation}\label{eq:gauge-comparison}
 \dcc(e,[u,\tau])\asymp_G |u|+|\tau|^{1/2}, 
\end{equation}
which implies $\dcc^2(x,y)\leq C(G,R)\dE(x,y)$ on
$B_{\mathrm{cc}}(e,R)$; see
\cite[Proposition~2.8]{DragoniLiuZhang2025}. 

For $\lambda\in V_2^*$, define the skew-adjoint map
$J_\lambda:V_1\to V_1$ by
\begin{equation}\label{eq:J-definition}
	\langle J_\lambda v,w\rangle=\lambda([v,w]).
\end{equation}
The group is fat if and only if $J_\lambda$ is invertible for every
$\lambda\ne0$, equivalently, $[v,V_1]=V_2$ for every $v\ne0$.
By compactness, fatness yields 
\[
\min_{|\lambda|=|v|=1}|J_\lambda v|>0.
\]
Fat groups admit no nonconstant abnormal horizontal curves
\cite[Section~4.1]{FigalliRifford10}. A general step-two Carnot
group, however, may admit nonconstant abnormal minimizers; by
\cite[Theorem~10]{Rizzi2016}, this occurs precisely when the group
is not fat. Thus the ideal Carnot groups, those without such
minimizers, are precisely the fat ones. H-type groups form a subclass of fat Carnot groups. A group
is of \emph{H-type} if some inner product on $V_2$, with dual norm
$|\cdot|_*$ on $V_2^*$, satisfies
$J_\lambda^2=-|\lambda|_*^2\Id$ for every $\lambda\in V_2^*$. The standard $(2n+1)$-dimensional Heisenberg group $\HH^n$
is an H-type group. Apart from the inner product in the H-type definition, norms on $V_2$ and
$V_2^*$ refer to the fixed coordinates; their choice is included in constants depending on $G$.

The proof of Kantorovich potential stability uses the following result, which holds on every step-two Carnot group. 

\begin{proposition}[Horizontal semiconcavity of the cost]\label{prop:horizontal-semiconcavity}
There exists $K_G\geq0$ such that, for every $x,y\in G$ and every unit
$v\in V_1$, the function $s\mapsto c(x\cdot[sv,0],y)$ satisfies
\begin{equation}\label{eq:horizontal-semiconcavity}
 \frac{\der^2}{\der s^2}c(x\cdot[sv,0],y)\leq K_G
\end{equation}
in distributions on $\RR$. On a segment contained in $S$, its first derivative exists almost
everywhere in $s$ and is bounded in absolute value by
$D=\diam_{\mathrm{cc}}(\overline S\cup Y)$, uniformly for $y\in Y$. 
\end{proposition}
\begin{proof}
By \cite[Theorem~1.1 and Definition~2.17]{DragoniLiuZhang2025},
$\dcc^2(e,\cdot)$ is horizontally semiconcave with a constant depending only on $G$.
Restricting to a horizontal line and using
$c(x,y)=\dcc^2(e,y^{-1}x)/2$ yields 
\eqref{eq:horizontal-semiconcavity}. The bound on the first derivative follows
from the triangle inequality and the unit speed of the line.
\end{proof}

\begin{remark}\label{rem:nonfat-obstruction}
	Horizontal semiconcavity does not imply the full semiconcavity in
	source coordinates required by \textup{(A)}. Consider the orthogonal
	product $G=\HH^1\times\RR$. Let $Z$ be a nonzero central vector in
	$\HH^1$ and fix $a\ne0$. For $x=(h,r)$, set
	$y=(h,r+a)$ and $x_s=(h\exp(sZ),r)$. The product distance and
	Heisenberg dilations yield 
	\[
	c(x_s,y)=\frac12\bigl(a^2+b|s|\bigr),\qquad
	b:=\dcc^2\bigl((e,0),(\exp Z,0)\bigr)>0.
	\]
	Thus $x\ne y$, but $c(\cdot,y)$ has an upward cusp at $x$ and is
	not semiconcave there. The variation $x_s$ is central, so this does
	not contradict horizontal semiconcavity.
	
	The same cusp occurs at pairs $(x,T(x))$ for an optimal transport map $T$. For any compactly supported
	absolutely continuous source $\rho$, let $T(h,r)=(h,r+a)$ and
	$\mu=T_\#\rho$. By Jensen's inequality, every $\pi\in\Pi(\rho,\mu)$ satisfies 
	\[
	\int c\,\d\pi
	\geq\frac12\int |r'-r|^2\,\d\pi
	\geq\frac{a^2}{2}
	=\int c(x,T(x))\,\d\rho(x),
	\]
	which implies $T$ is optimal. Thus the semiconcavity assumption in our map argument can fail
	even at pairs transported by an optimal map, although the stability estimate itself
	may still hold. For the general relation between abnormal minimizers and
	semiconcavity, and for failure of horizontal semiconcavity in the
	step-three Engel group, see \cite[Theorems~1.1--1.2 and Section~4]{MontanariMorbidelli2016}. 
\end{remark} 

\section{Coercivity for Kantorovich potentials}\label{sec:potential-stability} 
We use the quantitative concavity method of
M\'erigot--Delalande--Chazal and Delalande--M\'erigot
\cite{MerigotDelalandeChazal2020,DelalandeMerigot2023},
using the entropic regularization on manifolds developed by
Kitagawa--Letrouit--M\'erigot \cite{KitagawaLetrouitMerigot2025}.
We combine local variance estimates on John domains as in
Letrouit--M\'erigot and Kitagawa--Letrouit--M\'erigot
\cite{LetrouitMerigot2026,KitagawaLetrouitMerigot2025}. 
We prove the potential coercivity estimate
\eqref{eq:intro-potential-coercivity} in its three forms.
Horizontal paths and directional cost bounds give the linear
modulus and Theorem~\ref{thm:potential-stability}\textup{(iii)};
cost approximation gives every exponent below $1/2$ in \textup{(ii)}.
Quantitative affine approximation on Carnot groups and equiregular
charts proves the logarithmic estimate in \textup{(i)}.
The estimate in \textup{(iv)}, where one target has finite support,
follows separately from Lipschitz bounds and a Poincar\'e inequality.
Throughout this section, $M$, $S$, $\rho$, and $Y$ satisfy the
hypotheses of Theorem~\ref{thm:potential-stability}. The constants in the
resulting stability estimates have the dependencies stated there. 

\subsection{Concavity and stability}\label{sec:potential-conversion}
Let $U$ be a relatively compact neighborhood of $\overline S$ and
let $b$ be a continuous cost on a neighborhood of $\overline U\times Y$.
For $\theta\in C(Y)$, define
\[
 \theta^b(x)=\min_{y\in Y}\{b(x,y)-\theta(y)\},\qquad
 K_b(\theta)=\int_S\theta^b\,\d\rho.
\]
For $\theta_0,\theta_1\in C(Y)$, set
\begin{equation}\label{eq:potential-gap-variance}
 \begin{aligned}
 J_b(\theta_0,\theta_1)
 =K_b\!\left(\frac{\theta_0+\theta_1}{2}\right)
   -\frac{K_b(\theta_0)+K_b(\theta_1)}2,\qquad V_b(\theta_0,\theta_1)=\Var_\rho(\theta_1^b-\theta_0^b).
 \end{aligned}
\end{equation}
Concavity of $K_b$ yields $J_b\geq0$. We seek bounds for $V_b$ in
terms of $J_b$. The next proposition bounds $J_c$ by $W_1$ using Kantorovich
duality, as in \cite{DelalandeMerigot2023,KitagawaLetrouitMerigot2025}.

\begin{proposition}[From concavity to stability]\label{prop:potential-conversion}
For normalized $c$-conjugate Kantorovich potentials $(\phi_i,\psi_i)$
from $\rho$ to $\mu_i$, $i=0,1$, where $c=\dcc^2/2$,
\begin{equation}\label{eq:optimal-jensen-gap}
 J_c(\psi_0,\psi_1)
 \leq\frac14\int_Y(\psi_1-\psi_0)\,\d(\mu_1-\mu_0)
 \leq\frac D2W_1(\mu_0,\mu_1),
 \qquad \osc_Y(\psi_1-\psi_0)\leq2D^2.
\end{equation}
Suppose that $\Omega:[0,\infty)\to[0,\infty)$ is nondecreasing,
$\Omega(0)=0$, and
$V_c(\theta_0,\theta_1)\leq\Omega(J_c(\theta_0,\theta_1))$
for all $\theta_0,\theta_1\in C(Y)$ with
$\osc_Y(\theta_1-\theta_0)\leq2D^2$. Then, for every $\mu\in\PP(Y)$, the normalized source
Kantorovich potential $\phi_\mu$ is unique on $S$, and 
\begin{equation}\label{eq:potential-modulus-stability}
 \norm{\phi_{\mu_1}-\phi_{\mu_0}}_{L^2(\rho)}^2
 \leq\Omega\!\left(\frac D2W_1(\mu_0,\mu_1)\right).
\end{equation}
\end{proposition}
\begin{proof}
	Set $\bar\psi=(\psi_0+\psi_1)/2$ and
	$\eta=(\psi_1-\psi_0)/2$. Testing $\bar\psi=\psi_0+\eta
	=\psi_1-\eta$ against optimal couplings for $\mu_0$ and $\mu_1$ yields 
	\[
	K_c(\bar\psi)\leq K_c(\psi_0)-\int_Y\eta\,\d\mu_0,
	\qquad
	K_c(\bar\psi)\leq K_c(\psi_1)+\int_Y\eta\,\d\mu_1, 
	\]
	which implies the first inequality in
	\eqref{eq:optimal-jensen-gap}. By
	Lemma~\ref{lem:transport-foundations},
	$\Lip(\psi_1-\psi_0;Y)\leq2D$ and
	$\osc_Y(\psi_1-\psi_0)\leq2D^2$. Combining the Kantorovich--Rubinstein formula \eqref{eq:KR-duality} with a McShane extension \cite{McShane1934} implies the second inequality in \eqref{eq:optimal-jensen-gap}. Since the $\phi_i=\psi_i^c$ are normalized,
	$V_c(\psi_0,\psi_1)=\norm{\phi_1-\phi_0}_{L^2(\rho)}^2$.
	The assumption on $\Omega$ and its monotonicity prove
	\eqref{eq:potential-modulus-stability}. If $\mu_0=\mu_1$, the estimate applied to any two normalized
	potentials gives $\phi_0=\phi_1$ $\rho$-almost everywhere.
	Since $\rho\geq a_1\mm\restr S$ and the potentials are continuous,
	they agree on $S$. 
\end{proof} 

\subsection{Directional estimates}\label{sec:directional-costs}
For this subsection assume that $M$ is Riemannian or two-generating.
For every field $X$ used in Lemma~\ref{lem:averaged-paths}, every
$y\in Y$, and every flow segment in the source neighborhood,
assume that $s\mapsto b(\Phi_X^s x,y)$ satisfies
\begin{equation}\label{eq:general-directional-bounds}
 |\partial_s b|\leq L_1,\qquad \partial_s^2b\leq L_2,
 \qquad L_1,L_2\geq0, 
\end{equation}
where the first inequality holds almost everywhere and the second
holds in distributions. The constants are uniform in $x$, $y$,
and the selected field. The bounds are imposed on an open
neighborhood of the segments used below, so they also hold on
slightly extended segments.

\begin{theorem}[Variance from directional bounds]\label{thm:directional-jensen}
	Under the standing assumptions of this subsection, suppose that
	\eqref{eq:general-directional-bounds} holds. Then, for every $R_0\geq0$ and
	$\theta_0,\theta_1\in C(Y)$ with
	$\osc_Y(\theta_1-\theta_0)\leq R_0$,
	\begin{equation}\label{eq:directional-jensen}
		V_b(\theta_0,\theta_1)
		\leq C(1+L_1+L_2+R_0)J_b(\theta_0,\theta_1).
	\end{equation}
	The constant $C$ depends only on the fixed source geometry and
	density bounds, not on $\theta_0,\theta_1$ or higher derivatives of $b$. 
\end{theorem} 

\subsubsection{Horizontal paths and John domains}
We use averages over horizontal paths, as in
\cite{LanconelliMorbidelli2000Poincare,Jerison1986,FranchiLuWheeden1995},
and the arguments for summing local variances on John domains
of Letrouit--M\'erigot
\cite[Sections~3 and~4.1]{LetrouitMerigot2026} and
Kitagawa--Letrouit--M\'erigot
\cite[Section~3.3.1]{KitagawaLetrouitMerigot2025}.
The latter use the classical methods of
\cite{BuckleyKoskelaLu1996,HajlaszKoskela2000}.
All balls, paths, and measures in this section refer to the original
geometry, including when we later vary the cost.

\begin{lemma}[Averaged path estimate]\label{lem:averaged-paths}
Fix a compact region of a smooth two-generating or Riemannian
structure. There are constants $\kappa\geq2$, $r_*>0$, and $C$
depending only on this region and the structure. Let
$B=B_{\mathrm{cc}}(p,r)$ satisfy $r\leq r_*$ and $\kappa B\subseteq S$.
Each pair $x,y\in B$ can be joined by a path
$\gamma_{xy}\subseteq\kappa B$. The paths can be chosen Borel
measurably in $(x,y)$. Each path consists of a uniformly
bounded number of flow segments from a fixed finite collection of
locally defined smooth unit horizontal fields. Its total parameter length is at most $Cr$, and for every nonnegative Borel $F$,
\begin{equation}\label{eq:local-main-integral}
 \int_B\int_B\int_{\gamma_{xy}}F\,\d s\,\d\mm(y)\,\d\mm(x)
 \leq Cr\mm(B)\int_{\kappa B}F\,\d\mm.
\end{equation}
In a group one can take $r_*=D$, with constants depending on $G,D$.
\end{lemma}
\begin{proof}
Choose a local orthonormal frame $X_1,\ldots,X_h$ and
$k=\dim M-h$ brackets $[X_{i_\alpha},X_{j_\alpha}]$,
$\alpha=1,\ldots,k$, forming a basis modulo $\mathcal H$
throughout a smaller chart. Their determinant in an adapted
frame is uniformly bounded away from zero.
Set $R=Lr$, with $L$ a fixed sufficiently large constant,
and write $a\in\RR^h$ and $b\in\RR^k$. For each selected bracket, define
\[
\mathcal C_{\alpha,R}(b_\alpha)
=\Phi_{X_{j_\alpha}}^{-Rb_\alpha}
\circ\Phi_{X_{i_\alpha}}^{-R}
\circ\Phi_{X_{j_\alpha}}^{Rb_\alpha}
\circ\Phi_{X_{i_\alpha}}^{R},
\qquad \alpha=1,\ldots,k.
\]
Starting from $x$, follow the horizontal frame flows and then
the four successive flows defining each $\mathcal C_{\alpha,R}$. Their endpoint is
\[
\Psi_{x,R}(a,b)
=\mathcal C_{k,R}(b_k)\circ\cdots\circ
\mathcal C_{1,R}(b_1)\circ
\Phi_{X_h}^{Ra_h}\circ\cdots\circ
\Phi_{X_1}^{Ra_1}(x).
\]
Since $\mathcal C_{\alpha,R}(0)=\Id$, we have $\Psi_{x,R}(0,0)=x$. 

By \eqref{eq:rescaled-fields}, the endpoint maps, written in
privileged coordinates at $x$ and rescaled by $\delta_R$,
converge in $C^1$ to the endpoint maps for the limiting fields,
uniformly in $x$. The derivatives of the limiting maps at zero consist of the
horizontal basis and, up to signs, the selected bracket vectors,
and are uniformly invertible. The uniform inverse function theorem gives a fixed bounded
parameter neighborhood $\mathcal A$ on which $\Psi_{x,R}$ is
injective, with image containing $\kappa_x\delta_R(\mathcal U)$
for a fixed neighborhood $\mathcal U$ of zero in rescaled coordinates. By the uniform ball--box estimate, choose $L$
large enough, and then decrease $r_*$, so that
\[
B(p,r)\subseteq\Psi_{x,R}(\mathcal A)
\qquad\text{for every }x\in B(p,r).
\]
The inverse parameters $\xi(x,y)=\Psi_{x,R}^{-1}(y)$ are locally
smooth and hence Borel. Let $\gamma_{xy}$ be the corresponding
concatenated path. Write $\xi=(a,b)$. The endpoint Jacobian satisfies
\[
\Psi_{x,R}^*\mm=J_{x,R}(\xi)\,\d\xi,\qquad
cR^Q\leq J_{x,R}(\xi)\leq CR^Q,
\qquad Q=h+2k.
\]
Let $\tau_j(\xi)$ be the signed flow time of the $j$th segment 
and $\Gamma_{\xi,j,t}(x)$ be its point at parameter fraction
$t\in[0,1]$. For all $\xi\in\mathcal A$, increasing $\kappa$ if necessary yields 
\[
N=h+4k,\qquad 
\sum_{j=1}^N|\tau_j(\xi)|\leq CR,\qquad
\Gamma_{\xi,j,t}(B)\subseteq\kappa B. 
\]

\medskip
For fixed $\xi,j,t$, the map $x\mapsto\Gamma_{\xi,j,t}(x)$
is a composition of flows with fixed times $O(R)$.
Its Jacobian and inverse Jacobian are uniformly bounded, so
\[
\int_B F(\Gamma_{\xi,j,t}(x))\,\d\mm(x)
\leq C\int_{\kappa B}F\,\d\mm.
\]
Changing variables from $y$ to $\xi$ and applying Fubini's theorem, we obtain 
\[
\begin{aligned}
	\int_B\int_B\int_{\gamma_{xy}}
	F\,\d s\,\d\mm(y)\,\d\mm(x)
	&\leq CR^Q
	\int_{\mathcal A}\sum_{j=1}^N|\tau_j(\xi)|
	\int_0^1\int_B
	F(\Gamma_{\xi,j,t}(x))\,\d\mm(x)\,\d t\,\d\xi\\
	&\leq CR^{Q+1}\int_{\kappa B}F\,\d\mm
	\leq Cr\mm(B)\int_{\kappa B}F\,\d\mm.
\end{aligned}
\]
The last inequality uses $R=Lr$ and $\mm(B)\asymp r^Q$,
proving \eqref{eq:local-main-integral}.

\medskip
Fix a finite chart cover and order its charts. For each ball, use
the first chart in which the preceding construction applies. This
gives the required Borel choice of paths. In the Riemannian case, use only the
flows of the Riemannian frame fields, with $k=0$ and $Q=\dim M$.
On a Carnot group the normalized fields are independent of
scale; translation, dilation, and Haar invariance give the
assertion up to $r_*=D$, with constants depending on $G,D$. 
\end{proof}

\begin{lemma}[Summing local variance estimates]\label{lem:john-gluing}
On a smooth structure equiregular near $\overline S$, let $S$
satisfy the source assumptions of Theorem~\ref{thm:potential-stability}.
Fix $\kappa\geq2$ and a sufficiently small $r_*>0$.
There exists a countable collection $\calB$ of balls with
\begin{equation}\label{eq:john-cover}
 S=\bigcup_{B\in\calB}B,\quad r_B\leq r_*,\quad
 \kappa B\subseteq S,\quad
 \sum_{B\in\calB}\one_{\kappa B}\leq C\one_S,
\end{equation}
such that every continuous $u\in L^2(S,\mm)$ satisfies
\begin{equation}\label{eq:john-gluing}
 \int_S|u-u_{S,\mm}|^2\,\d\mm
 \leq C\sum_{B\in\calB}\int_B|u-u_{B,\mm}|^2\,\d\mm.
\end{equation}
The constant depends only on the structure and $\mm$ near $\overline S$,
$\kappa,r_*$, and the John parameter $\eta_0$. 
Here $u_{E,\mm}:=\mm(E)^{-1}\int_Eu\,\d\mm$ for any measurable
set $E$ of positive finite measure.
\end{lemma}
\begin{proof}
Suppose first $S\neq M$ and set $\delta(x)=\dcc(x,M\setminus S)$.
We adopt the construction of a Whitney cover by balls in
\cite[proof of Proposition~3.7]{KitagawaLetrouitMerigot2025} and choose
countably many balls
\[
 b_i=B_{\mathrm{cc}}(x_i,q_i),\qquad
 q_i=\min\{\delta(x_i)/(1000\kappa),r_*/20\},
\]
such that the $b_i$ are disjoint at each dyadic scale and the
$5b_i$ cover $S$, and put $\calB=\{B_i=10b_i\}$.
The radius function is $(1000\kappa)^{-1}$-Lipschitz. This bound and
local doubling yield bounded overlap of the $\kappa B_i$; moreover
$\kappa B_i\subseteq S$ and $r_{B_i}\leq r_*$. Thus
\eqref{eq:john-cover} holds. The construction using John curves in that proof, with step sizes
adjusted to these radii, gives a central ball and chains
$\mathcal C(i)$ from $B_i$ to it, with no repeated balls.
Consecutive balls $B_j,B_{j'}$ satisfy
$\mm(B_j\cap B_{j'})\geq c\max\{\mm(B_j),\mm(B_{j'})\}$,
and $B_i\subseteq C B_j$ for every $j\in\mathcal C(i)$.
The proof of \cite[Lemma~3.8]{KitagawaLetrouitMerigot2025}
uses doubling only for balls centered near $\overline S$ with
bounded radius. For small radii this follows from local doubling
\cite{NagelSteinWainger1985}. For the remaining radii it follows
from a uniform positive lower bound for the measure of a fixed
small ball and a finite upper bound for the measure of one compact
neighborhood containing all the enlarged balls. Applying the
variance estimate to
$\nu=\mm\restr S/\mm(S)$ and multiplying by $\mm(S)$ then implies \eqref{eq:john-gluing}.

If $S=M$ is compact, a finite connected cover by sufficiently small
balls has bounded overlap and chains with the same lower bound
on the measure of consecutive intersections. Applying the same
variance estimate to these finite chains proves the result. 
\end{proof} 

\subsubsection{Entropic regularization and variance estimates}
For interpolation arguments, given $\theta_0,\theta_1\in C(Y)$, we use
$v:=\theta_1-\theta_0$ and
$\theta_\lambda:=\theta_0+\lambda v$, $0\leq\lambda\leq1$.
The line and ball estimates below are stated for arbitrary
$\theta,v\in C(Y)$.
Fix a probability measure $\sigma$ on $Y$ with $\spt\sigma=Y$.
Such a measure exists for every nonempty compact $Y$; for example,
one can put positive summable masses on a countable dense subset.
For $\varphi\in C(Y)$ and $t>0$, define the regularized $b$-transform
\begin{equation}\label{eq:Phi}
 \Phi_t[\varphi](x)
 :=-t\log\int_Y
       e^{(\varphi(y)-b(x,y))/t}\,\d\sigma(y),\qquad x\in U,
\end{equation}
and the regularized Kantorovich functional
\begin{equation}\label{eq:Kt}
 K_t[\varphi]:=\int_S\Phi_t[\varphi]\,\d\rho.
\end{equation}
The associated Gibbs measures are
\begin{equation}\label{eq:endpoint-Gibbs}
 \d\mu_x^t[\varphi](y)
 :=\frac{e^{(\varphi(y)-b(x,y))/t}\,\d\sigma(y)}
          {\int_Y e^{(\varphi(z)-b(x,z))/t}\,\d\sigma(z)},
 \qquad
 \mu^t[\varphi]:=\int_S\mu_x^t[\varphi]\,\d\rho(x).
\end{equation}
Thus, for $v\in C(Y)$,
\begin{equation}\label{eq:averaged-Gibbs}
 \E_{\mu^t[\varphi]}v
 =\int_S\E_{\mu_x^t[\varphi]}v\,\d\rho(x).
\end{equation}
The regularization uses the Gibbs kernel $e^{-b/t}$ directly on the
compact target set $Y$.
The next identities are the formulas for the first and second variations
of Kitagawa--Letrouit--M\'erigot
\cite[Lemmas~2.3--2.4]{KitagawaLetrouitMerigot2025}, written for the $c$-transform defined by an infimum in
\eqref{eq:mutual-c-transforms}. We combine these identities with
horizontal semiconcavity below.

\begin{lemma}[Variations of the regularized functional]\label{lem:variation}
Let $\theta_0,\theta_1\in C(Y)$, with $v$ and
$\theta_\lambda$ as above.
For $t>0$, the following identities hold:
\begin{equation}\label{eq:regularized-functional-variations}
 \frac{\der}{\der\lambda}K_t[\theta_\lambda]
 =-\E_{\mu^t[\theta_\lambda]}v,
 \qquad
 \frac{\der^2}{\der\lambda^2}K_t[\theta_\lambda]
 =-\frac1t\int_S\Var_{\mu_x^t[\theta_\lambda]}(v)\,\d\rho(x),
\end{equation}
and
\begin{equation}\label{eq:centered-transform-variation}
 \frac{\der}{\der\lambda}
       (\Phi_t[\theta_\lambda]-K_t[\theta_\lambda])(x)
 =-\bigl(\E_{\mu_x^t[\theta_\lambda]}v
          -\E_{\mu^t[\theta_\lambda]}v\bigr).
\end{equation}
\end{lemma}
\begin{proof}
The functions $v$, $\theta_\lambda$, and $b$ are bounded on the
relevant compact sets. Differentiation under the integral in
\eqref{eq:Phi} gives
$\partial_\lambda\Phi_t[\theta_\lambda]
=-\E_{\mu_x^t[\theta_\lambda]}v$ and
$\partial_\lambda^2\Phi_t[\theta_\lambda]
=-t^{-1}\Var_{\mu_x^t[\theta_\lambda]}(v)$.
Integrating over $S$ proves the assertions.
\end{proof}

\begin{lemma}[Oscillation along horizontal flows]
\label{lem:line-oscillation}
Let $\gamma:[s_0,s_1]\to M$ be a segment of one of the horizontal
flows used in Lemma~\ref{lem:averaged-paths}, with its flow parameter. In a group,
$\gamma(s)=x\cdot[su,0]$ with $|u|_{\mathrm E}=1$. For
$\theta,v\in C(Y)$ and $t>0$, put
$\mu_s=\mu_{\gamma(s)}^t[\theta]$. Then
\begin{equation}\label{eq:line-oscillation}
 |\E_{\mu_{s_1}}v-\E_{\mu_{s_0}}v|^2
 \leq\frac{L_2(s_1-s_0)+2L_1}{t}
       \int_{s_0}^{s_1}\Var_{\mu_s}(v)\,\d s.
\end{equation}
\end{lemma}
\begin{proof}
Write $b_s(y)=b(\gamma(s),y)$,
$F(s)=\Phi_t[\theta](\gamma(s))$, and $g(s)=\E_{\mu_s}v$.
For a cost smooth in $s$, differentiation yields 
\begin{equation}\label{eq:directional-Gibbs-identities}
 F'=\E_{\mu_s}b_s',\qquad
 F''=\E_{\mu_s}b_s''-t^{-1}\Var_{\mu_s}(b_s'),\qquad
 g'=-t^{-1}\Cov_{\mu_s}(v,b_s').
\end{equation}
By integrating the second identity, we then obtain 
\begin{equation}\label{eq:integrated-score-variance}
 \int_{s_0}^{s_1}\Var_{\mu_s}(b_s')\,\d s
 \leq t\{L_2(s_1-s_0)+2L_1\}.
\end{equation}
The covariance inequality and Cauchy--Schwarz in $s$ then prove
\eqref{eq:line-oscillation}.

For a continuous cost satisfying
\eqref{eq:general-directional-bounds}, extend the flow segment
slightly and mollify $b_s(y)$ in $s$, using the same nonnegative
mollifier for every $y$. The resulting costs $b_\varepsilon$ are
smooth in $s$ and satisfy
\[
|\partial_s b_\varepsilon|\leq L_1,\qquad
\partial_s^2 b_\varepsilon\leq L_2.
\]
Moreover, $b_\varepsilon(s,y)\to b_s(y)$ uniformly on
$[s_0,s_1]\times Y$. For fixed $t>0$, the associated Gibbs
expectations and variances of $v$ therefore converge uniformly
in $s$. Passing to the limit in the smooth inequality proves
\eqref{eq:line-oscillation}. 
\end{proof}

\begin{lemma}[Local ball estimate]\label{lem:local-L2}
For $0<t\leq1$, $\theta,v\in C(Y)$, and a ball as in
Lemma~\ref{lem:averaged-paths}, put
$g(x)=\E_{\mu_x^t[\theta]}v$. Then
\begin{equation}\label{eq:local-L2}
 \int_B|g-g_{B,\mm}|^2\,\d\mm
 \leq\frac{C(1+L_1+L_2)r}{t}\int_{\kappa B}
          \Var_{\mu_x^t[\theta]}(v)\,\d\mm(x),
 \qquad g_{B,\mm}=\frac1{\mm(B)}\int_Bg\,\d\mm.
\end{equation}
\end{lemma}
\begin{proof}
Apply Lemma~\ref{lem:line-oscillation} to each nonconstant piece
of $\gamma_{xy}$. The number of pieces and their parameter lengths are uniformly
bounded. Then Cauchy--Schwarz for their sum yields 
\[
 |g(x)-g(y)|^2\leq\frac{C(1+L_1+L_2)}t
       \int_{\gamma_{xy}}\Var_{\mu_z^t[\theta]}(v)\,\d s.
\]
Integrating over $B\times B$ and using \eqref{eq:local-main-integral}, we complete the proof. 
\end{proof}

\begin{proposition}[Global oscillation estimate]
\label{prop:global-gibbs-oscillation}
For every $0<t\leq1$ and $\theta,v\in C(Y)$,
\begin{equation}\label{eq:global-gibbs-oscillation}
 \norm{\E_{\mu_x^t[\theta]}v-\E_{\mu^t[\theta]}v}_{L^2(\rho)}^2
 \leq\frac{C(1+L_1+L_2)}t\int_S\Var_{\mu_x^t[\theta]}(v)\,\d\rho(x).
\end{equation}
\end{proposition}
\begin{proof}
	Set $g(x)=\E_{\mu_x^t[\theta]}v$ and
	$A(x)=\Var_{\mu_x^t[\theta]}(v)$. These functions are continuous and bounded. Comparing with the constant $g_{S,\mm}$ and applying Lemmas~\ref{lem:john-gluing} and~\ref{lem:local-L2}, we obtain 
	\[
	\begin{aligned}
		\Var_\rho(g)
		\leq a_2\int_S|g-g_{S,\mm}|^2\,\d\mm
		&\leq Ca_2\sum_{B\in\calB}
		\int_B|g-g_{B,\mm}|^2\,\d\mm\\
		&\leq \frac{Ca_2(1+L_1+L_2)}{t}
		\sum_{B\in\calB}r_B\int_{\kappa B}A\,\d\mm.
	\end{aligned}
	\]
	By $r_B\leq r_*$, bounded overlap, and the lower density bound, 
	\[
	\sum_{B\in\calB}r_B\int_{\kappa B}A\,\d\mm
	\leq Cr_*\int_S A\,\d\mm
	\leq \frac{Cr_*}{a_1}\int_S A\,\d\rho.
	\]
	Since $\E_\rho g=\E_{\mu^t[\theta]}v$, these estimates prove
	\eqref{eq:global-gibbs-oscillation}.
\end{proof} 

\subsubsection{Variance of potential differences}
We now pass from the variance of the regularized derivatives to
the variance of a difference of potentials. 

\begin{lemma}[Uniform Laplace principle]\label{lem:laplace-principle}
Let $\theta_0,\theta_1\in C(Y)$, with $\theta_\lambda$ as above. Then
\begin{equation}\label{eq:laplace-limit}
 \Phi_t[\theta_\lambda](x)\longrightarrow\theta_\lambda^b(x)
 \quad\text{uniformly on }\overline S\times[0,1]
 \quad\text{as }t\downarrow0.
\end{equation}
Consequently $k_t(\lambda):=K_t[\theta_\lambda]$ converges uniformly
to the concave function
$k(\lambda):=\int_S\theta_\lambda^b\,\d\rho$.
\end{lemma}
\begin{proof}
Let $m(x,\lambda):=\min_{y\in Y}\{b(x,y)-\theta_\lambda(y)\}$. Since $\sigma$ is a probability measure,
$\Phi_t[\theta_\lambda](x)\geq m(x,\lambda)$.
Given $\eta>0$, uniform continuity provides $r>0$ such that,
whenever $\dcc(y,z)<r$,
\[
 |b(x,y)-\theta_\lambda(y)-b(x,z)+\theta_\lambda(z)|<\eta
\]
uniformly in $x,\lambda$. Choose a finite cover
$Y\subseteq\bigcup_{j=1}^N B_{\mathrm{cc}}(y_j,r/2)$,
with $y_j\in Y$. Since $\sigma$ has full support,
\[
b_r:=\inf_{y\in Y}\sigma(B_{\mathrm{cc}}(y,r)\cap Y)
\geq\min_{1\leq j\leq N}
\sigma(B_{\mathrm{cc}}(y_j,r/2)\cap Y)>0.
\]
Integrating over the $r$-ball around a minimizing point yields 
\[
0\leq\Phi_t[\theta_\lambda](x)-m(x,\lambda)
\leq\eta-t\log b_r.
\]
Letting $t\downarrow0$ and then $\eta\downarrow0$ proves uniform
convergence to $m$. Integration against $\rho$ then yields uniform convergence of $k_t$. 
\end{proof}

\begin{proof}[Proof of Theorem~\ref{thm:directional-jensen}]
Fix $0<t\leq1$, and set $F_\lambda=\Phi_t[\theta_\lambda]$ and
$k_t(\lambda)=K_t[\theta_\lambda]$. By Lemma~\ref{lem:variation} and
Proposition~\ref{prop:global-gibbs-oscillation}, 
\[
 \Var_\rho(\partial_\lambda F_\lambda)
 \leq A[-k_t''(\lambda)],\qquad A=C(1+L_1+L_2).
\]
The functions $F_\lambda$ are concave in $\lambda$, and their slopes lie between
$-\max_Y(\theta_1-\theta_0)$ and $-\min_Y(\theta_1-\theta_0)$.
Thus the common slope interval has length at most $R_0$. 

Set $\beta=F_{1/2}-(F_0+F_1)/2$ and
$J_t=\int\beta\,\d\rho$.
By concavity and the slope bound,
\[
|\partial_\lambda F_\lambda-(F_1-F_0)|
\leq8\beta,\qquad \lambda\in[1/4,3/4],\qquad 0\leq\beta\leq R_0/4. 
\]
Moreover, 
\[
J_t=\frac12\int_0^1
\min\{\lambda,1-\lambda\}[-k_t''(\lambda)]\,\d\lambda
\geq\frac18\int_{1/4}^{3/4}[-k_t''(\lambda)]\,\d\lambda.
\]
By the variance bound and Cauchy--Schwarz,
averaging over $[1/4,3/4]$ yields
\[
\begin{aligned}
	\Var_\rho(F_1-F_0)\leq4A\int_{1/4}^{3/4}[-k_t''(\lambda)]\,\d\lambda+128\int\beta^2\,\d\rho\leq32(A+R_0)J_t.
\end{aligned}
\]
Finally, letting $t\downarrow0$ and applying
Lemma~\ref{lem:laplace-principle} proves \eqref{eq:directional-jensen}. 
\end{proof}

\subsection{Cost approximation}\label{sec:cost-approximation}
The following proposition transfers a variance estimate from an
approximating cost to the original cost. The error is controlled
by the uniform distance between the two costs.

\begin{proposition}[Comparison of costs]\label{prop:cost-comparison}
Let $b,c$ be continuous costs with $\norm{b-c}_\infty\leq e$, where
$e\geq0$. Here costs are compared on $\overline U\times Y$ and
their transforms on $\overline U$. Then for all $\theta,\theta_0,\theta_1\in C(Y)$, 
\[
 \norm{\theta^b-\theta^c}_\infty\leq e,\qquad
 |J_b(\theta_0,\theta_1)-J_c(\theta_0,\theta_1)|\leq2e.
\]
If $V_b(\theta_0,\theta_1)\leq A J_b(\theta_0,\theta_1)$ for $A\geq0$, then
\begin{equation}\label{eq:perturbed-jensen}
 V_c(\theta_0,\theta_1)
 \leq2A\bigl(J_c(\theta_0,\theta_1)+2e\bigr)+8e^2.
\end{equation}
Consequently, if costs $b_\varepsilon$, $0<\varepsilon<\varepsilon_0$, satisfy
$\norm{b_\varepsilon-c}_\infty\leq e_\varepsilon$ and
$V_{b_\varepsilon}\leq A_\varepsilon J_{b_\varepsilon}$
on a fixed class of pairs $(\theta_0,\theta_1)\in C(Y)^2$, then
on that class
\begin{equation}\label{eq:approximation-infimum}
 V_c\leq\inf_{0<\varepsilon<\varepsilon_0}
 \left\{2A_\varepsilon(J_c+2e_\varepsilon)+8e_\varepsilon^2\right\}.
\end{equation}
\end{proposition}

\begin{proof}
	The bounds for transforms and gaps follow by comparing minima. The two potential differences differ by at most $2e$
	in the uniform norm. By the triangle inequality in $L^2(\rho)$ for centered functions,
	followed by Cauchy--Schwarz,
	\[
	\begin{aligned}
		V_c\leq2V_b+8e^2\leq2A(J_c+2e)+8e^2.
	\end{aligned}
	\]
	This proves \eqref{eq:perturbed-jensen}.
	Applying the estimate to each $b_\varepsilon$ and taking
	the infimum gives \eqref{eq:approximation-infimum}.
\end{proof} 

In particular, $e_\varepsilon\leq C\varepsilon$ and
$A_\varepsilon\leq C\varepsilon^{-\eta}$, with $0\leq\eta<1$, yield $V_c\leq C'J_c^{1-\eta}$ when $\osc_Y(\theta_1-\theta_0)$ has a fixed bound.
For small positive $J_c$, take $\varepsilon=J_c$; for $J_c=0$,
let $\varepsilon\downarrow0$. For larger $J_c$, use the oscillation
bound. Proposition~\ref{prop:potential-conversion} then gives the
potential exponent $(1-\eta)/2$. 

\Needspace{10\baselineskip}
\subsection{Costs given by squared distances}\label{sec:potential-geometries}
\subsubsection{Riemannian manifolds, fat manifolds, and step-two Carnot groups}
The directional estimates of Section~\ref{sec:realizations}
give a linear variance bound for the original cost.

\begin{proof}[Proof of Theorem~\ref{thm:potential-stability}\textup{(iii)}]
Take normalized $c$-conjugate Kantorovich potentials $(\phi_i,\psi_i)$
from $\rho$ to $\mu_i$, $i=0,1$, and put $b=c$.
Propositions~\ref{prop:fat-directional} and
\ref{prop:horizontal-semiconcavity} give
\eqref{eq:general-directional-bounds} in the stated geometries. 
By Lemma~\ref{lem:transport-foundations}, the potentials
$\psi_0,\psi_1$ are $D$-Lipschitz. Then $\osc_Y(\psi_1-\psi_0)\leq2D^2$. Combining Theorem~\ref{thm:directional-jensen} with \eqref{eq:optimal-jensen-gap}, we then obtain  
\begin{equation}\label{eq:potential-dual-gap}
 \norm{\phi_1-\phi_0}_{L^2(\rho)}^2
 \leq C\int_Y(\psi_1-\psi_0)\,\d(\mu_1-\mu_0)
 \leq C W_1(\mu_0,\mu_1).
\end{equation}
Proposition~\ref{prop:potential-conversion} also gives uniqueness
of the normalized source potential. The argument uses optimal
couplings and does not assume uniqueness of optimal maps. 
\end{proof}

\subsubsection{Penalized energy}\label{sec:two-generating-potentials}
Assume now that $M$ is complete and smooth, with a two-generating
horizontal distribution. The uniform directional semiconcavity
bound used above may fail; see Section~\ref{sec:horizontal-obstruction}.
Riemannian approximation is classical in sub-Riemannian geometry \cite[Section~1.2]{Ge1993}, 
\cite[Section~2.4]{CapognaDanielliPaulsTyson2007}
and was used in Heisenberg optimal transport by Ambrosio--Rigot
\cite[Section~6]{AmbrosioRigot2004}.
The squared distance is locally Lipschitz in local coordinates
\cite[Corollary~6.2]{AgrachevLee2009}
(see also \cite[Proposition~4.2]{FigalliRifford10}).
This bounds both the cost error and the penalized control components
by $O(\varepsilon)$, and hence bounds the normal covectors uniformly.
Horizontal control variations whose higher-order time moments
vanish then give bounds on the directional second derivatives. First-order
optimality removes the need to bound second derivatives of the
corrected controls. Proposition~\ref{prop:cost-comparison}
then gives every potential exponent below $1/2$.

Fix compact endpoint neighborhoods containing $\overline S\cup Y$.
Choose a larger neighborhood containing all horizontal minimizing
curves between these endpoints. On this neighborhood, choose an
auxiliary Riemannian metric $g_{\mathrm R}$ and finitely many smooth
fields $F_a,G_b$ such that
\[
 \sum_a F_a\otimes F_a=g_{\mathcal H}^{-1},\qquad
 \sum_b G_b\otimes G_b=g_{\mathrm R}^{-1}.
\]
Here a cometric is the quadratic form on covectors associated
with a metric, as in the construction of the fields $F_i$ above. For $0<\varepsilon<1$ use the dynamics and energy
\begin{equation}\label{eq:penalized-energy}
 \dot\gamma=\sum_a u_aF_a(\gamma)+\sum_b w_bG_b(\gamma),\qquad
 \calA_\varepsilon(u,w)=\tfrac12\norm{u}_{L^2}^2
                         +\tfrac1{2\varepsilon}\norm{w}_{L^2}^2.
\end{equation}
For each curve, we choose the controls $(u,w)$ that minimize the
integrand in \eqref{eq:penalized-energy} among all representations
of its velocity. Minimizing over such representations gives the
Riemannian metric with cometric
$g_{\mathcal H}^{-1}+\varepsilon g_{\mathrm R}^{-1}$.
Extend this metric smoothly to a complete Riemannian metric outside
a larger compact set. Denote its distance by $\dist_\varepsilon$
and its cost by $c_\varepsilon=\dist_\varepsilon^2/2$.

\begin{lemma}[Confinement of penalized minimizers]\label{lem:penalized-confinement}
For all sufficiently small $\varepsilon$, the minimizing curves
between the fixed endpoint neighborhoods stay in one compact
region where \eqref{eq:penalized-energy} holds. The completion
outside this region does not affect their costs, and $c_\varepsilon\leq c$ on the fixed endpoint region. 
\end{lemma}
\begin{proof}
Choose $E_*$ such that $c(x,y)+1<E_*$ for all endpoint pairs under consideration. By completeness, original horizontal curves with energy at most $E_*$ and initial point in the endpoint region stay in a fixed
compact neighborhood. Choose successively larger open
neighborhoods of this set. Extend $F_a$ with cutoffs so that they
remain horizontal and $\sum_aF_a\otimes F_a\leq g_{\mathcal H}^{-1}$.
A curve driven only by $u$ then has horizontal energy at most
$\norm{u}_{L^2}^2/2$. 

Consider a penalized curve of energy at most $E_*$.
Up to its possible first exit from the larger neighborhood,
the energy bound yields
\[
\norm{u}_{L^2}\leq\sqrt{2E_*},\qquad
\norm{w}_{L^2}\leq\sqrt{2E_*\varepsilon}.
\] 
Delete $w$ and solve the horizontal equation with the same $u$.
The comparison path $\bar\gamma$ stays in the smaller compact set.
The smooth control equations and Gronwall's inequality yield 
\begin{equation}\label{eq:penalized-deletion}
 \sup_t\dR(\gamma(t),\bar\gamma(t))\leq C\norm{w}_{L^1}
\end{equation}
up to a possible first exit. Here $\dR$ is the distance of the
fixed auxiliary metric. Finitely many charts give the estimate on
the whole compact comparison region. For sufficiently small $\varepsilon$, the bound in
\eqref{eq:penalized-deletion} is smaller than the distance
from the compact comparison region to the boundary of the
larger neighborhood. Thus no first exit can occur.
Original horizontal minimizers are admissible with $w=0$,
so $c_\varepsilon\leq c<E_*$. Hence every penalized minimizer
satisfies the preceding energy bound and remains in the
larger neighborhood. 
\end{proof} 

\begin{proposition}[Cost approximation and normal covectors]
	\label{prop:penalized-approximation}
	For all sufficiently small $\varepsilon>0$, uniformly on the
	fixed endpoint region, 
	\begin{equation}\label{eq:linear-cost-approximation}
		0\leq c-c_\varepsilon\leq C\varepsilon.
	\end{equation}
	For every constant-speed penalized minimizer on $[0,1]$
	with endpoints in these regions, its normal covector
	$p_\varepsilon$ satisfies
	\[
	\sup_{t\in[0,1]}|p_\varepsilon(t)|_{g_{\mathrm R}^*}\leq C.
	\]
	In fixed local coordinates, the paths, covectors, and controls $(u,w)$
	have uniformly bounded time derivatives of every fixed order.
	The constants may depend on the derivative order, but not on
	$\varepsilon$ or the endpoints.
\end{proposition} 
\begin{proof}
Let $(u,w)$ be the controls of a penalized minimizer, chosen as above,
from $x$ to $y$, and put $b=\norm{w}_{L^2}$.
By \eqref{eq:penalized-deletion}, deleting $w$ gives a horizontal
path ending at $y'$ with $\dR(y,y')\leq Cb$.
The energy bound implies $b=O(\sqrt\varepsilon)$, so $y'$ remains
in a fixed compact neighborhood of the target region.
By the uniform coordinate Lipschitz bound for $c$, 
\[
 c(x,y)\leq c(x,y')+Bb
 \leq\tfrac12\norm{u}_{L^2}^2+Bb
 =c_\varepsilon(x,y)-\tfrac{b^2}{2\varepsilon}+Bb.
\]
Combining with $c_\varepsilon\leq c$, we obtain 
\[
 0\leq c(x,y)-c_\varepsilon(x,y)
 \leq Bb-\frac{b^2}{2\varepsilon}
 \leq\frac{B^2}{2}\varepsilon. 
\]

\medskip
For the normal covector $p_\varepsilon$, these controls
satisfy $u_a=p_\varepsilon(F_a)$ and
$w_b=\varepsilon p_\varepsilon(G_b)$. By the preceding inequality,
\[
 b\leq2B\varepsilon,\qquad
\int_0^1|p_\varepsilon|_{g_{\mathrm R}^*}^2\d t
=\varepsilon^{-2}b^2\leq4B^2.
\] 
By the adjoint equation and the uniform $L^1$ bounds on the
controls, Gronwall's inequality yields
\[
\sup_{t\in[0,1]}|p_\varepsilon(t)|_{g_{\mathrm R}^*}
\leq C\left(
\int_0^1|p_\varepsilon|_{g_{\mathrm R}^*}^2\d t
\right)^{1/2}
\leq C.
\]
On the resulting compact subset of $T^*M$, the Hamiltonians
$\frac12\sum_a p(F_a)^2+\frac\varepsilon2\sum_b p(G_b)^2$
have uniformly bounded derivatives of every fixed order.
By differentiating Hamilton's equations and using the
identities for $u$ and $w$, we obtain the asserted
bounds on time derivatives. 
\end{proof}

\subsubsection{Endpoint corrections}
Write $U=(u,w/\sqrt\varepsilon)$, and let $E_\varepsilon(U)$ be
the endpoint map with fixed initial point. Its fields are
$(F_a,\sqrt\varepsilon G_b)$ and its energy is $\norm{U}_{L^2}^2/2$.
Let $P(1,t)$ be the differential of the flow from time $t$ to time one, with the
control fixed,
along the minimizer. Set
$A_a(t)=P(1,t)F_a(\gamma(t))$ and
$B_b(t)=P(1,t)G_b(\gamma(t))$. Then
\[
 DE_\varepsilon(U)(h,k)
 =\sum_a\int_0^1A_a(t)h_a(t)\d t
  +\sqrt\varepsilon\sum_b\int_0^1B_b(t)k_b(t)\d t.
\]
Write $L_Gk=\sum_b\int_0^1B_bk_b\d t$.
In this subsection, adjoints and operator norms use the Euclidean
inner product in the fixed endpoint coordinates and the usual
$L^2$ inner product on controls.
This flow differential and its inverse are uniformly bounded.
The cometric identity for $G_b$ yields $L_GL_G^*\geq c\Id$, and hence 
\begin{equation}\label{eq:penalized-right-inverse}
 Q_\varepsilon=
 (0,\varepsilon^{-1/2}L_G^*(L_GL_G^*)^{-1}),
 \qquad DE_\varepsilon(U)Q_\varepsilon=\Id,
 \qquad \norm{Q_\varepsilon}\leq C\varepsilon^{-1/2}.
\end{equation}
The first two derivatives of the endpoint map are uniformly bounded
on a fixed $L^2$ neighborhood, by their variational equations.
The right inverse $Q_\varepsilon$ gives only a squared control norm
of order $\varepsilon^{-1}$. To improve this bound, we use a horizontal
variation with vanishing time moments and correct only its endpoint error.

\begin{lemma}[Horizontal endpoint variations]\label{lem:moment-correction}
Fix the finite collection of smooth unit horizontal fields used
for the source paths. For every integer $N\geq0$, each field $X$,
and each penalized minimizer ending at $x$, there exists a control variation $q_X$ such that
\begin{equation}\label{eq:moment-correction}
 DE_\varepsilon(U)q_X=X(x),\qquad
 \norm{q_X}_{L^2}^2\leq C_N\varepsilon^{-1/(2N+3)}.
\end{equation}
\end{lemma}
\begin{proof}
Choose a polynomial $k_N$ of degree at most $N$ satisfying $\int_0^1k_N=1$ and $\int_0^1s^jk_N(s)\d s=0$ for $1\leq j\leq N$. 
Such a polynomial exists since the moment matrix $\bigl(\int_0^1 s^{i+j}\,\d s\bigr)_{i,j=0}^N$ is positive definite. Write $X(x)=\sum_a\xi_aF_a(x)$ with
bounded $\xi$. For $0<\tau<1$ define
\[
 h_a(t)=\frac{\xi_a}{\tau}
 k_N\left(\frac{1-t}{\tau}\right)\one_{[1-\tau,1]}(t).
\]
By Proposition~\ref{prop:penalized-approximation}, the derivatives
of $A_a$ of every fixed order are uniformly bounded.
Since $A_a(1)=F_a(x)$, Taylor expansion at $t=1$ and the
moment conditions on $k_N$ yield 
\[
DE_\varepsilon(U)(h,0)=X(x)+r_\tau,\qquad
|r_\tau|\leq C_N\tau^{N+1}.
\]
The definition of $h$ also gives
$\norm{h}_{L^2}^2\leq C_N/\tau$.
Set $q_X=(h,0)-Q_\varepsilon r_\tau$. Then 
\[
 \norm{q_X}_{L^2}^2\leq C_N
 (\tau^{-1}+\varepsilon^{-1}\tau^{2N+2}).
\]
Taking $\tau=\varepsilon^{1/(2N+3)}$ proves the result. 
\end{proof}

\begin{proposition}[Directional bounds for the penalized cost]
\label{prop:penalized-directional}
For every $N\geq0$, the cost $c_\varepsilon$ satisfies
\eqref{eq:general-directional-bounds} with
\[
 L_1\leq C,\qquad L_2\leq C_N\varepsilon^{-1/(2N+3)}. 
\]
The bounds hold along all the selected flows, including at cut
points and diagonal points.
\end{proposition}

\begin{proof}
	Since $X$ is unit horizontal and $\dist_\varepsilon\leq\dcc$,
	\[
	\dist_\varepsilon(\Phi_X^s x,\Phi_X^t x)\leq |s-t|.
	\]
	The triangle inequality and the uniform distance bound yield $L_1\leq C$.
	
	Reverse a penalized minimizer so that it ends at $x$, and
	let $U=(u,w/\sqrt\varepsilon)$ be its control in the rescaled variables. Choose $q_X$ from
	Lemma~\ref{lem:moment-correction} and a linear functional
	$\ell$ with $\ell(X(x))=1$. Set
	\[
	\widehat Q_\varepsilon
	=Q_\varepsilon+
	(q_X-Q_\varepsilon X(x))\otimes\ell.
	\]
	Then $DE_\varepsilon(U)\widehat Q_\varepsilon=\Id$ and
	$\widehat Q_\varepsilon X(x)=q_X$.
	The inverse function theorem applied to
	$a\mapsto E_\varepsilon(U+\widehat Q_\varepsilon a)$ yields a $C^2$ curve of controls satisfying $E_\varepsilon(U(s))=\Phi_X^s x$, $U(0)=U$, and $U'(0)=q_X$. 
	Thus $S(s)=\norm{U(s)}_{L^2}^2/2$ is an upper support for
	$f(s):=c_\varepsilon(\Phi_X^s x,y)$ at $s=0$.

	\medskip
	Let $p$ be the terminal normal covector. By first-order optimality
	under the endpoint constraint, $U=DE_\varepsilon(U)^*p$.
	In fixed endpoint coordinates, differentiating the constraint
	twice yields 
	\[
	DE_\varepsilon(U)U''(0)
	+D^2E_\varepsilon(U)[q_X,q_X]
	=\left.\frac{\der^2}{\der s^2}\Phi_X^s x\right|_{s=0}.
	\]
	Hence
	\[
	\begin{aligned}
		S''(0)
		=\norm{q_X}_{L^2}^2+\langle U,U''(0)\rangle_{L^2}=\norm{q_X}_{L^2}^2-p\bigl(D^2E_\varepsilon(U)[q_X,q_X]\bigr)+p\left(
		\left.\frac{\der^2}{\der s^2}\Phi_X^s x\right|_{s=0}
		\right).
	\end{aligned}
	\] 
	By the uniform bounds on $p$, $D^2E_\varepsilon(U)$,
	and the flow acceleration, together with
	Lemma~\ref{lem:moment-correction},
	\[
	S''(0)\leq C(1+\norm{q_X}_{L^2}^2)
	\leq C_N\varepsilon^{-1/(2N+3)}=:K.
	\]

	\medskip
	For every $\delta>0$, continuity of $S''$ gives, near zero,
	\[
	f(s)\leq S(s)
	\leq f(0)+S'(0)s+\tfrac12(K+\delta)s^2.
	\]
	The same construction applies at every point of the flow interval.
	Hence $f(s)-(K+\delta)s^2/2$ admits a local affine upper support
	at every point and is concave by the one-dimensional
	criterion that a continuous function with a local affine upper
	support at every point is concave. Thus $f''\leq K+\delta$ in
	distributions. Letting $\delta\downarrow0$ proves the bound.
	The argument also applies when $x=y$, with $U=0$ and $p=0$. 
\end{proof}

\begin{proof}[Proof of Theorem~\ref{thm:potential-stability}\textup{(ii)}]
First take arbitrary $\theta_0,\theta_1\in C(Y)$ with
$\osc_Y(\theta_1-\theta_0)\leq R_0$. By Theorem~\ref{thm:directional-jensen}
and Proposition~\ref{prop:penalized-directional}, we obtain 
\[
 \Var_\rho(\theta_1^{c_\varepsilon}-\theta_0^{c_\varepsilon})
 \leq C_{N,R_0}\varepsilon^{-1/(2N+3)}
 J_{c_\varepsilon}(\theta_0,\theta_1).
\]
Set $J=J_c(\theta_0,\theta_1)$. By \eqref{eq:linear-cost-approximation} and
Proposition~\ref{prop:cost-comparison}, 
\[
 \Var_\rho(\theta_1^c-\theta_0^c)\leq C_{N,R_0}\varepsilon^{-1/(2N+3)}(J+\varepsilon)
       +C\varepsilon^2.
\]
By the rate estimate in Section~\ref{sec:cost-approximation}
with $\eta=1/(2N+3)$, we obtain
\begin{equation}\label{eq:subcritical-jensen}
 \Var_\rho(\theta_1^c-\theta_0^c)
 \leq C_{N,R_0}J_c(\theta_0,\theta_1)^{(2N+2)/(2N+3)}.
\end{equation}
For target Kantorovich potentials, take $R_0=2D^2$.
By \eqref{eq:optimal-jensen-gap} and
Proposition~\ref{prop:potential-conversion}, we obtain uniqueness
and stability with exponent $(N+1)/(2N+3)$.
These exponents increase to $1/2$; since $W_1$ is bounded,
every $0<\alpha<1/2$ follows by choosing $N$ sufficiently large. 
\end{proof}

\Needspace{12\baselineskip}
\subsection{Quantitative approximation at arbitrary step}\label{sec:general-carnot-potentials}
\subsubsection{Carnot groups}
We now replace the directional second-order bounds by quantitative
first-order approximation. Throughout this subsubsection, $G$ is a
Carnot group of arbitrary step, with Haar measure $\mm$ and
homogeneous dimension
\[
 Q=\sum_{j=1}^s j\dim V_j.
\]
The source domain $S$ is bounded and John with respect to the
Carnot--Carath\'eodory distance,
$Y\subset G$ is compact, and
$a_1\mm\restr S\leq\rho\leq a_2\mm\restr S$.
We retain the notation $J_c,V_c$ from
\eqref{eq:potential-gap-variance}.

\begin{theorem}[Logarithmic coercivity]\label{thm:carnot-log-coercivity}
There exist $C,J_0>0$, depending only on the fixed group, source,
target region, and density bounds, such that, for every
$\theta_0,\theta_1\in C(Y)$,
\begin{equation}\label{eq:carnot-log-coercivity}
 V_c(\theta_0,\theta_1)
 \leq C\left[\log\left(e+\frac{J_0}{J_c(\theta_0,\theta_1)}\right)
       \right]^{-1/(Q+2)}.
\end{equation}
The right-hand side is defined to be zero when $J_c=0$.
Consequently the normalized source potential is unique for every
$\mu\in\PP(Y)$, and, with $J_0$ as above,
\begin{equation}\label{eq:carnot-log-stability}
 \norm{\phi_{\mu_1}-\phi_{\mu_0}}_{L^2(\rho)}
 \leq C\left[\log\left(e+\frac{2J_0}{D W_1(\mu_0,\mu_1)}\right)
       \right]^{-1/[2(Q+2)]}.
\end{equation}
\end{theorem}

The proof uses first-order approximation in the source variable
and allows arbitrary targets and abnormal minimizing curves.
The first lemma bounds the difference between the horizontal
gradients of two affine approximations when the corresponding
functions satisfy contact inequalities with the same target point
$y$ and with errors $e_0,e_1$.
Let $\pi_1:G\to V_1$ be projection onto the first layer in
exponential coordinates. This identifies the abelian quotient
$G/[G,G]$ with the additive group $V_1$, so $\pi_1$ is a group
homomorphism and
$|\pi_1(x^{-1}z)|\leq\dcc(x,z)$. In this subsection, we identify
horizontal gradients with their components in the left-invariant
orthonormal frame.

\begin{lemma}[Gradient comparison from contact inequalities]\label{lem:common-approximate-contact}
Fix bounds $D_*,L_*>0$. Let $x,y\in G$, with
$\dcc(x,y)\leq D_*$, and $0<r\leq1$.
For $i=0,1$, suppose that $f_i$ and a function affine in $\pi_1(x^{-1}z)$,
\[
 A_i(z)=a_i+\langle p_i,\pi_1(x^{-1}z)\rangle
\]
satisfy $|p_i|\leq L_*$,
\[
 \sup_{B(x,r)}|f_i-A_i|\leq\varepsilon_i r,
 \qquad
 f_i(z)-f_i(x)\leq c(z,y)-c(x,y)+e_i
 \quad(z\in B(x,r)),
\]
where $e_i,\varepsilon_i\geq0$. Then
\begin{equation}\label{eq:common-contact-slopes}
 |p_1-p_0|^2
 \leq C\left(r+\varepsilon_0+\varepsilon_1
                +\frac{e_0+e_1}{r}\right),
\end{equation}
where $C$ depends only on $D_*,L_*$.
\end{lemma}
\begin{proof}
Put $d=\dcc(x,y)$ and
$t_i=r/2+2\varepsilon_i+e_i/r$.
Testing at $z=x\exp(ru)$, $|u|=1$, and using
$\dcc(z,y)\leq d+r$, gives
\begin{equation}\label{eq:contact-slope-norm}
 |p_i|\leq d+t_i.
\end{equation}
The estimate at distance $r$ follows by taking interior points
and using continuity of $A_i$ and $c$.
If $d\geq r$, take a unit-speed minimizing curve from $x$ to
$y$, and let $z$ be its point at time $r$. Then
$\dcc(z,y)=d-r$ and
$v=\pi_1(x^{-1}z)/r$ satisfies $|v|\leq1$. Hence
\[
 \langle p_i,v\rangle\leq-d+t_i.
\]
For $t_0+t_1\leq1$, this and
\eqref{eq:contact-slope-norm} imply
\[
 |p_i+dv|^2
 \leq(d+t_i)^2+2d(-d+t_i)+d^2
 \leq(4D_*+1)t_i.
\]
The triangle inequality followed by Cauchy--Schwarz proves the
claim in this case.
If $t_0+t_1>1$, the bound $|p_i|\leq L_*$ proves it directly.
Finally, if $d<r$, \eqref{eq:contact-slope-norm} and the same
uniform bound give
\[
 |p_1-p_0|^2\leq2L_*(|p_0|+|p_1|)
 \leq C(r+t_0+t_1).
\]
This proves \eqref{eq:common-contact-slopes}.
\end{proof}

For a compactly supported Lipschitz function $f:G\to\RR$, let
$A_{x,r}f$ be its orthogonal projection in $L^2(B(x,2r),\mm)$
onto the functions affine in $\pi_1(x^{-1}z)$. Write
\[
 A_{x,r}f(z)=a_{x,r}+\langle p_{x,r},\pi_1(x^{-1}z)\rangle,
\]
and set
\begin{equation}\label{eq:affine-errors}
 \begin{aligned}
 b_f(x,r)^2&=
 \frac{1}{r^2\mm(B(x,2r))}
 \int_{B(x,2r)}|f-A_{x,r}f|^2\,\d\mm,\\
 \varepsilon_f(x,r)&=
 \frac1r\sup_{B(x,r)}|f-A_{x,r}f|.
 \end{aligned}
\end{equation}
We use Dorronsoro's theorem in the following form.

\begin{lemma}[Affine approximation and smoothing]\label{lem:carnot-affine-smoothing}
If $f$ is compactly supported and $L$-Lipschitz, then
\begin{equation}\label{eq:carnot-dorronsoro}
 \int_0^\infty\int_G b_f(x,r)^2\,\d\mm(x)\frac{\d r}{r}
 \leq C_G\int_G|\nabla_{\mathcal H}f|^2\,\d\mm.
\end{equation}
Moreover,
\begin{equation}\label{eq:affine-uniform-from-L2}
 |p_{x,r}|\leq C_GL,\qquad
 \varepsilon_f(x,r)\leq C_GL,\qquad
 \varepsilon_f(x,r)^{Q+2}\leq C_GL^Q b_f(x,r)^2.
\end{equation}
There are smooth averaging operators $P_r$ such that
\begin{equation}\label{eq:carnot-smoothing-affine}
 \norm{P_rf-f}_{L^\infty(G)}\leq Lr,
 \qquad
 |\nabla_{\mathcal H}P_rf(x)-p_{x,r}|
 \leq C_G\varepsilon_f(x,r).
\end{equation}
\end{lemma}
\begin{proof}
Equation \eqref{eq:carnot-dorronsoro} is the case
$p=q=2$, $\alpha=1$ of Ryoo's Dorronsoro theorem on Carnot
groups \cite[Theorem~34]{Ryoo2026}. The polynomials of weighted
degree at most one in that theorem are precisely the functions affine in the first-layer coordinates. For the nonnegative sub-Laplacian
$\mathscr L=-\sum_{j=1}^{\dim V_1}X_j^2$, its quadratic form gives
$\|\mathscr L^{1/2}f\|_{L^2}=\|\nabla_{\mathcal H}f\|_{L^2}$. Replacing its ball radius
by $2r$ changes only a fixed factor.

Translation and dilation reduce the first two estimates in
\eqref{eq:affine-uniform-from-L2} to orthogonal projection onto a
fixed finite-dimensional space on $B(e,2)$. The function
$z\mapsto[f(x\delta_r z)-f(x)]/r$ is bounded by $2L$ on $B(e,2)$.
Its projection has coefficients bounded by $C_GL$. Scaling back
gives the bounds for $p_{x,r}$ and $\varepsilon_f(x,r)$.
Thus $f-A_{x,r}f$ is $C_GL$-Lipschitz. If its absolute value
at a point of $B(x,r)$ is $\varepsilon r$, it is at least
$\varepsilon r/2$ on a ball of radius $c\varepsilon r/L$
contained in $B(x,2r)$. Haar measure scales as the $Q$th
power of the radius. Integrating on this ball proves the last
estimate in \eqref{eq:affine-uniform-from-L2}; the case $L=0$ is immediate because $f$ is constant.

Choose a nonnegative $\eta\in C_c^\infty(B(e,1/2))$ with
integral one, put
$\eta_r(z)=r^{-Q}\eta(\delta_{1/r}z)$, and define
\[
 P_rf(x)=(f*\eta_r)(x)
       =\int_G f(z)\eta_r(z^{-1}x)\,\d\mm(z).
\]
The kernel is supported where $\dcc(x,z)<r/2$, giving the
first estimate in \eqref{eq:carnot-smoothing-affine}.
For a left-invariant unit horizontal field $X$,
\[
 X(P_rf)(x)=\int_G f(z)(X\eta_r)(z^{-1}x)\,\d\mm(z),
 \qquad \norm{X\eta_r}_{L^1}\leq C_G/r.
\]
Convolution sends each function affine in $\pi_1$ to another
function of this form with the same horizontal gradient. Subtracting
$A_{x,r}f$ inside this derivative therefore proves the second
estimate in \eqref{eq:carnot-smoothing-affine}.
\end{proof}

\begin{proof}[Proof of Theorem~\ref{thm:carnot-log-coercivity}]
Choose a relatively compact neighborhood $U$ of $\overline S$ and
$r_0\in(0,1)$ such that $B(x,2r_0)\subset U$ for all $x\in S$.
Put
\[
 \phi_i=\theta_i^c,\qquad
 \phi_{1/2}=\left(\frac{\theta_0+\theta_1}{2}\right)^c,
 \qquad
 \gamma=\phi_{1/2}-\frac{\phi_0+\phi_1}{2}.
\]
Then $\gamma\geq0$ and $\int_S\gamma\,\d\rho=J$, where
$J=J_c(\theta_0,\theta_1)$.
The triangle inequality implies that the $\phi_i$ are uniformly
Lipschitz on $U$, independently of $\theta_0,\theta_1$. Subtract the value of each $\phi_i$ at a fixed point. Extend the
resulting functions by the McShane formula. Truncate the extensions
to a fixed bounded interval, then multiply by a fixed Lipschitz
cutoff equal to one on $U$. The resulting compactly supported
functions $f_i$ agree with the normalized $\phi_i$ on $U$ and satisfy
\begin{equation}\label{eq:extended-potential-energy}
 \Lip(f_i)\leq L,\qquad
 \int_G|\nabla_{\mathcal H}f_i|^2\,\d\mm\leq E,
\end{equation}
for constants independent of $\theta_i$.
Subtracting constants does not affect the estimates below.

Fix $x\in S$. Compactness of $Y$ gives a minimizer $y$ in the
formula for $\phi_{1/2}(x)$. The contact errors
\[
 e_i=c(x,y)-\theta_i(y)-\phi_i(x)\geq0
\]
satisfy $e_0+e_1=2\gamma(x)$, and the defining infimum gives
\[
 \phi_i(z)-\phi_i(x)
 \leq c(z,y)-c(x,y)+e_i\qquad(z\in U).
\]
Use the affine functions in \eqref{eq:affine-errors} and
Lemma~\ref{lem:common-approximate-contact}. Together with
\eqref{eq:affine-uniform-from-L2} and
\eqref{eq:carnot-smoothing-affine}, this yields, for $0<r\leq r_0$,
\begin{equation}\label{eq:smoothed-gradient-contact}
 |\nabla_{\mathcal H}P_r(f_1-f_0)(x)|^2
 \leq C\left(r+\varepsilon_{f_0}(x,r)
                  +\varepsilon_{f_1}(x,r)+\frac{\gamma(x)}r\right).
\end{equation}
Here the squared smoothing errors are absorbed using their uniform
bound. The choice of $y$ is only pointwise; no selection theorem
is needed to integrate this inequality.

The local horizontal $L^2$ Poincar\'e inequality on Carnot groups
\cite{Jerison1986}, followed by the summation of local variances in
Lemma~\ref{lem:john-gluing}, gives
\begin{equation}\label{eq:carnot-john-Poincare}
 \Var_\rho(u)\leq C\int_S|\nabla_{\mathcal H}u|^2\,\d\rho.
\end{equation}
This summation uses only local doubling, the local Poincar\'e inequality,
and the John condition, all valid at arbitrary step. Applying it
to $P_r(f_1-f_0)$ and using
\eqref{eq:carnot-smoothing-affine} and
\eqref{eq:smoothed-gradient-contact}, we obtain
\begin{equation}\label{eq:carnot-one-scale-coercivity}
 V_c(\theta_0,\theta_1)
 \leq C\left(r+\int_S(\varepsilon_{f_0}(x,r)
                    +\varepsilon_{f_1}(x,r))\,\d\rho(x)
                 +\frac Jr\right).
\end{equation}
No boundary term arises, since the extensions agree with the
potentials on a fixed neighborhood of $\overline S$ and the
contact errors in \eqref{eq:smoothed-gradient-contact} are evaluated
only on $S$.

For $0<a<b\leq r_0$, equations
\eqref{eq:carnot-dorronsoro} and
\eqref{eq:extended-potential-energy} provide a radius $r\in[a,b]$
with
\[
 \int_S\bigl(b_{f_0}(x,r)^2+b_{f_1}(x,r)^2\bigr)\,\d\rho(x)
 \leq\frac{C}{\log(b/a)}.
\]
By \eqref{eq:affine-uniform-from-L2} and H\"older's inequality,
\[
 \int_S(\varepsilon_{f_0}(x,r)+\varepsilon_{f_1}(x,r))\,\d\rho(x)
 \leq C[\log(b/a)]^{-1/(Q+2)}.
\]
Consequently
\begin{equation}\label{eq:carnot-scale-window}
 V_c(\theta_0,\theta_1)
 \leq C\left(b+[\log(b/a)]^{-1/(Q+2)}+\frac Ja\right).
\end{equation}
For $0<J<J_0/e$, set $\tau=J/J_0$ and choose
$a=r_0\tau^{1/2}$, $b=r_0\tau^{1/4}$.
Then \eqref{eq:carnot-scale-window} gives
\eqref{eq:carnot-log-coercivity}. For $J\geq J_0/e$, the same
claim follows by increasing $C$, since the oscillations of the
$\phi_i$ on $S$ are uniformly bounded.
If $J=0$, take $a=r_0e^{-2N}$ and $b=r_0e^{-N}$ in
\eqref{eq:carnot-scale-window} and let $N\to\infty$.
Finally Proposition~\ref{prop:potential-conversion} proves
uniqueness and \eqref{eq:carnot-log-stability}.
For fixed $a,b>0$, the functions $\log(e+a/w)$ and
$\log(e+b/w)$ are comparable uniformly for $w>0$.
Taking $a=2J_0/D$ and $b=D$ therefore gives the normalization
in~\eqref{eq:main-log-potential-stability}.
\end{proof}

The logarithmic exponent in \eqref{eq:carnot-log-stability} comes
from passing from $L^2$ to uniform affine approximation on balls
of homogeneous dimension $Q$, then selecting a scale at which
the approximation error is small. We do not know whether this
exponent is optimal. On step-two Carnot groups, horizontal
semiconcavity gives $V_c\leq CJ_c$ directly through
Theorem~\ref{thm:directional-jensen}, yielding the sharp exponent
$1/2$ without scale selection.

\subsubsection{Equiregular structures}\label{sec:equiregular-potentials}
The argument extends to equiregular structures through quantitative
nilpotent approximation in the source variable. To retain the
logarithmic rate, we need bounds on the approximation errors and constants uniform
over the tangent groups. The next lemma gives these bounds.

Shrinking the equiregular neighborhood if necessary, choose
$\overline S\subset U\Subset U_+$ within it. The step $s$,
horizontal rank $h_0$, and homogeneous dimension $Q$ are fixed
there. All constants below are uniform on these neighborhoods.
The case $s=1$ is included. Put $n=\dim M$. Choose smooth frames
$Z_1,\ldots,Z_n$ adapted to the spaces $\mathcal H^j$ on finitely
many neighborhoods, with $Z_1,\ldots,Z_{h_0}$ orthonormal and
horizontal. Write $X_a=Z_a$ for $1\leq a\leq h_0$. Let $w_j$
be the least integer for which $Z_j$ belongs to $\mathcal H^{w_j}$. Set
$\delta_R(z_1,\ldots,z_n)=(R^{w_1}z_1,\ldots,R^{w_n}z_n)$ and
$\Phi_p(z)=\exp(\sum_jz_jZ_j)(p)$. These canonical coordinates
of the first kind are Carnot coordinates by
\cite[Proposition~6.2]{ChoiPonge2019II}. Write
$\Phi_{p,R}(z)=\Phi_p(\delta_Rz)$ and
$X_a^{p,R}=R(\delta_{1/R})_*(\Phi_p^{-1})_*X_a$ for this horizontal frame. Denote the
tangent Carnot group by $G_p$, its distance by $\widehat d_p$,
and its Haar measure by $\mm_{G_p}$, normalized as in the proof
below. The notation $B_{G_p}$ refers to balls for $\widehat d_p$.
Write $\widehat X_a^p$ for the left-invariant horizontal fields
of the nilpotent approximation.
The uniform convergence of the resulting rescaled fields and local
distances is given by \cite[Theorem~4.6]{DonMagnani2023}.
We need a rate for the fields: on any fixed coordinate box,
\begin{equation}\label{eq:equiregular-field-error}
 X_a^{p,R}=\widehat X_a^p+O(R)\quad\hbox{in }C^k,
 \qquad a=1,\ldots,h_0,
\end{equation}
for every fixed $k$, uniformly in $p$.
The expansion for a fixed base point is given in
\cite[Theorem~5.6, equation~(5.5)]{ChoiPonge2019I}.
To see the uniform rate, write the horizontal field in these
coordinates as $\sum_j a_{aj}(p,z)\partial_{z_j}$, where $z_j$
has weight $w_j$. Its rescaled coefficient is
$R^{1-w_j}a_{aj}(p,\delta_Rz)$. The weighted Taylor coefficients
below degree $w_j-1$ vanish, and the terms of degree $w_j-1$
give $\widehat X_a^p$. All remaining terms have a factor $R$.
The coefficients depend smoothly on $p$ by smooth dependence of
the coordinate flows. Taylor's formula to sufficiently high finite
order therefore bounds the remainder in $C^k$ by $C_kR$ on
each compact parameter set. A finite cover makes the bound uniform.
The coefficients in the first-layer directions of $\widehat X_a^p$ are the constant
coordinate vectors. The reference measure has the form
$\Phi_{p,R}^*\mm=R^Qj_{p,R}(z)\,\d z$, with
$c\leq j_{p,R}\leq C$ on these boxes.

\begin{lemma}[Uniform estimates on the tangent Carnot groups]\label{lem:uniform-tangent-estimates}
After choosing fixed nested coordinate neighborhoods, the following
estimates hold uniformly in $p$ and sufficiently small $R$.
\begin{enumerate}[(a)]
\item For $z,w$ in the smaller neighborhood,
\begin{equation}\label{eq:equiregular-metric-error}
 \big|\dcc(\Phi_{p,R}z,\Phi_{p,R}w)
       -R\widehat d_p(z,w)\big|\leq CR^{1+1/s}.
\end{equation}
\item If $F$ is uniformly Lipschitz on a fixed neighborhood in
$G_p$, after multiplication by a cutoff supported in a fixed coordinate
neighborhood it satisfies
\begin{equation}\label{eq:uniform-tangent-dorronsoro}
 \int_{B_{G_p}(e,A)}\int_0^{r_*}
 \beta_F(z,r)^2\,\frac{\d r}{r}\,\d\mm_{G_p}(z)\leq C.
\end{equation}
Here $A,r_*>0$ are fixed, the cutoff equals one on every ball
$B_{G_p}(z,8r)$ with $z\in B_{G_p}(e,A)$ and $r\leq r_*$, and
\[
 \beta_F(z,r)^2:=\inf_{a,\ell}
 \frac{1}{\mm_{G_p}(B_{G_p}(z,8r))}\int_{B_{G_p}(z,8r)}
 \frac{|F(w)-a-\ell(\pi_1(z^{-1}w))|^2}{r^2}\,\d\mm_{G_p}(w).
\]
The bound assumes that $F(e)$ and its Lipschitz constant are bounded.
\item There are smooth kernels $\eta_{p,r}$ supported in
$B_{G_p}(e,r/2)$, of Haar integral one, such that for every
left-invariant field $Z_j^p$ in the chosen graded frame, of
weight $w_j$,
\begin{equation}\label{eq:uniform-tangent-kernels}
 \|Z_j^p\eta_{p,r}\|_{L^1(G_p)}\leq Cr^{-w_j}.
\end{equation}
For $P_r^pF=F*\eta_{p,r}$ and $F$ with Lipschitz constant $L$,
\begin{equation}\label{eq:uniform-tangent-smoothing}
 \|P_r^pF-F\|_\infty\leq Lr,\qquad
 |Z_j^pP_r^pF|\leq CLr^{1-w_j}.
\end{equation}
The estimates for the affine projections and horizontal gradients in
\eqref{eq:affine-uniform-from-L2} and
\eqref{eq:carnot-smoothing-affine} also hold with constants
independent of $p$.
\end{enumerate}
\end{lemma}

\begin{proof}
For (a), compare curves with the same horizontal controls for the
rescaled fields and their nilpotent approximations. Curves of
bounded length remain in a fixed larger coordinate neighborhood.
Equation~\eqref{eq:equiregular-field-error} and Gronwall's inequality
give an $O(R)$ Euclidean discrepancy between their endpoints.
The uniform ball--box estimate corrects this discrepancy by a
horizontal curve of length at most $CR^{1/s}$ in either rescaled
structure. Apply this first to a minimizing tangent curve and then
to a minimizing rescaled curve. The selected brackets in the
adapted frame have determinants bounded away from zero for both
the limiting and rescaled fields; this gives the uniform ball--box
constants used in the correction. Choose the smaller and larger
coordinate neighborhoods so that every curve of the relevant
bounded length stays in the larger one. The local distances then
equal the ambient distances at the scales in the original manifold
considered here; see also \cite[Proposition~2.15]{DonMagnani2023}.
Rescaling proves~\eqref{eq:equiregular-metric-error}.

For (b), obtain uniform constants from a single free Carnot group.
Let $\mathbb F$ be the free Carnot group of rank $h_0$ and step
$s$, with the standard inner product on its first layer.
The Lie algebra of $\mathbb F$ is the free nilpotent Lie algebra
on $h_0$ generators, with all brackets of length greater than $s$
set to zero. For each $p$, there is a surjective Lie group
homomorphism $\Pi_p:\mathbb F\to G_p$ commuting with dilations.
Its differential sends the free generators isometrically to the
chosen orthonormal horizontal frame. Its kernel $N_p$ is a normal
subgroup invariant under dilations. In particular,
\[
 \mathfrak n_p\cap V_1(\mathbb F)=\{0\},\qquad
 \dim(\mathfrak n_p\cap V_j(\mathbb F))
       =\dim V_j(\mathbb F)-\dim V_j(G_p),
\]
where $\mathfrak n_p$ is the Lie algebra of $N_p$.
These dimensions are constant by equiregularity.
Horizontal curves project and lift with the same control and length.
Consequently $\Pi_p(B_{\mathbb F}(u,r))=B_{G_p}(\Pi_pu,r)$
for every $u,r$. Put
$Q_{\mathbb F}=\sum_j j\dim V_j(\mathbb F)$, the homogeneous
dimension of $\mathbb F$.
Give each homogeneous layer of the free Lie algebra a fixed
Euclidean inner product, and give $N_p$ the product of the induced
layerwise volumes in exponential coordinates. This is Haar measure.
Disintegrate the fixed Haar measure of $\mathbb F$ along these
fibers to obtain the quotient measure $\mm_{G_p}$.
The layerwise matrices of $\der\Pi_p$ vary smoothly and are
surjective with constant rank. On a compact parameter set their
nonzero singular values have positive lower and finite upper
bounds. The quotient measures are consequently uniformly
comparable with Lebesgue measure $\d z$ in exponential coordinates.

We use the following ball estimate for the disintegration. If
$v\in B_{G_p}(\Pi_pu,r/4)$, this equality of ball images gives a lift
$\widetilde v\in B_{\mathbb F}(u,r/4)$. Therefore
\[
 \widetilde v\bigl(N_p\cap B_{\mathbb F}(e,r/4)\bigr)
       \subseteq B_{\mathbb F}(u,r).
\]
The Haar measure of this fiber piece is bounded below by
$cr^{Q_{\mathbb F}-Q}$, uniformly in $p$. Indeed, the unit ball of $\mathbb F$, viewed in exponential coordinates,
contains a Euclidean ball of positive radius and is contained
in a bounded Euclidean ball. Its intersections with
the linear spaces $\mathfrak n_p$ have induced volumes uniformly
bounded above and away from zero, since their dimension is fixed.
The restriction of $\delta_r$ to $\mathfrak n_p$ has determinant
$r^{Q_{\mathbb F}-Q}$ by the layerwise dimension formula.
Translation within a fiber preserves its Haar measure.
If the kernel is trivial, use unit mass on that fiber.
Consequently,
\begin{equation}\label{eq:free-quotient-ball}
 (\Pi_p)_\#\bigl(\one_{B_{\mathbb F}(u,r)}\mm_{\mathbb F}\bigr)
 \geq cr^{Q_{\mathbb F}-Q}
       \one_{B_{G_p}(\Pi_pu,r/4)}\mm_{G_p}.
\end{equation}

Lift $F$ to $F\circ\Pi_p$ and multiply by a fixed cutoff on
$\mathbb F$, equal to one on all the balls used below. Denote the
result by $\widetilde F_p$. Its horizontal Dirichlet energy
$\int_{\mathbb F}|\nabla_{\mathcal H}\widetilde F_p|^2\,\d\mm_{\mathbb F}$
is bounded independently of $p$: the lift preserves the horizontal
Lipschitz bound, and $F(e)$ bounds its values on the fixed support.
Every weighted polynomial of degree at most one on $\mathbb F$
has the form $a+\ell(\pi_1^{\mathbb F}(u^{-1}v))$.
Since $\der\Pi_p$ is an isomorphism on the first layer, this
polynomial descends to a function affine in the first-layer
coordinates of $G_p$.

More explicitly, let $u$ lie in a fixed ball of $\mathbb F$ and take
the best affine approximant to $\widetilde F_p$ on
$B_{\mathbb F}(u,32r)$. Apply~\eqref{eq:free-quotient-ball}
at radius $32r$ to the squared error of its descended polynomial.
The ball volumes on $\mathbb F$ and $G_p$ are comparable to
$r^{Q_{\mathbb F}}$ and $r^Q$, respectively, uniformly in $p$. Dividing by the volumes
and the squared radii gives
\begin{equation}\label{eq:free-quotient-affine-error}
 \beta_F(\Pi_pu,r)^2
       \leq C b_{\widetilde F_p}(u,16r)^2,
\end{equation}
where $b$ is defined by~\eqref{eq:affine-errors} on $\mathbb F$.
Replacing the descended polynomial by the best affine approximation
on $G_p$ can only reduce the left-hand side.
Choose a fixed radius $A_{\mathbb F}>4A$ and a cutoff equal to
one on $B_{\mathbb F}(e,A_{\mathbb F}+32r_*)$.
A further application of~\eqref{eq:free-quotient-ball}, at the
fixed radius $A_{\mathbb F}$, gives
\[
 \int_{B_{G_p}(e,A)}\beta_F(z,r)^2\,\d\mm_{G_p}(z)
 \leq C\int_{B_{\mathbb F}(e,A_{\mathbb F})}
                 b_{\widetilde F_p}(u,16r)^2\,\d\mm_{\mathbb F}(u).
\]
Integrate in $\d r/r$ and apply~\eqref{eq:carnot-dorronsoro}
on the single fixed group $\mathbb F$. This proves (b), without
assuming continuous dependence of Dorronsoro constants on a group.

For (c), work in the smoothly varying graded exponential frames.
The group laws and the coefficients of their left-invariant
fields are polynomials with uniformly bounded coefficients on
fixed boxes. The uniform ball--box estimates give a fixed
coordinate neighborhood contained in every $B_{G_p}(e,1/2)$.
Choose a nonnegative smooth function supported in this neighborhood
and divide it by its $\mm_{G_p}$-integral to obtain $\eta_p$.
The normalization factors and coordinate derivatives of every fixed order
are uniformly bounded. Set
$\eta_{p,r}(z)=r^{-Q}\eta_p(\delta_{1/r}z)$.
Homogeneity of $Z_j^p$ proves~\eqref{eq:uniform-tangent-kernels}.
For the convolution in the order used in
Lemma~\ref{lem:carnot-affine-smoothing}, cancellation gives
\[
 Z_j^pP_r^pF(x)=\int_{G_p}(F(z)-F(x))
        (Z_j^p\eta_{p,r})(z^{-1}x)\,\d\mm_{G_p}(z).
\]
The support and Lipschitz bound imply
\eqref{eq:uniform-tangent-smoothing}.
For uniformity of the affine projections, reduce to a unit ball.
That ball contains a fixed coordinate neighborhood and is
contained in a fixed box. The Gram matrix of $1,z_1,\ldots,z_{h_0}$
therefore has its smallest eigenvalue bounded below and its
largest eigenvalue bounded above. This bounds the affine slopes
uniformly. The lower bound on the measure of a ball then gives the uniform estimate of the supremum error by the $L^2$ error exactly as in
\eqref{eq:affine-uniform-from-L2}.
Convolution preserves the horizontal slope of affine functions;
subtracting the affine approximant and using
\eqref{eq:uniform-tangent-kernels} proves the gradient assertion.
Finally, coordinate derivatives on a fixed box are bounded
linear combinations of the fields $Z_j^p$, with uniformly
bounded coefficients. Hence
$|\partial_{z_j}P_r^pF|\leq CLr^{1-s}$ for $r\leq1$,
which will be used to compare tangent and manifold derivatives.
\end{proof}

\begin{proof}[Proof of Theorem~\ref{thm:potential-stability}\textup{(i)}]
We transfer the logarithmic estimate from the tangent groups. Put
$\phi_i=\theta_i^c$, $i=0,1$, and
\[
 \gamma(x)=\left(\frac{\theta_0+\theta_1}{2}\right)^c(x)
       -\frac{\phi_0(x)+\phi_1(x)}2,\qquad
 J=\int_S \gamma\,\d\rho.
\]
Thus $\gamma\geq0$. The functions $\phi_i$ are uniformly
$\dcc$-Lipschitz on $U$, independently of $\theta_0,\theta_1$.

\medskip
\noindent\textbf{Step 1: Localization in rescaled coordinates.}
Fix a constant $J_0>0$, independent of $\theta_0,\theta_1$.
First suppose $0<J<J_0\tau_0$, and put
\begin{equation}\label{eq:equiregular-scales}
 \tau=J/J_0,\qquad R=R_0\tau^{1/4},\qquad
 \tau^{1/(8s)}\leq r\leq\tau^{1/(16s)},\qquad H=Rr.
\end{equation}
Here $R_0>0$ is fixed so that all enlarged patches lie in $U_+$, and
$\tau_0>0$ is small enough that all scales remain in the chosen
charts. Fixed factors depending on $R_0,J_0$ are absorbed in $C$.

Cover a fixed neighborhood of $\overline S$ by coordinate patches
of Carnot--Carath\'eodory diameter comparable to $R$, with bounded
overlap and a smooth partition of unity
$\{\chi_p\}$ satisfying $|\nabla_{\mathcal H}\chi_p|\leq C/R$.
The patches and their fixed enlargements have measure comparable
with $R^Q$, and their total measure is bounded independently of $R$.
Such a cover is obtained from a maximal separated set at scale
$R$. Pulling back fixed smooth cutoffs in Carnot coordinates and
dividing by their sum gives the stated partition and derivative
bound; the ball--box estimates give bounded overlap.
For each patch define
\[
 f_{p,i}(z)=\frac{\phi_i(\Phi_{p,R}z)-\phi_i(p)}R.
\]
By~\eqref{eq:equiregular-metric-error}, these functions satisfy
\[
 |f_{p,i}(z)-f_{p,i}(w)|
 \leq L\widehat d_p(z,w)+CR^{1/s}.
\]
On a fixed enlarged coordinate domain $\mathcal V_p$ containing
all the tangent balls used below, define the McShane infimum
\[
 F_{p,i}(z)=\inf_{w\in\mathcal V_p}
       \{f_{p,i}(w)+L\widehat d_p(z,w)\}.
\]
It is $L$-Lipschitz on $G_p$ and satisfies
$|F_{p,i}-f_{p,i}|\leq CR^{1/s}$ on the coordinate domain.
Multiply it by a cutoff equal to one on all the smaller patches
used below. These functions satisfy the hypotheses of
Lemma~\ref{lem:uniform-tangent-estimates}(b) uniformly.

\medskip
\noindent\textbf{Step 2: A common scale of affine approximation.}
Multiply~\eqref{eq:uniform-tangent-dorronsoro} for each patch by
$R^Q$ and sum. By bounded overlap and the coordinate density bounds,
there is one common $r$ in~\eqref{eq:equiregular-scales} such that
\begin{equation}\label{eq:equiregular-beta-sum}
 \sum_p\int_{S\cap\operatorname{supp}\chi_p}
 \bigl(\beta_{F_{p,0}}(z_p(x),r)^2+
       \beta_{F_{p,1}}(z_p(x),r)^2\bigr)\,\d\rho(x)
 \leq\frac C{\log(1/\tau)},
\end{equation}
where $z_p(x)=\Phi_{p,R}^{-1}(x)$.
For each center take a best affine approximation, with horizontal
slope $q_{p,i}(x)$. The slopes are uniformly bounded. The argument for
\eqref{eq:affine-uniform-from-L2}, with the uniform bounds on the measures of balls just proved, gives, on $B_{G_p}(z_p(x),4r)$,
\begin{equation}\label{eq:equiregular-affine-error}
 \frac{\|f_{p,i}-A_{p,i,x,r}\|_\infty}{r}
 \leq C\beta_{F_{p,i}}(z_p(x),r)^{2/(Q+2)}
       +C\frac{R^{1/s}}r=:\eta_{p,i}(x).
\end{equation}
All balls used here remain in the enlarged coordinate patch.
The quantities $\eta_{p,i}$ are uniformly bounded: the Lipschitz
bound controls $\beta_{F_{p,i}}$, and the interval of scales in \eqref{eq:equiregular-scales} gives
$R^{1/s}/r\leq C\tau^{1/(8s)}$. In particular, their squares
can be bounded by a fixed multiple of $\eta_{p,i}$ below.

\medskip
\noindent\textbf{Step 3: Comparison of horizontal gradients.}
The argument of Lemma~\ref{lem:common-approximate-contact} also gives
\begin{equation}\label{eq:equiregular-contact}
 |q_{p,0}(x)-q_{p,1}(x)|^2
 \leq C\left(H+R+\eta_{p,0}(x)+\eta_{p,1}(x)
                    +\frac{\gamma(x)}H\right).
\end{equation}
Choose $y\in Y$ attaining the infimum defining
$((\theta_0+\theta_1)/2)^c(x)$, and put
$e_i(x)=c(x,y)-\theta_i(y)-\phi_i(x)\geq0$.
These contact errors sum to $2\gamma(x)$. For a unit horizontal vector $v$,
the tangent point $z_p(x)\exp(rv)$ corresponds to a source point
at distance at most $H+CR^{1+1/s}$ from $x$, by
\eqref{eq:equiregular-metric-error}. Testing
$\phi_i(z)-\phi_i(x)\leq c(z,y)-c(x,y)+e_i(x)$ at that point gives
\[
 |q_{p,i}(x)|\leq\dcc(x,y)
 +C\left(H+\eta_{p,i}(x)+\frac{R^{1/s}}r+
                    \frac{e_i(x)}H\right).
\]
If $\dcc(x,y)\geq H$, the point at length $H$ of a minimizing
curve from $x$ to $y$, denoted by $x_H$, has the following rescaled displacement in the first layer:
\[
 a=\frac{\pi_1(z_p(x)^{-1}z_p(x_H))}{r},\qquad
 |a|\leq1+CR^{1/s}/r.
\]
Here products and $\pi_1$ are taken in $G_p$.
For sufficiently small $\tau_0$, the metric comparison places
$z_p(x_H)$ in $B_{G_p}(z_p(x),2r)$, where the affine error
bound applies. The exact decrease of the squared
distance along this curve yields
\[
 q_{p,i}(x)\cdot a\leq-\dcc(x,y)
 +C\left(H+\eta_{p,i}(x)+\frac{e_i(x)}H\right).
\]
Expand $|q_{p,i}(x)+\dcc(x,y)a|^2$, exactly as in
Lemma~\ref{lem:common-approximate-contact}. The additional error
$CR^{1/s}/r$ is already included in~\eqref{eq:equiregular-affine-error}.
If $\dcc(x,y)<H$, the first displayed bound and boundedness of the
slopes suffice. This proves~\eqref{eq:equiregular-contact}.
The bound is independent of the chosen target point $y$ and minimizing curve. The quantities to be integrated are Borel: the affine
coefficients are obtained by inverting their Gram matrices, and
the affine errors are integrals or suprema of continuous
functions over metric balls.

\medskip
\noindent\textbf{Step 4: Smoothing and the global Poincar\'e inequality.}
Use the smoothing operators from
Lemma~\ref{lem:uniform-tangent-estimates}(c). Define smooth
functions on the source neighborhood by
\[
 u_i(x)=\sum_p\chi_p(x)
       \bigl[\phi_i(p)+R(P_r^pF_{p,i})(z_p(x))\bigr].
\]
Then
\begin{equation}\label{eq:equiregular-smoothing-errors}
 \|u_i-\phi_i\|_\infty\leq C(H+R^{1+1/s}).
\end{equation}
The horizontal derivative estimate for convolution on $G_p$ and
\eqref{eq:equiregular-affine-error} approximate the horizontal
derivatives of each smoothed summand by $q_{p,i}$ with error
$C\eta_{p,i}$. There are two further errors. First,
\eqref{eq:equiregular-field-error} changes tangent derivatives to
manifold derivatives by at most $CRr^{1-s}$, by the coordinate
bound in Lemma~\ref{lem:uniform-tangent-estimates}(c).
The factor $R$ in each summand cancels the factor $R^{-1}$
introduced by rescaling the horizontal derivative. Norms of horizontal slopes
are computed in the orthonormal frame on that patch; changing
patches amounts to an orthogonal change of horizontal components.
Second, differentiating the
partition gives at most $C(r+R^{1/s})$. To obtain the latter bound,
subtract $\phi_i(x)$ in every differentiated partition term and use
$\sum_p\nabla_{\mathcal H}\chi_p=0$ together with
\eqref{eq:equiregular-smoothing-errors} on each patch.

The global $L^2$ Poincar\'e inequality on the John domain $S$ yields
\begin{align}
 \Var_\rho(\phi_1-\phi_0)
 &\leq C\Bigg[
 H+R+\frac JH+
 \sum_p\int_{S\cap\operatorname{supp}\chi_p}
      (\eta_{p,0}+\eta_{p,1})\,\d\rho\notag\\*
 &\hspace{2.5em}+(Rr^{1-s})^2+(r+R^{1/s})^2
       +(H+R^{1+1/s})^2\Bigg].
 \label{eq:equiregular-final-bound}
\end{align}
Here bounded overlap, the bounded affine slopes, and
\eqref{eq:equiregular-contact} were used to estimate the gradient
of $u_1-u_0$. The required Poincar\'e inequality follows from the
local Poincar\'e inequality for bracket-generating vector fields
\cite{Jerison1986} and the same John chains as above; neither
argument requires step two.

By~\eqref{eq:equiregular-beta-sum}, H\"older's inequality, and
\eqref{eq:equiregular-affine-error}, the sum of affine errors is at
most
\[
 C[\log(1/\tau)]^{-1/(Q+2)}+C R^{1/s}/r.
\]
Every remaining error in~\eqref{eq:equiregular-final-bound} is a
positive power of $\tau$. In particular,
\[
 \frac{R^{1/s}}r\leq C\tau^{1/(8s)},\qquad
 Rr^{1-s}\leq C\tau^{1/4-(s-1)/(8s)},\qquad
 \frac JH\leq C\tau^{3/4-1/(8s)}.
\]
We have proved
\begin{equation}\label{eq:equiregular-log-jensen}
 V_c(\theta_0,\theta_1)
 \leq C\left[\log\left(e+\frac{J_0}{J_c(\theta_0,\theta_1)}\right)
       \right]^{-1/(Q+2)}.
\end{equation}
If $J=0$, apply the same argument with an arbitrary
$0<\tau<\tau_0$ in place of $J/J_0$; the term $J/H$ is zero, and we then
let $\tau\downarrow0$. The bounded oscillation of source potential
differences gives the estimate for larger $J$. Proposition~\ref{prop:potential-conversion}
and~\eqref{eq:optimal-jensen-gap} give the logarithmic stability
estimate, with exponent $1/[2(Q+2)]$, and uniqueness of normalized
source potentials. Only a neighborhood of the source is used in the nilpotent
approximation, so equiregularity is required only there.
\end{proof}

\subsection{One finitely supported target}\label{sec:finite-targets} 
If one target has a bounded number of atoms, the exponent $1/2$
follows from Lipschitz bounds and an $L^1$ Poincar\'e inequality
even when the other target is
arbitrary. We assume that, for a fixed $P>0$, every function $u$
Lipschitz on a neighborhood of $S$ satisfies
\begin{equation}\label{eq:pairwise-L1-poincare}
	\int_S\int_S|u(x)-u(x')|\,\d\rho(x)\,\d\rho(x')
	\leq P\int_S|\nabla_{\mathcal H}u|\,\d\rho.
\end{equation}
Here $\nabla_{\mathcal H}u$ is the weak horizontal gradient:
$g_{\mathcal H}(\nabla_{\mathcal H}u,X)=Xu$
almost everywhere for each smooth horizontal field $X$.

\begin{proposition}[Stability from one finitely supported target]\label{prop:one-finite-target}
	Let $S$ and $\rho$ satisfy the standing source assumptions of
	this section. Let $U$ be a neighborhood of $\overline S$, and let
	$b\in C(U\times Y)$.
	Suppose that $\rho$ satisfies the $L^1$ Poincar\'e inequality
	\eqref{eq:pairwise-L1-poincare} with constant $P$, and that
	\[
	\begin{aligned}
		\Lip(b(\cdot,y);U)\leq L_x,\quad y\in Y,\qquad 
		\Lip(b(x,\cdot);Y)\leq L_y,\quad x\in S.
	\end{aligned}
	\] 
	For $i=0,1$, let $(\phi_i,\theta_i)$ be continuous
	$b$-conjugate Kantorovich potentials from $\rho$ to
	$\mu_i\in\PP(Y)$.
	If $\#\spt\mu_0\leq m$ for $m\in\N^+$, then
	\begin{equation}\label{eq:one-finite-target}
		\Var_\rho(\phi_1-\phi_0)
		\leq8mPL_xL_yW_1(\mu_0,\mu_1).
	\end{equation}
\end{proposition} 

\begin{proof} 
	Set $g=\phi_1-\phi_0$, $v=\theta_1-\theta_0$, and $ A=-v(\spt\mu_0)$. Then $A$ has at most $m$ points. Let $\pi_i$ be an optimal coupling
	from $\rho$ to $\mu_i$, $i=0,1$, and set
	$R_i=b-\phi_i-\theta_i$.
	For $\pi_0$-almost every $(x,y)$, $-v(y)\in A$ and
	$\dE(g(x),A)\leq -g(x)-v(y)=R_1(x,y)$.
	Since $R_0\geq0$ and $R_i=0$ $\pi_i$-almost everywhere, we obtain 
	\begin{equation}\label{eq:one-sided-residual}
		\begin{aligned}
			\int_S\dE(g,A)\,\d\rho
			\leq\int R_1\,\d\pi_0+\int R_0\,\d\pi_1=\int_Yv\,\d(\mu_1-\mu_0)
			\leq2L_yW_1(\mu_0,\mu_1).
		\end{aligned}
	\end{equation}
	The last inequality uses target conjugacy, which gives
	$\Lip(v;Y)\leq2L_y$, and Kantorovich--Rubinstein duality.
	
	By the source Lipschitz bound on $b$, both $\phi_0$ and $\phi_1$
	are $L_x$-Lipschitz on $U$. Hence $|\nabla_{\mathcal H}g|\leq2L_x$ almost everywhere. 
	Put $H_A(t)=\int_0^t\dE(r,A)\,\d r$, where
	$\dE(r,A)=\min_{a\in A}|r-a|$. To verify the scalar estimate
	$|s-t|^2\leq4m|H_A(s)-H_A(t)|$, assume $s<t$ and partition
	$[s,t]$ into intervals on each of which one fixed point of $A$
	is nearest. There are at most $m$ such intervals. On an interval of length $\ell_j$ assigned to
	$a_j\in A$, the integral of $|r-a_j|$ is at least
	$\ell_j^2/4$. Thus
	$H_A(t)-H_A(s)\geq\frac14\sum_j\ell_j^2
	\geq(t-s)^2/(4m)$ by Cauchy--Schwarz.
	The variance identity, \eqref{eq:pairwise-L1-poincare}, and the
	chain rule now give
	\begin{align*}
		\Var_\rho g
		&\leq2m\iint|H_A(g(x))-H_A(g(x'))|\,\d\rho(x)\,\d\rho(x')\\
		&\leq2mP\int_S\dE(g,A)|\nabla_{\mathcal H}g|\,\d\rho\\
		&\leq8mPL_xL_yW_1(\mu_0,\mu_1), 
	\end{align*}
	which proves the claim. 
\end{proof}

\begin{lemma}[A first-order estimate on John domains]\label{lem:john-L1}
The sources allowed in Theorem~\ref{thm:potential-stability} satisfy
\eqref{eq:pairwise-L1-poincare} with a fixed $P<\infty$.
\end{lemma}

\begin{proof}
	First take $u$ smooth. On a compact equiregular neighborhood,
	local doubling and the horizontal $L^1$ Poincar\'e inequality hold
	at all sufficiently small radii
	\cite{NagelSteinWainger1985,FranchiLuWheeden1995}.
	With the balls and their interior enlargements from
	Lemma~\ref{lem:john-gluing}, this gives
	\[
	e_B:=\int_B|u-u_{B,\mm}|\,\d\mm
	\leq Cr_B\int_{\kappa B}|\nabla_{\mathcal H}u|\,\d\mm.
	\]
	
	Let $B_i$ and $\mathcal C(i)$ be the Whitney balls and
	chains without repeated balls constructed in
	Lemma~\ref{lem:john-gluing}.
	Set $u_i=u_{B_i,\mm}$, and let $u_*$ be the mean on the
	central ball. By the overlap estimates,
	$|u_i-u_*|\leq
	C\sum_{j\in\mathcal C(i)}e_{B_j}/\mm(B_j)$.
	Choose a measurable partition $S=\bigsqcup_iE_i$ with
	$E_i\subseteq B_i$.
	Since $B_i\subseteq C B_j$ whenever $j\in\mathcal C(i)$,
	local doubling yields $\sum_{i:j\in\mathcal C(i)}\mm(E_i)\leq C\mm(B_j)$. Thus, by Fubini's theorem, 
	\[
	\begin{aligned}
		\int_S|u-u_*|\,\d\mm\leq\sum_i e_{B_i}+\sum_i\mm(E_i)|u_i-u_*|\leq C\sum_j e_{B_j}
		\leq C\int_S|\nabla_{\mathcal H}u|\,\d\mm, 
	\end{aligned}
	\]
   where the last step uses $r_B\leq r_*$ and bounded overlap
	of the enlarged balls.
	
	By the density bounds and the triangle inequality,
	\[
	\begin{aligned}
		\int_S\int_S|u(x)-u(x')|\,\d\rho(x)\,\d\rho(x')\leq2\int_S|u-u_*|\,\d\rho\leq C\frac{a_2}{a_1}
		\int_S|\nabla_{\mathcal H}u|\,\d\rho.
	\end{aligned}
	\]
	For compact $S=M$, the same argument uses the finite
	connected cover from Lemma~\ref{lem:john-gluing}.
	Finally, smooth approximation in local charts, with
	$L^1$ convergence of the functions and their weak horizontal
	gradients, extends the estimate to functions Lipschitz on a
	neighborhood of $S$.
\end{proof} 

\begin{proof}[Proof of Theorem~\ref{thm:potential-stability}\textup{(iv)}]
After exchanging the indices if necessary, assume that
$\#\spt\mu_0\leq m$. For every $\varepsilon>0$, the cost $c$ is $(D+\varepsilon)$-Lipschitz
in the source variable on a sufficiently small neighborhood of
$\overline S$, uniformly for $y\in Y$. By Lemma~\ref{lem:john-L1}, we apply
Proposition~\ref{prop:one-finite-target} with
$L_x=D+\varepsilon$ and $L_y=D$ and then let $\varepsilon\downarrow0$ to obtain 
\begin{equation}\label{eq:finite-target-bound}
	\norm{\phi_{\mu_1}-\phi_{\mu_0}}_{L^2(\rho)}^2
	\leq8mPD^2W_1(\mu_0,\mu_1), 
\end{equation}
which proves the claim. 
\end{proof} 

For fixed finite target sites, Poincar\'e inequalities are used
to bound the modulus of strong concavity of the semi-discrete
dual functional
\cite[Theorem~5.1]{KitagawaMerigotThibert2019} and stability estimates
for target dual weights \cite[Theorem~1.10]{BansilKitagawa2022}. Here we estimate source potentials in $L^2(\rho)$ in terms of
$W_1$ when only one target is finitely supported.
The proof combines the duality defect estimate with the
$L^1$-Poincar\'e inequality; the cost enters only through its
Lipschitz bounds. 

\begin{remark}\label{rem:cardinality}
	The $\sqrt m$ factor in Proposition~\ref{prop:one-finite-target}
	is optimal for general Lipschitz costs. Indeed, choose $S=(0,1)$, $\rho=\mathcal L^1\restr S$,
	$b(x,y)=|x-y|$, and $Y_N=\{j/N:0\leq j\leq N\}$.
	Fix $0<a<1$ and set $\theta_0=0$, $\theta_1(y)=ay$.
	For $i=0,1$, denote $\phi_i=\theta_i^b$, let $T_i^N(x)$
	minimize $b(x,y)-\theta_i(y)$ over $Y_N$, and set
	$\mu_i^N=(T_i^N)_\#\rho$.
	
	For $\theta_0$, the cell boundaries are $(j+1/2)/N$,
	$0\leq j<N$.
	For $\theta_1$, each boundary moves left by $a/(2N)$.
	Every cell has positive mass, so $(\phi_i,\theta_i)$ are
	mutual $b$-conjugate Kantorovich pairs. Since the two measures differ only at $0$ and $1$, we obtain 
	\[
	\mu_1^N-\mu_0^N=\frac{a}{2N}(\delta_1-\delta_0),
	\qquad W_1(\mu_0^N,\mu_1^N)=\frac{a}{2N}.
	\]
	Moreover, $\phi_1-\phi_0=-ax+O(N^{-1})$ uniformly. Hence 
	\[
	\Var_\rho(\phi_1-\phi_0)\longrightarrow\frac{a^2}{12}, \qquad  \frac{\Var_\rho(\phi_1-\phi_0)^{1/2}}
	{W_1^{1/2}(\mu_0^N,\mu_1^N)}\asymp\sqrt N. 
	\]
	Here $L_x=L_y=1$ and the Poincar\'e constant is independent
	of $N$. Since normalization leaves variance unchanged and
	$m=N+1$, the coefficient in a uniform $L^2$ estimate with
	$W_1$ exponent $1/2$ must grow at least as $\sqrt m$. The optimal dependence on $m$ for the squared
	Carnot--Carath\'eodory distance remains open. 
\end{remark} 

\subsection{The role of step two for potentials}\label{sec:step-two-potentials}
For potentials, step two enters at two distinct points of the proof.
First, on a complete two-generating structure the squared
distance is locally Lipschitz in local coordinates, including
near the diagonal \cite[Corollary~6.2]{AgrachevLee2009}.
For Carnot groups, the obstruction in higher step is already
visible at the diagonal: for
$0\ne Z\in V_j$, homogeneity and symmetry give
\begin{equation}\label{eq:layer-cost-scaling}
 c(e,\exp(tZ))=|t|^{2/j}c(e,\exp Z).
\end{equation}
Thus layers of weight greater than two prevent coordinate
Lipschitz regularity. In Proposition~\ref{prop:penalized-approximation},
this regularity gives the linear approximation error and the
uniform covector bound. The endpoint corrections then yield
every potential exponent below $1/2$.

Second, on a step-two Carnot group the stronger horizontal
semiconcavity estimate of
Proposition~\ref{prop:horizontal-semiconcavity} gives a linear
variance bound directly, and hence the potential exponent $1/2$.
This second property is not a consequence of two-generation alone.
The example below separates the two properties on a smooth
two-generating structure.

At higher step, the failure of horizontal semiconcavity on the Engel
group \cite[Proposition~4.5]{DragoniLiuZhang2025} rules out that
criterion based on directional second derivatives, but does not
rule out a H\"older estimate for potentials.

\subsection{Failure of horizontal semiconcavity}\label{sec:horizontal-obstruction}
Uniform horizontal semiconcavity may fail on smooth
two-generating manifolds, even away from the diagonal. The example below also shows that the positive parts of directional
second-difference quotients need not admit a uniform $L^1$ bound
at pairs $(x,T(x))$ for one fixed optimal transport map $T$.
These differences are taken with the target fixed, whereas
the source potential minimizes over the target set.
The example therefore shows that the directional hypothesis of
Theorem~\ref{thm:directional-jensen} can fail; it leaves open the
potential stability estimate with exponent $1/2$. 

On $\RR^4$ with coordinates $(x,y,z,s)$, put
\[
 f(s)=\begin{cases}0,&s\leq0,\\ e^{-1/s^2},&s>0,\end{cases}
 \qquad F(s)=\int_0^s f(r)\,\d r,
\]
and declare the fields
\begin{equation}\label{eq:curved-two-step-model}
 X=\partial_x,\qquad Y_0=\partial_y+x\partial_z,
 \qquad W=\partial_s+f(s)\partial_x
\end{equation}
orthonormal. The distribution is that of $\HH^1\times\RR$, since
it is spanned by $X,Y_0,\partial_s$. Moreover,
$[X,Y_0]=\partial_z$, so $(\dim\mathcal H,\dim(\mathcal H+[\mathcal H,\mathcal H]))
=(3,4)$ everywhere.
It is not fat: the Levi bracket of $\partial_s$ with the
horizontal space vanishes. Since $0\leq f\leq1$, the coefficient
map $(a,b,c)\mapsto(a+f(s)c,b,c)$ and its inverse are uniformly
bounded. Hence the metric is uniformly comparable to the
complete product metric and is complete.

\Needspace{6\baselineskip}
\begin{proposition}[Failure of directional semiconcavity]
\label{prop:horizontal-obstruction}
For the cost $c=\dcc^2/2$ of \eqref{eq:curved-two-step-model}:
\begin{enumerate}[(i)]
\item There exists a fixed point $p$ such that $c(p,\cdot)$ has no
uniform semiconcavity constant along $Y_0$-flows on a fixed compact
region away from $p$.
\item There exist a fixed John domain $S$, its normalized Lebesgue measure $\rho$, a
fixed compact target set $Y$, and a map $T:S\to Y$ that is optimal from $\rho$ to $T_\#\rho$ with the following property. For 
\[
 K_h(q)=\left[
 \frac{c(\Phi_{Y_0}^h q,T(q))+c(\Phi_{Y_0}^{-h}q,T(q))
                      -2c(q,T(q))}{h^2}\right]_+,
 \qquad [a]_+=\max\{a,0\},
\]
one has
\begin{equation}\label{eq:nonintegrable-second-differences}
 \liminf_{h\to0}\int_S K_h\,\d\rho=+\infty.
\end{equation}
The associated source Kantorovich potential is nevertheless
constant along every $Y_0$-flow segment contained in $S$.
\end{enumerate}
\end{proposition}

\begin{proof}
	\emph{Part (i).}
	For $0<e<1$, put
	\[
	p=(0,0,0,-1),\quad \ell=1+e,\quad A=F(e),\quad
	q_e(h)=(A,h,Ah,e),\quad q_e=q_e(0).
	\]
	Set $u^0=(0,0,\ell)$, $a_e(t)=F(-1+\ell t)$, and
	$m_j=\int_0^1a_e(t)^j\,\d t$, $j=1,2$.
	The control $u^0$ reaches $q_e$. By $\dot s=u_3$, every control from $p$ to $q_e(h)$ satisfies
	$\int_0^1u_3\,\d t=1+e=\ell$. Hence 
	\begin{equation}\label{eq:model-energy-excess}
		\calA(u)-\frac{\ell^2}{2}
		=\frac12\norm{u-u^0}_{L^2}^2,
		\qquad \calA(u)=\frac12\norm{u}_{L^2}^2.
	\end{equation}
	
	Since $x(t)=\int_0^t u_1(r)\,\d r+F(s(t))$,
	linearizing the endpoint equations at $u^0$ yields 
	\[
	\int_0^1v_1=\int_0^1v_3=0,\qquad
	\int_0^1v_2=1,\qquad
	\int_0^1a_ev_2=A.
	\]
	As $a_e$ is nonconstant, the least squared norm under these
	constraints is
	\[
	\kappa_e
	=(1,A)\begin{pmatrix}1&m_1\\m_1&m_2\end{pmatrix}^{-1}
	\binom1A
	=1+\frac{(A-m_1)^2}{m_2-m_1^2},
	\]
	attained with $v_1=v_3=0$. Let $v_2$ attain the minimum above. The control $(0,hv_2,\ell)$
	joins $p$ to $q_e(h)$ exactly. Hence  
	\[
	c(p,q_e(h))\leq\frac{\ell^2}{2}+\frac{\kappa_e}{2}h^2.
	\]
	
	Conversely, for fixed $e>0$, minimizing controls $u_h$ satisfy
	$\norm{u_h-u^0}_{L^2}\leq\sqrt{\kappa_e}|h|$.
	By differentiability of the endpoint map, every weak limit of
	$(u_h-u^0)/h$ satisfies the displayed constraints.
	Equation~\eqref{eq:model-energy-excess} and weak lower
	semicontinuity yield the matching lower bound. Therefore, 
	\begin{equation}\label{eq:exact-directional-coefficient}
		c(p,q_e(h))
		=\frac{\ell^2}{2}+\frac{\kappa_e}{2}h^2+o(h^2)
		\qquad\text{for each fixed }e>0.
	\end{equation}
	Since $m_j=\ell^{-1}\int_0^eF(s)^j\,\d s$, 
	\[
	A\sim\frac12e^3e^{-1/e^2},\qquad
	m_1\sim\frac{e^3A}{2\ell},\qquad
	m_2\sim\frac{e^3A^2}{4\ell}, 
	\]
	which implies $\kappa_e\sim4\ell/e^3\to\infty$.
	Since $q_e\to0$, $\dcc(p,q_e)=\ell\geq1$, and
	$q_e(h)=\Phi_{Y_0}^h q_e$, we obtain \textup{(i)}.
	
	\emph{Part (ii).}
	Let $S$ be a sufficiently small Carnot--Carath\'eodory ball
	centered at the origin, of radius less than $1/4$, and let
	$\rho$ be its normalized Lebesgue measure.
	Minimizing paths to the center show that $S$ is John. Denote 
	\[
	T(x,y,z,s)=(x-F(s),y,z,-1),\qquad Y=T(\overline S).
	\]
	Every horizontal curve from the plane $s=-1$ to
	$q=(x,y,z,s)$ has length at least $1+s$, attained by the
	$W$-flow from $T(q)$.
	Thus the source transform for $\psi=0$ is $	\phi(q)=(1+s)^2/2$. 
	The contact identity proves that $T$ is optimal from $\rho$
	to $T_\#\rho$. Note that the maps 
	\[
	L_{a,b,d}(x,y,z,s)=(a+x,b+y,d+z+ay,s)
	\]
	preserve $X,Y_0,W$, hence are isometries and preserve the
	$Y_0$-flow.
	For $q=(x,y,z,s)$ with $s>0$, the map $L_{x-F(s),y,z}$
	takes $(p,q_s)$ to $(T(q),q)$.
	By symmetry of the cost and
	\eqref{eq:exact-directional-coefficient}, 
	\[
	\lim_{h\to0}K_h(q)=\kappa_s\sim4s^{-3},
	\qquad s\downarrow0.
	\]
	Since $S$ is an open neighborhood of the origin in the manifold topology,
	it contains $B_0\times(0,\delta)$ for a three-dimensional
	box $B_0$ of positive volume. Hence 
	\[
	\liminf_{h\to0}\int_S K_h\,\d\rho
	\geq\int_{B_0\times(0,\delta)}\kappa_s\,\d\rho
	=+\infty.
	\]
	Finally, $Y_0s=0$ implies that $\phi$ is constant along
	every $Y_0$-flow segment contained in $S$.
\end{proof}

\section{Sharpness and geometric obstructions}\label{sec:sharpness}
We give two constructions. Perturbations of target measures supported on three points show that the map and potential exponents cannot exceed $1/4$
and $1/2$, respectively. Families with smooth source and target
densities give a further restriction on the map exponent through
the anisotropic scaling of the Carnot--Carath\'eodory distance.

\medskip
Letrouit constructs unstable transport maps by perturbing discrete
transport configurations
\cite[Theorem~2 and Section~3.1]{LetrouitUnstable2026}.
We also perturb finitely supported targets.
For a finitely supported target, target locations and dual weights
determine the source regions assigned to each atom, called
Laguerre cells, and hence their masses; see
\cite{KitagawaMerigotThibert2019,deGournayKahnLebrat2019,BansilKitagawa2022}.
We use three target atoms in a region where the cost is smooth.
We vary their cell masses while keeping the source domain and
density fixed. This gives lower bounds for changes in both optimal
transport maps and Kantorovich potentials. The limiting transport
problem also has a unique optimal plan.

Two nearby target points may correspond to source cells separated by another cell. In our construction, the targets $A_\varepsilon,B_\varepsilon$ are at distance $\varepsilon$, while their source cells are separated by the cell
of a fixed third target $C$. We transfer target mass of order $\varepsilon$ over distance
$\varepsilon$. The maps then change across the interfaces
$A_\varepsilon$--$C$ and $C$--$B_\varepsilon$; see
Figure~\ref{fig:three-cell-construction}. 

Under the hypotheses of Theorem~\ref{thm:simultaneous-sharpness},
fix a nonempty open set $O\subseteq M$. Throughout this section,
$\dcc$ is the original ambient distance, $c=\dcc^2/2$, and constants
may depend on the fixed local geometry and $\omega$, but not on
the perturbation parameter. 

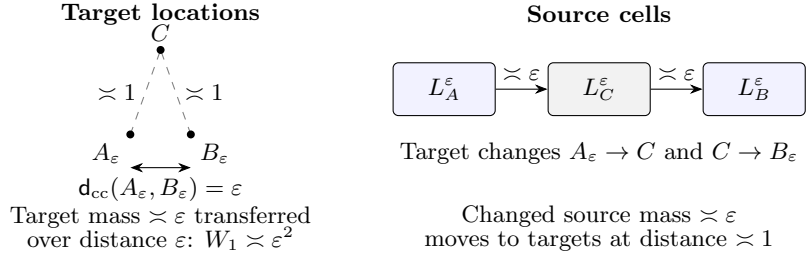
\begin{figure}[htbp]
\centering
\begin{tikzpicture}[x=1cm,y=.8cm,>=Stealth,font=\small]
\node[font=\small\bfseries] at (1.7,2.65) {Target locations};
\node[font=\small\bfseries] at (7.5,2.65) {Source cells};
\coordinate (A) at (1.3,.65);
\coordinate (B) at (2.1,.65);
\coordinate (C) at (1.7,2.05);
\draw[dashed,gray] (A)--node[left,black] {$\asymp1$}(C);
\draw[dashed,gray] (B)--node[right,black] {$\asymp1$}(C);
\fill (A) circle (1.5pt) node[below left] {$A_\varepsilon$};
\fill (B) circle (1.5pt) node[below right] {$B_\varepsilon$};
\fill (C) circle (1.5pt) node[above] {$C$};
\draw[<->] (1.3,.1)--node[below] {$\dcc(A_\varepsilon,B_\varepsilon)=\varepsilon$}(2.1,.1);
\node[align=center] at (1.7,-.9) {Target mass $\asymp\varepsilon$ transferred\\[-2pt]
  over distance $\varepsilon$: $W_1\asymp\varepsilon^2$};
\node[draw,rounded corners=2pt,fill=blue!5,minimum width=1.35cm,minimum height=.7cm] (LA) at (5.45,1.4) {$L_A^\varepsilon$};
\node[draw,rounded corners=2pt,fill=gray!10,minimum width=1.35cm,minimum height=.7cm] (LC) at (7.5,1.4) {$L_C^\varepsilon$};
\node[draw,rounded corners=2pt,fill=blue!5,minimum width=1.35cm,minimum height=.7cm] (LB) at (9.55,1.4) {$L_B^\varepsilon$};
\draw[->] (LA)--node[above] {$\asymp\varepsilon$}(LC);
\draw[->] (LC)--node[above] {$\asymp\varepsilon$}(LB);
\node[align=center] at (7.5,.35) {Target changes $A_\varepsilon\to C$ and $C\to B_\varepsilon$};
\node[align=center] at (7.5,-.9) {Changed source mass $\asymp\varepsilon$\\[-2pt]
  moves to targets at distance $\asymp1$};
\end{tikzpicture}
\caption{Target locations and source cells (schematic, not to scale). 
Each arrow in the right panel represents a change of target
on a source set of mass $\asymp\varepsilon$.
The panel shows which cells exchange source mass. Their geometric
arrangement in $M$ or in the $(p,q)$-plane is not depicted.} \label{fig:three-cell-construction}
\end{figure} 

\subsection{Localization}
\begin{lemma}[Completion preserving local distances]
\label{lem:local-completion}
For $x_*\in M$ and a neighborhood $U$ of $x_*$, there exist a
neighborhood $V\Subset U$ of $x_*$ and a smooth horizontal metric
$\widehat g_{\mathcal H}\geq g_{\mathcal H}$ on the same distribution
such that $\widehat g_{\mathcal H}=g_{\mathcal H}$ near $\overline V$,
$\widehat g_{\mathcal H}$ is the restriction of a complete Riemannian
metric, and its Carnot--Carath\'eodory distance satisfies
\begin{equation}\label{eq:local-distance-equality}
 \widehat\dcc\geq\dcc\quad\text{on }M\times M,
 \qquad \widehat\dcc=\dcc\quad\text{on }V\times V.
\end{equation}
Minimizers between points of $V$ exist for both metrics and coincide.
\end{lemma}

\begin{proof}
	Extend $g_{\mathcal H}$ to a smooth Riemannian metric $g_0$ 
	and choose $x_*\in U_1\Subset U_2\Subset U$.
	Take a complete Riemannian metric $h$, rescaled so that
	$h\leq g_0$ near $\overline U_2$. Choose a smooth function
	$F\geq1$, equal to one on $U_1$, such that $F^2g_0\geq h$,
	and set
	\[
	\widehat g_0=F^2g_0,\qquad
	\widehat g_{\mathcal H}=\widehat g_0|_{\mathcal H}.
	\]
	Since $\widehat g_0\geq h$, $\widehat g_0$ is complete.
	The associated distance $\widehat\dcc$ induces the manifold
	topology, and its closed balls are closed subsets of compact
	$h$-balls. Hence it is proper and admits minimizing curves.
	
	Choose $r>0$ such that
	$\overline{B_{\mathrm{cc}}(x_*,4r)}\Subset U_1$, and put
	$V=B_{\mathrm{cc}}(x_*,r)$. Since $F\geq1$,
	$\widehat\dcc\geq\dcc$ globally.
	For $x,y\in V$, we have $\dcc(x,y)<2r$.
	Every horizontal curve from $x$ to $y$ with original length
	less than $3r$ stays in $B_{\mathrm{cc}}(x_*,4r)$,
	where the two metrics agree. Applying this to curves whose
	lengths approach $\dcc(x,y)$ yields 
	$\widehat\dcc(x,y)\leq\dcc(x,y)$, proving
	\eqref{eq:local-distance-equality}.
	Every minimizer for either metric has original length less
	than $2r$ and stays in $U_1$, where its two lengths agree.
	Thus the two sets of minimizers coincide.
\end{proof} 

\subsection{A construction with three Laguerre cells}
\begin{lemma}\label{lem:sharp-open-chain}
	There exist $x_*\in O$, a John domain $S\Subset O$, a compact set
	$Y\subseteq O$, a fixed point $C\in Y$, and constants
	$a,c_0,k_1,k_2>0$. Set $\rho=\omega\restr S/\omega(S)$.
	For all sufficiently small $\varepsilon>0$, there exist
	$A_\varepsilon,B_\varepsilon\in Y$ with the following properties.
	
	For $y\in Y$, put $h_y(x)=c(x,y)-c(x_*,y)$, and let
	$L_A^\varepsilon(m_A,m_B)$, $L_B^\varepsilon(m_A,m_B)$, and
	$L_C^\varepsilon(m_A,m_B)$ be the Laguerre cells in $S$ associated
	with $h_{A_\varepsilon}-m_A,h_{B_\varepsilon}-m_B,h_C$.
	Each cell consists of the points where the corresponding function
	attains the minimum of these three functions.
	The interface of two cells is the set where their corresponding
	functions both attain this minimum.
	
	\begin{enumerate}
		\item[\textup{(i)}] The target points satisfy
		\begin{equation}\label{eq:sharp-target-distances}
			\dcc(A_\varepsilon,B_\varepsilon)=\varepsilon,
			\qquad
			\min\{\dcc(A_\varepsilon,C),\dcc(B_\varepsilon,C)\}\geq c_0,
		\end{equation}
		and $c$ is smooth near $\overline S\times Y$.
		
		\item[\textup{(ii)}] The cells with $m_A=m_B=0$ satisfy
		\[
		\rho(L_i^\varepsilon(0,0))\geq c_0,
		\qquad i\in\{A,B,C\}.
		\]
		
		\item[\textup{(iii)}] For $|m_A|,|m_B|\leq a\varepsilon$,
		the interfaces are smooth hypersurfaces, there is no interface between the cells of $A_\varepsilon$ and $B_\varepsilon$, and
		\begin{equation}\label{eq:sharp-no-cross}
			L_A^\varepsilon(m_A,m_B)=L_A^\varepsilon(m_A,0),
			\qquad
			L_B^\varepsilon(m_A,m_B)=L_B^\varepsilon(0,m_B).
		\end{equation}
		
		\item[\textup{(iv)}] For $-a\varepsilon\leq u<v\leq a\varepsilon$,
		\begin{equation}\label{eq:sharp-cell-mass}
			\begin{aligned}
				k_1(v-u)
				&\leq \rho(L_A^\varepsilon(v,0))
				-\rho(L_A^\varepsilon(u,0))
				\leq k_2(v-u),\\
				k_1(v-u)
				&\leq \rho(L_B^\varepsilon(0,v))
				-\rho(L_B^\varepsilon(0,u))
				\leq k_2(v-u). 
			\end{aligned}
		\end{equation} 
	\end{enumerate}
\end{lemma} 

\begin{proof}
\textbf{Step 1: Target points and comparison functions.}
Choose $x_*\in O$ and apply Lemma~\ref{lem:local-completion} with a
coordinate neighborhood $U\Subset O$. Apply Agrachev's theorem to the completed structure
\cite[Theorem~1, expanded version]{AgrachevSmooth2009}. It gives
$y_0\in V\setminus\{x_*\}$ and $\lambda_0\in T_{x_*}^*M$ such that
$y_0=\Exp_{x_*}(\lambda_0)$. The differential
$D_\lambda\Exp_{x_*}(\lambda_0)$ is invertible, and the corresponding
geodesic is uniquely minimizing.
Here $\Exp$ is the time-one normal exponential map for the completed
structure. By \cite[Theorem~* and Lemma~2, expanded version]{AgrachevSmooth2009},
the cost is jointly smooth near this pair. The branch remains minimizing under
small perturbations of the covectors and endpoints.
Choose product neighborhoods $U_x\times U_y$ of $(x_*,y_0)$,
with $\overline U_x\times\overline U_y\subseteq V\times V$, on which
this branch and its cost are smooth. By
\eqref{eq:local-distance-equality}, this is also a smooth minimizing
branch for the original cost.

Write $\ell=\lambda_0|_{\mathcal H_{x_*}}$ and
$L=|\ell|_{g_{\mathcal H}^*}>0$, using the dual horizontal norm.
Quadratic homogeneity of the Hamiltonian gives the unit-speed curve
\[
 y_s=\Exp_{x_*}\bigl((1+s/L)\lambda_0\bigr),\qquad |s|\leq s_0,
\]
for small $s_0<L$, with its image compactly contained in $U_y$.
The geodesic remains minimizing up to $y_{s_0}$.
Its subsegments are therefore minimizing for the original ambient
distance. Thus $\dcc(y_s,y_t)=|s-t|$.
Since the horizontal rank is at least two, choose
$\zeta\in T_{x_*}^*M$ with nonzero horizontal restriction
orthogonal to $\ell$ in the dual horizontal metric, and fix a small $\tau>0$.
Choose $\tau$ small enough that
$C=\Exp_{x_*}(\lambda_0+\tau\zeta)\in U_y$ lies in the same inverse
branch. The covector $\lambda_0+\tau\zeta$ is not a scalar multiple
of $\lambda_0$. Injectivity of this branch then yields 
$C\ne y_s$ for every $s$. Set
\[
 Y:=\{y_s:|s|\leq s_0\}\cup\{C\}\Subset U_y,
 \qquad x_*\in U_*\Subset U_x.
\]
Then $\overline U_*\times Y$ is a compact subset of the smooth
product neighborhood $U_x\times U_y$. In particular, the original
cost is smooth on one neighborhood of the entire product
$\overline U_*\times Y$.
Put $A_\varepsilon=y_{-\varepsilon/2}$,
$B_\varepsilon=y_{\varepsilon/2}$, and
\[
 p(x)=\left.\frac{\der}{\der s}\right|_{s=0}h_{y_s}(x),
 \qquad q(x)=h_{y_0}(x)-h_C(x).
\]
Write $\der_{\mathcal H}f(x)=\der f(x)|_{\mathcal H_x}$ for the
horizontal differential of a smooth function $f$.
By the first variation, $\der_xc(x_*,y_s)=-(1+s/L)\lambda_0$ and
$\der_xc(x_*,C)=-(\lambda_0+\tau\zeta)$. Consequently,
\begin{equation}\label{eq:sharp-horizontal-transversality}
 \der_{\mathcal H}p(x_*)=-\ell/L,
 \qquad \der_{\mathcal H}q(x_*)=\tau\zeta|_{\mathcal H_{x_*}}.
\end{equation}

By the independence of these horizontal differentials at $x_*$,
we can shrink $U_*$ and choose smooth horizontal fields
$V_p,V_q$ such that 
\begin{equation}\label{eq:sharp-coordinate-fields}
 V_pp=V_qq=1,\qquad V_pq=V_qp=0.
\end{equation}
Since $p(x_*)=q(x_*)=0$, Taylor's formula yields 
\begin{align}
 h_{A_\varepsilon}=h_{y_0}-\tfrac\varepsilon2p+O_{C^1(U_*)}(\varepsilon^2),\qquad
 h_{B_\varepsilon}=h_{y_0}+\tfrac\varepsilon2p+O_{C^1(U_*)}(\varepsilon^2),
 \qquad h_C=h_{y_0}-q, 
 \label{eq:sharp-branch-expansion}
\end{align}
where the remainders have $C^1(U_*)$ norm $O(\varepsilon^2)$.
Since $A_\varepsilon,B_\varepsilon\to y_0\ne C$,
the lower bounds in \eqref{eq:sharp-target-distances} hold
after decreasing $c_0$ and taking $\varepsilon$ small enough.

\medskip
\textbf{Step 2: A fixed John domain.}
Write $\Phi_t^V$ for the local flow of a vector field $V$ at time $t$.
For a small fixed $\eta>0$, start at
$x_0=\Phi_\eta^{V_p}\Phi_{-\eta}^{V_q}(x_*)$ and concatenate the
flows of $V_q,-V_p,-V_q$, each for time $2\eta$.
By \eqref{eq:sharp-coordinate-fields}, the resulting horizontal
curve in $U_*$ traces, in the $(p,q)$ coordinates,
\[
 (\eta,-\eta)\longrightarrow(\eta,\eta)
 \longrightarrow(-\eta,\eta)\longrightarrow(-\eta,-\eta).
\]
Choose a small fixed $R>0$ and consecutive centers $z_1,\ldots,z_N$
along the curve so that the $R/2$ balls cover it,
$\dcc(z_j,z_{j+1})\leq R/2$, and the closed $R$ balls are compactly
contained in $U_*$. Make $R$ small enough that balls centered on the
right, upper, and left segments lie respectively in
$\{p\geq\eta/2\}$, $\{q\geq\eta/2\}$, and $\{p\leq-\eta/2\}$.
Set $S=\bigcup_{j=1}^NB_{\mathrm{cc}}(z_j,R)$ and
$\rho=\omega\restr S/\omega(S)$.
In particular,
\[
 S\subseteq\{p\geq\eta/2\}\cup\{q\geq\eta/2\}
             \cup\{p\leq-\eta/2\}.
\]

We claim that \(S\) is a John domain. Indeed, fix \(z_1\) as the
distinguished center.  For \(x\in B_{\mathrm{cc}}(z_j,R)\), join $x$ to $z_j$ and then through
$z_{j-1},\ldots,z_1$ using minimizers for $\widehat\dcc$.
Lemma~\ref{lem:local-completion} makes these minimizers for $\dcc$.
Each segment stays in the ball about its terminal or initial center,
so the path lies in $S$. Parametrized by original arc length, it obeys
\[
 \dcc(\gamma(s),M\setminus S)\geq s
 \quad\text{on the first segment},\qquad
 \dcc(\gamma(s),M\setminus S)\geq R/2
 \quad\text{on subsequent segments}.
\]
Since its total length is at most $R+R(N-1)/2$, this is a John path
with parameter $1/(N+1)$ for the ambient distance. Since $S\Subset U_*$, the cost is smooth near
$\overline S\times Y$ by Step~1.

\medskip
\textbf{Step 3: Cell comparisons and mass estimates.}
By \eqref{eq:sharp-branch-expansion}, for sufficiently small $\varepsilon$, 
\begin{equation*}
		h_{A_\varepsilon}-h_{B_\varepsilon}
	\leq-\eta\varepsilon/4, 
	\quad\text{on }S\cap\{p\geq\eta/2\},\qquad h_{B_\varepsilon}-h_{A_\varepsilon}\leq-\eta\varepsilon/4, \quad \text{on }S\cap\{p\leq-\eta/2\}, 
\end{equation*}
and 
\begin{equation*}
	h_C\leq\min\{h_{A_\varepsilon},h_{B_\varepsilon}\}-\eta/4, \quad \text{on } S\cap\{q\geq\eta/2\}. 
\end{equation*} 
Thus $h_{A_\varepsilon}$ and $h_{B_\varepsilon}$ never attain
the minimum together.
Choose fixed balls $K_A,K_B$ about the two lower endpoints and
$K_C$ about an interior upper point, with
\begin{align*}
 K_A\Subset S\cap\{p>\eta/2,\ q<-\eta/2\},\quad 
 K_B\Subset S\cap\{p<-\eta/2,\ q<-\eta/2\},\quad 
 K_C\Subset S\cap\{q>\eta/2\}.
\end{align*}
By the expansion and these comparisons, $K_i$ lies in the
cell $L_i^\varepsilon(0,0)$ for small $\varepsilon$. Hence
\[
 \rho(L_i^\varepsilon(0,0))\geq\rho(K_i)>0,
 \qquad i\in\{A,B,C\}.
\]
Decreasing $c_0$ if necessary proves \textup{(ii)} while preserving
\textup{(i)}.

Choose $a=\eta/16$. For $|m_A|,|m_B|\leq a\varepsilon$,
the comparison on $S\cap\{p\geq\eta/2\}$ gives 
\[
 (h_{A_\varepsilon}-m_A)-(h_{B_\varepsilon}-m_B)
 \leq-\eta\varepsilon/4+2a\varepsilon
 =-\eta\varepsilon/8.
\]
The analogous comparison holds on $S\cap\{p\leq-\eta/2\}$.
On $S\cap\{q\geq\eta/2\}$, $h_C$ remains strictly below both
$h_{A_\varepsilon}-m_A$ and $h_{B_\varepsilon}-m_B$ for small
$\varepsilon$. These three subsets cover $S$, so the comparisons
prove \eqref{eq:sharp-no-cross} and show that the cells of
$A_\varepsilon$ and $B_\varepsilon$ have no common interface.

Put $F_{A,\varepsilon}=h_{A_\varepsilon}-h_C$ and
$F_{B,\varepsilon}=h_{B_\varepsilon}-h_C$.
On a fixed neighborhood of $\overline S$,
\begin{equation}\label{eq:sharp-A-transversality}
 \tfrac12\leq V_qF_{i,\varepsilon}\leq2,\qquad i=A,B.
\end{equation}
By \eqref{eq:sharp-A-transversality}, the level sets
$F_{A,\varepsilon}=m_A$ and $F_{B,\varepsilon}=m_B$
are smooth hypersurfaces. Since the cells of $A_\varepsilon$ and $B_\varepsilon$ have no
common interface, the interfaces are relatively open subsets
of these hypersurfaces. Hence they are smooth and have zero
$\omega$-measure, proving \textup{(iii)}. Note that 
\[
\rho(L_A^\varepsilon(v,0))-\rho(L_A^\varepsilon(u,0))= \rho\bigl(S\cap\{p\geq\eta/2\}\cap
                 \{u<F_{A,\varepsilon}\leq v\}\bigr).
\]
In a $V_q$ flow box write
$\d\omega=J(t,\xi)\,\d t\,\d\xi$. A finite cover of
$\overline S$ by relatively compact flow boxes has uniformly bounded
positive Jacobians and finite transverse coordinate measures.
By \eqref{eq:sharp-A-transversality}, each flow line spends at most
$2(v-u)$ in the slab, giving the upper mass bound. 

For the lower bound, choose a fixed $V_q$-flow box around the
point of the constructed curve with $(p,q)=(\eta,0)$, compactly contained in
$S\cap\{p>\eta/2\}$. Let $\Sigma\subseteq\{q=0\}$ be its
transverse section, with positive coordinate measure, and fix
endpoint times $\pm t_*$. By the branch expansion,
\[
F_{A,\varepsilon}(\Phi_{\pm t_*}^{V_q}(\xi))
=\pm t_*-\tfrac\varepsilon2p(\xi)+O(\varepsilon^2),
\qquad \xi\in\Sigma.
\]
For small $\varepsilon$, these endpoint values lie below
$-a\varepsilon$ and above $a\varepsilon$, uniformly in $\xi$.
Thus every flow line crosses both levels $u$ and $v$.
Let $t_u(\xi)<t_v(\xi)$ be the corresponding crossing times.
By \eqref{eq:sharp-A-transversality}, $v-u\leq2\bigl(t_v(\xi)-t_u(\xi)\bigr)$. Integrating over $\Sigma$ yields 
\[
\begin{aligned}
	\rho(L_A^\varepsilon(v,0))-\rho(L_A^\varepsilon(u,0))\geq\frac1{\omega(S)}
	\int_\Sigma\int_{t_u(\xi)}^{t_v(\xi)}
	J(t,\xi)\,\d t\,\d\xi\geq k_1(v-u),
\end{aligned}
\]
where $k_1>0$ follows from the positive lower bound for $J$
and the positive coordinate measure of $\Sigma$.
The same argument near the curve point $(p,q)=(-\eta,0)$ gives
the estimate for the $B$ cell,
proving \textup{(iv)}. 
\end{proof}

\subsection{Proof of Theorem~\ref{thm:simultaneous-sharpness}}
\begin{proof}[Proof of Theorem~\ref{thm:simultaneous-sharpness}]
\textbf{Step 1: Perturbed cell masses and target distance.}
Use the source, targets, and constants from
Lemma~\ref{lem:sharp-open-chain}. Fix $b=k_1a/4$.
By \eqref{eq:sharp-cell-mass}, the functions giving the masses of the $A$ and $B$ cells are
continuous and strictly increasing. Choose
$-a\varepsilon<m_A<0<m_B<a\varepsilon$ so that the $A$ mass decreases
by $b\varepsilon$ and the $B$ mass increases by $b\varepsilon$.
By \eqref{eq:sharp-no-cross}, the same changes occur when both weights are applied,
and the $C$ mass is unchanged. Moreover, 
\begin{equation}\label{eq:sharp-weight-bounds}
 b\varepsilon/k_2\leq -m_A,m_B\leq b\varepsilon/k_1.
\end{equation}
Define $\mu_\varepsilon$ by the cell masses at $(0,0)$ and
$\nu_\varepsilon$ by the cell masses at $(m_A,m_B)$, both on
$\{A_\varepsilon,B_\varepsilon,C\}$. All masses are positive for
small $\varepsilon$, and
\[
 \nu_\varepsilon-\mu_\varepsilon
 =b\varepsilon(\delta_{B_\varepsilon}-\delta_{A_\varepsilon}).
\]
Transporting this mass directly from $A_\varepsilon$ to
$B_\varepsilon$ yields $W_1\leq b\varepsilon^2$.
The $1$-Lipschitz test function $y\mapsto\dcc(A_\varepsilon,y)$
gives equality, proving \eqref{eq:universal-W1}.
For every $1\leq p<\infty$ and sufficiently small $\varepsilon$,
\[
 W_p^p(\mu_\varepsilon,\nu_\varepsilon)=b\varepsilon^{p+1}.
\]
Indeed, any coupling sends at least $b\varepsilon$ mass out of
$A_\varepsilon$. Its other possible destinations are
$B_\varepsilon$ and $C$, both at distance at least $\varepsilon$
for small $\varepsilon$, by \eqref{eq:sharp-target-distances}.
This gives the lower bound $b\varepsilon^{p+1}$, attained by
leaving the common masses fixed and transporting $b\varepsilon$
directly from $A_\varepsilon$ to $B_\varepsilon$.

\medskip
\textbf{Step 2: Optimality and uniqueness.}
Set
\begin{align*}
 \tilde\phi_0=\min\{h_{A_\varepsilon},h_{B_\varepsilon},h_C\},\qquad 
 \tilde\phi_1=\min\{h_{A_\varepsilon}-m_A,
                         h_{B_\varepsilon}-m_B,h_C\}.
\end{align*}
Let $T_0,T_1$ select the minimizing atoms in the two minima,
with ties resolved in a fixed order. Since the interfaces
have zero $\rho$-measure, the minimizing atom is unique for
$\rho$-almost every source point. By construction, $(T_0)_\#\rho=\mu_\varepsilon$ and $(T_1)_\#\rho=\nu_\varepsilon$. 
On the three atoms, set $\psi_0(y)=c(x_*,y)$ and define
$\psi_1$ by adding $m_A,m_B,0$ at
$A_\varepsilon,B_\varepsilon,C$, respectively.
The pairs $(\tilde\phi_i,\psi_i)$ are dual feasible,
with equality on the graphs of $T_i$. Hence both graph
couplings are optimal.

For either problem, every optimal coupling $\pi$ satisfies
\[
\int\bigl(c(x,y)-\tilde\phi_i(x)-\psi_i(y)\bigr)
\,\d\pi(x,y)=0.
\]
The integrand is nonnegative, so $\pi$ is concentrated on
minimizing atoms. Their almost-everywhere uniqueness forces
$\pi=(\Id,T_i)_\#\rho$, proving uniqueness of the optimal coupling. For $i=0,1$, extend $\psi_i$ to $Y$ by
\[
\psi_i(y)=\inf_{x\in S}
\{c(x,y)-\tilde\phi_i(x)\}.
\]
Each atom has a nonempty cell, so its assigned value is unchanged,
and $\psi_i^c=\tilde\phi_i$ on $S$.

\medskip
For either target, let $(\phi,\psi)$ and $(\phi',\psi')$ be
continuous Kantorovich pairs. Equality along the optimal transport map $T$ yields 
\[
\phi(x)-\phi'(x)=\psi'(T(x))-\psi(T(x)), 
\qquad\text{for }\rho\text{-almost every }x.
\]
The right-hand side takes at most three values.
Every nonempty relatively open subset of $S$ has positive
$\rho$-measure, so the almost-everywhere equality holds on a
dense subset of $S$. Continuity then implies that $\phi-\phi'$
takes values in the same finite set throughout $S$. Since $S$ is connected,
this difference is constant. Subtracting source means yields
the unique normalized potentials
$\phi_{\mu_\varepsilon},\phi_{\nu_\varepsilon}$. 

The limiting target can also be identified explicitly. Since
$\rho(\{q=0\})=0$ and the weights are $O(\varepsilon)$, convergence
of the three branches yields 
\begin{equation}\label{eq:sharp-limiting-target}
 \mu_\varepsilon,\nu_\varepsilon\rightharpoonup
 a_0\delta_{y_0}+(1-a_0)\delta_C,
 \qquad a_0:=\rho(S\cap\{q<0\})\in(0,1).
\end{equation} 
The limiting branches $h_{y_0},h_C$ have interface $\{q=0\}$ with
$V_qq=1$, so the same argument using Kantorovich duality gives a unique optimal plan for
this target measure supported on the two points $y_0,C$.

\medskip
\textbf{Step 3: Error estimates and sharpness.}
Put $f=\tilde\phi_1-\tilde\phi_0$.
On the part of the $A$ cell that remains after the perturbation,
$f=-m_A$. On the original $B$ cell, $f=-m_B$. For small $\varepsilon$,
\[
 \rho\bigl(L_A^\varepsilon(m_A,m_B)\bigr)
 \geq c_0-b\varepsilon\geq\frac{c_0}{2},
 \qquad
 \rho\bigl(L_B^\varepsilon(0,0)\bigr)\geq c_0.
\]
Hence by \eqref{eq:sharp-weight-bounds}, 
\begin{equation}\label{eq:sharp-potential-variance}
 \|\phi_{\mu_\varepsilon}-\phi_{\nu_\varepsilon}\|_{L^2(\rho)}^2
 =\Var_\rho(f)
 =\tfrac12\iint|f(x)-f(z)|^2\,\d\rho(x)\,\d\rho(z)
 \geq c_*\varepsilon^2. 
\end{equation}
Conversely,
$\|f\|_\infty\leq\max\{-m_A,m_B\}\leq C\varepsilon$, which proves \eqref{eq:universal-potentials}. 
By \eqref{eq:sharp-no-cross}, up to $\rho$-null sets,
the maps differ only on the part removed from the $A$ cell
and the part added to the $B$ cell. These two sets are disjoint,
each with $\rho$-mass $b\varepsilon$; and the maps agree elsewhere. Hence 
\begin{equation}\label{eq:sharp-map-error}
 \int_S\dcc^2(T_0,T_1)\,\d\rho
 =b\varepsilon\bigl(\dcc^2(A_\varepsilon,C)
                     +\dcc^2(C,B_\varepsilon)\bigr)\asymp\varepsilon,
\end{equation}
which proves \eqref{eq:universal-maps}.
If $\alpha>1/2$ and $\beta>1/4$, the ratios of the potential and
map errors to $W_1^\alpha$ and $W_1^\beta$ are comparable to
$\varepsilon^{1-2\alpha}$ and $\varepsilon^{1/2-2\beta}$,
respectively. Both ratios tend to infinity.
\end{proof}

The calculation in Step~1 also gives, for every fixed
$1\leq p<\infty$,
\[
 \begin{aligned}
 \norm{\dcc(T_{\mu_\varepsilon},T_{\nu_\varepsilon})}_{L^2(\rho)}
 &\asymp W_p^{p/[2(p+1)]}(\mu_\varepsilon,\nu_\varepsilon),\\
 \norm{\phi_{\mu_\varepsilon}-\phi_{\nu_\varepsilon}}_{L^2(\rho)}
 &\asymp W_p^{p/(p+1)}(\mu_\varepsilon,\nu_\varepsilon).
 \end{aligned}
\]
Thus, over the source class of
Theorem~\ref{thm:simultaneous-sharpness}, the possible $W_p$
exponents are at most $p/[2(p+1)]$ for maps and $p/(p+1)$ for
potentials. For $p=2$, these are $1/3$ and $2/3$.
These are restrictions on possible exponents, without matching
$W_p$ upper estimates asserted for $p>1$. In the Euclidean model
below, the source domain is John and need not be convex, so the
potential restriction does not address estimates restricted to
uniform measures on convex bodies.

\begin{corollary}[Sharpness on Carnot groups and Euclidean spaces]
\label{cor:carnot-euclidean-sharpness}
On every Carnot group with $\dim V_1\geq2$, of arbitrary step,
and inside every nonempty open set, the relations
\eqref{eq:universal-W1}--\eqref{eq:universal-maps} hold for a fixed
John domain $S$, the uniform source $\rho=\mm\restr S/\mm(S)$,
and a fixed compact target set $Y$.
In particular, for each fixed $d\geq2$, the potential exponent $1/2$
in the Euclidean John domain estimate of
\cite[Theorem~1.7]{LetrouitMerigot2026} is optimal.
\end{corollary}
\begin{proof}
Apply Theorem~\ref{thm:simultaneous-sharpness} with Haar measure on
$G$, and with Lebesgue measure when $G=\RR^d$.
The Euclidean model can be written explicitly. With $x_*=0$,
$y_0=e_1$ and $C=e_1+\tau e_2$, one has
\[
 h_y(x)=\tfrac12|x|^2-x\cdot y,
 \qquad p(x)=-x_1,\qquad q(x)=\tau x_2.
\]
Thus $V_p=-\partial_{x_1}$ and $V_q=\tau^{-1}\partial_{x_2}$,
and the source is the fixed union of overlapping balls along the
three-segment curve in Lemma~\ref{lem:sharp-open-chain}.
\end{proof}

\begin{corollary}[Sharpness on manifolds]
\label{cor:individual-sharpness}
On every complete smooth Riemannian manifold of dimension at least
two and every complete smooth fat sub-Riemannian manifold with horizontal rank at least two, the map
exponent in
Corollary~\ref{cor:map-stability} is optimal in every nonempty open
set. The potential exponent in
Theorem~\ref{thm:potential-stability}\textup{(iii)} is optimal on every Riemannian manifold of dimension at least two, every
fat sub-Riemannian manifold with horizontal rank at least two, and every step-two Carnot group.
\end{corollary}

\begin{remark}\label{rem:one-dimensional-exception}
	If $S\subseteq\RR$ is a bounded
	John domain, then $S$ is an interval. For its uniform source $\rho$,
	monotone rearrangement gives
	$(\phi_\mu-\phi_\nu)'=T_\nu-T_\mu$ a.e., and
	$(T_\mu,T_\nu)_\#\rho$ is an optimal coupling for $W_1$. Since the Kantorovich potentials are normalized, we have  
	\[
	\|\phi_\mu-\phi_\nu\|_{L^2(\rho)}
	\leq\osc_S(\phi_\mu-\phi_\nu)
	\leq\int_S|T_\mu-T_\nu|\,\d x
	=\mathcal L^1(S)W_1(\mu,\nu).
	\]
	Moreover, translating a Dirac target shows that this Lipschitz exponent is optimal. 
\end{remark}

\subsection{Sharpness for fixed marginals}
\begin{proposition}[Sharpness for a fixed transport problem]\label{prop:fixedsharp}
Under the hypotheses of Theorem~\ref{thm:simultaneous-sharpness},
inside every nonempty open set there exist a fixed compactly supported
source $\rho$ with a smooth positive density on a John domain and a
fixed target measure $\mu$ supported on two points. The optimal transport map $T$ from $\rho$ to $\mu$ is unique and there exist maps $S_h$ with $(S_h)_\#\rho=\mu$
such that
\begin{equation}\label{eq:fixedsharp}
 \int \dcc^2(T,S_h)\dd\rho\asymp h,\qquad
 \int c(x,S_h(x))\dd\rho-\int c(x,T(x))\dd\rho\asymp h^2.
\end{equation}
Consequently the power $2$ on $\calE$ in \eqref{eq:intro-coercivity} cannot be replaced by any smaller positive power, even for this one fixed transport problem.
\end{proposition}

\begin{proof}
	Choose $a=y_0$ and $b=C$ from the construction in
	Lemma~\ref{lem:sharp-open-chain}, and put
	\[
	k=c(x_*,a)-c(x_*,b),\qquad
	F(x)=c(x,a)-c(x,b)-k.
	\]
	Then $F(x_*)=0$ and $\der_{\mathcal H}F\ne0$ near $x_*$.
	Choose a sufficiently small fixed ball
	$S=B_{\mathrm{cc}}(x_*,r)$ whose closure lies in this smooth
	source region. Lemma~\ref{lem:local-completion} ensures the
	required minimizing curves to its center, even if the original
	metric is incomplete. These curves show that $S$ is John.
	Set $\rho=\omega\restr S/\omega(S)$. The branches $c(x,a)-k$ and $c(x,b)$ define $T=a$ on
	$\{F\leq0\}$ and $T=b$ on $\{F>0\}$. Since
	$\rho(\{F=0\})=0$, the minimizing atom is unique almost
	everywhere. For $\mu=T_\#\rho$, equality in the dual constraint
	makes $T$ optimal, and every optimal coupling is concentrated
	on its graph. Thus $T$ is the unique optimal transport map.
	
	A fixed flow box around $x_*$ and a finite cover of
	$\overline S$ yield, as in \eqref{eq:sharp-cell-mass},
	\[
	c_1(v-u)\leq\rho(\{u<F<v\})\leq c_2(v-u),
	\qquad -\delta<u<v<\delta,
	\]
	for fixed $\delta,c_1,c_2>0$.
	For small $h>0$, let
	\[
	A_h=S\cap\{-h<F<0\},\qquad
	B_{k_h}=S\cap\{0<F<k_h\},
	\]
	and choose $k_h$ so that
	\[
	\rho(A_h)=\rho(B_{k_h})\asymp h,\qquad k_h\asymp h.
	\]
	Such a choice follows from continuity and the two-sided estimate. Define $S_h=b$ on $A_h$, $S_h=a$ on $B_{k_h}$,
	and $S_h=T$ elsewhere. Since $\rho(A_h)=\rho(B_{k_h})$, $(S_h)_\#\rho=T_\#\rho=\mu$. Moreover, 
	\[
	\int_S\dcc^2(T,S_h)\,\dd\rho
	=2\dcc^2(a,b)\rho(A_h)\asymp h.
	\]
	
	Since the same equality of masses cancels the terms involving $k$, it follows that 
	\[
	\begin{aligned}
		\int_S\bigl[c(x,S_h(x))-c(x,T(x))\bigr]\,\dd\rho=\int_{A_h}(-F)\,\dd\rho+\int_{B_{k_h}}F\,\dd\rho\asymp h^2+k_h^2\asymp h^2.
	\end{aligned}
	\]
	For example, the part $\{-h<F<-h/2\}$ has mass at least
	$c_1h/2$ and contributes at least $c_1h^2/4$ to the first
	integral; its upper bound is $c_2h^2$. The same argument on
	$\{k_h/2<F<k_h\}$ treats the second integral. Together with
	$k_h\asymp h$, this proves \eqref{eq:fixedsharp}.
\end{proof} 

\begin{remark}\label{rem:heisenberg-l1-defect}
	Write $\HH^1=\mathbb C\times\RR$ in coordinates
	$[\zeta,t]$, with group law
	\[
	[\zeta,t]\cdot[\zeta',t']
	=[\zeta+\zeta',t+t'+\tfrac12\operatorname{Im}(\overline\zeta\zeta')].
	\]
	The horizontal metric makes the left translates of the two
	real coordinate vectors in $\mathbb C$ orthonormal.
	Let $\rho$ be uniform on the coordinate cube
	$S=\{[u+\iu v,t]:0\leq u,v,t<1\}$, and set
	$T(x)=x\cdot[1,0]$, $\mu=T_\#\rho$, and $Y=T(\overline S)$.
	By the horizontal projection bound, the weight
	$\psi([\zeta,t])=\operatorname{Re}\zeta$ gives
	$\psi^c([\zeta,t])=-\operatorname{Re}\zeta-\tfrac12$ on $S$,
	with equality in the dual constraint at $T(x)$.
	
	For $\eps_n=(2n)^{-1}$, let $F_n$ swap consecutive
	intervals of length $\eps_n$ in the central coordinate $t$,
	leaving the coordinates $u,v$ unchanged.
	Then $F_n$ preserves $\rho$ and satisfies
	$F_n(x)=x\cdot[0,\pm\eps_n]$, giving
	$\pi_n=(\Id,T\circ F_n)_\#\rho\in\Pi(\rho,\mu)$.
	Since $\eta\mapsto\dcc^2(e,[1,\eta])$ is smooth and minimized
	at zero, $\dcc^2(e,[1,\eta])=1+O(\eta^2)$.
	By left invariance and homogeneity,
	\[
	\left(\int_{S\times Y}\dcc(T(x),y)\,\d\pi_n(x,y)\right)^2
	\asymp\eps_n,\qquad
	\calD(\pi_n)=O(\eps_n^2).
	\]
	Thus the $L^1$ estimate in Remark~\ref{rem:riemannian-l1-defect}
	fails in general sub-Riemannian geometry, even with fixed marginals.
\end{remark} 

\subsection{Smooth transport families}\label{sec:smooth-response}
We next construct smooth families of optimal maps and densities.
The target measures satisfy $W_1(\mu_t,\mu_0)\asymp|t|$.
The map error has order $|t|$ for an auxiliary Riemannian distance.
For the Carnot--Carath\'eodory distance, it has order $|t|^{1/s}$. We use
Agrachev's theorem on points of smoothness to choose a region
where the cost is smooth
and the ball--box theorem to compare the two distances.

\begin{proof}[Proof of Theorem~\ref{thm:smooth-response}]
\textbf{Step 1: A smooth family of optimal maps.}
Lemma~\ref{lem:local-completion} and Agrachev's theorem on points of smoothness
\cite[Theorems~1 and~*, expanded version]{AgrachevSmooth2009}
give a pair of distinct points $x_*,y_*\in O$ where $c$ is smooth.
Uniqueness of the normal minimizing geodesic and invertibility of
the exponential differential persist under small perturbations of
both endpoints \cite[Lemma~2, expanded version]{AgrachevSmooth2009},
as used in Step~1 of Lemma~\ref{lem:sharp-open-chain}.
We may therefore choose coordinate neighborhoods $U,V\Subset O$
such that $c$ is smooth near $\overline U\times\overline V$,
and every pair in this product has a unique normal minimizing
geodesic with invertible exponential differential.
First variation gives $\der_xc(x,y)=-\lambda(x,y)$, where
$y=\Exp_x\lambda(x,y)$. Thus the coordinate matrix
\[
 B(x,y):=D^2_{xy}c(x,y)
       =-\bigl(D_\lambda\Exp_x(\lambda(x,y))\bigr)^{-1}
\]
is invertible. Here and below, derivatives and matrix products
are taken in these fixed source and target coordinates.

Take $U$ convex in source coordinates and choose $\kappa>0$
large enough that, for
\[
 \phi_0(x)=\der_xc(x_*,y_*)\cdot(x-x_*)
                 -\frac\kappa2|x-x_*|^2,
 \qquad
 A(x,y):=D^2_{xx}c(x,y)-D^2\phi_0(x),
\]
we have $A(x,y)\geq I$ on $\overline U\times\overline V$.
Let $v_*\in T_{y_*}M$ be any nonzero vector. Choose a linear
function $h$ in source coordinates with
$\der h=B(x_*,y_*)v_*$, and put $\phi_t=\phi_0+th$.
After shrinking $U$ and the interval of $t$, the implicit
function theorem gives a smooth family $T_t:U\to V$ with
\begin{equation}\label{eq:smooth-contact-equation}
 \der_xc(x,T_t(x))=\der\phi_t(x),
 \qquad T_0(x_*)=y_*.
\end{equation}
Differentiating in $t$ and in $x$ yields
\[
 B(x,T_t(x))\dot T_t(x)=\der h,
 \qquad
 B(x,T_t(x))D_xT_t(x)=-A(x,T_t(x)).
\]
In particular, $\dot T_0(x_*)=v_*$ and $D_xT_t$ is invertible.
For each $y\in V$, the function $x\mapsto c(x,y)-\phi_t(x)$
is strictly convex. It follows that $T_t$ is injective, hence
is a diffeomorphism onto its image.

For $y\in V$, set
\[
 \psi_t(y)=\min_{z\in\overline U}\{c(z,y)-\phi_t(z)\}.
\]
When $y=T_t(x)$ and $x\in U$, strict convexity makes $x$ the
unique minimizer. Hence
$c(x,T_t(x))=\phi_t(x)+\psi_t(T_t(x))$.
Choose $0\leq f\in C_c^\infty(U)$ with $\int f\,\d\omega=1$,
put $\rho=f\omega$, and set $\mu_t=(T_t)_\#\rho$.
Kantorovich duality proves optimality. For $\mu_t$-almost every
$y$, equality in the dual constraint forces its source to be
$T_t^{-1}(y)$. Thus the optimal plan is unique.
Choose a fixed compact target set $Y$ containing all
$T_t(\spt\rho)$ for small $|t|$ and contained in $T_0(U)$.
The support of $f$ and the interval of $t$ can be chosen small
enough to ensure this. On a neighborhood of $\spt\rho$, we also
have $\phi_t=\psi_t^c$ with the transform taken over $Y$;
choose $Y$ to contain the images of a slightly larger compact
source neighborhood. The pair $(\psi_t^c,\psi_t)$ is
$c$-conjugate on $U\times Y$, since $\psi_t=\phi_t^c$ on $Y$.
Its source component therefore agrees with the smooth function
$\phi_t$ near $\spt\rho$.

\medskip
\textbf{Step 2: The target measures.}
Change of variables under the diffeomorphisms $T_t$ shows that
$\mu_t=g_t\omega$, where $g_t$ is jointly smooth and supported
in the fixed compact set $Y$. The zero extension is smooth
because $f$ is compactly supported in $U$. In particular,
\[
 \int_M|g_t-g_0|\,\d\omega\leq C|t|,
 \qquad
 W_1(\mu_t,\mu_0)
 \leq\frac{\diam_{\mathrm{cc}}Y}{2}
       \int_M|g_t-g_0|\,\d\omega\leq C|t|.
\]
The second estimate couples the common density to itself and
moves the remaining mass within $Y$.
For the lower bound, choose $a\in C_c^\infty(V)$ with
$\der a(y_*)v_*>0$. Shrink the support of $f$ so that
$\der a(T_0(x))\dot T_0(x)$ is bounded below by a positive
constant there. The horizontal differential of $a$ is bounded,
so its zero extension is Lipschitz for the ambient
Carnot--Carath\'eodory distance. Therefore
\[
 \left.\frac{\der}{\der t}\right|_{t=0}\int a\,\d\mu_t
 =\int\der a(T_0(x))\dot T_0(x)\,\d\rho(x)>0.
\]
The Kantorovich--Rubinstein formula gives
$W_1(\mu_t,\mu_0)\asymp|t|$. Smoothness and the nonzero
velocity also give $\dR(T_t(x),T_0(x))\asymp|t|$ uniformly
on $\spt\rho$. This proves \eqref{eq:smooth-response-linear}
for any prescribed nonzero $v_*$.

\medskip
\textbf{Step 3: The Carnot--Carath\'eodory distance.}
Choose $v_*\notin\mathcal H^{s-1}_{y_*}$ in Step~1.
On a smaller neighborhood of $y_*$, take privileged coordinates
$z^y=(z^y_1,\ldots,z^y_n)$ centered at $y$, depending smoothly
on $y$, where $n=\dim M$. Their weights $w_i$ lie in
$\{1,\ldots,s\}$; the differentials of the weight-$s$ coordinates
span the annihilator of $\mathcal H^{s-1}_y$.
Such coordinates and the uniform ball--box bounds are standard;
see \cite{Bellaiche1996,ChoiPonge2019I,DonMagnani2023,NagelSteinWainger1985}.
For nearby $y,z$ in a fixed compact neighborhood,
\begin{equation}\label{eq:smooth-response-ballbox}
 \dcc(y,z)\asymp\max_i|z_i^y(z)|^{1/w_i}.
\end{equation}
All coordinates of $T_t(x)$ centered at $T_0(x)$ are $O(|t|)$.
For at least one weight-$s$ coordinate, the derivative at
$(t,x)=(0,x_*)$ is nonzero. By continuity, after shrinking the
support of $f$, that coordinate has absolute value comparable to $|t|$
uniformly for $x\in\spt\rho$ and small $|t|$.
Equation~\eqref{eq:smooth-response-ballbox} now gives
$\dcc(T_t(x),T_0(x))\asymp|t|^{1/s}$ on $\spt\rho$.
Integration proves \eqref{eq:smooth-response-weight}.
\end{proof}

The same example gives a restriction on defect estimates.
Unlike Proposition~\ref{prop:fixedsharp}, the next statement
allows the second marginal to vary.

\begin{proposition}[Scaling obstruction for coercivity]
\label{prop:smooth-defect-obstruction}
For the source and maps constructed above, fix the dual potentials
$\phi_0,\psi_0$ and the optimal map $T_0$. For
$\pi_t=(\Id,T_t)_\#\rho$, the defect and error satisfy
\begin{equation}\label{eq:smooth-defect-rates}
 \calD(\pi_t)\asymp |t|^2,
 \qquad \calE(\pi_t)\asymp |t|^{2/s}.
\end{equation}
Consequently, an estimate $\calE(\pi)^p\leq C\calD(\pi)$
for all plans with first marginal $\rho$ requires $p\geq s$.
In particular, \eqref{eq:intro-coercivity} fails in every
equiregular open set of step $s>2$.
\end{proposition}

\begin{proof}
On $Y\subset T_0(U)$, the minimizer defining $\psi_0(y)$ is
$z(y)=T_0^{-1}(y)$, so $\psi_0$ is smooth.
Recall $A=D^2_{xx}c-D^2\phi_0$ and $B=D^2_{xy}c$, and set
$R_0(x,y)=c(x,y)-\phi_0(x)-\psi_0(y)$.
At $y=T_0(x)$, we have $R_0=0$, $D_yR_0=0$, and
\[
 D^2_{yy}R_0(x,T_0(x))
 =B(x,T_0(x))^{\mathsf T}A(x,T_0(x))^{-1}B(x,T_0(x)).
\]
Indeed, differentiating the equation for the minimizing $z(y)$
gives $D_yz=-A^{-1}B$; differentiating
$D_y\psi_0(y)=D_yc(z(y),y)$ gives the displayed identity.
The matrix is uniformly positive definite near $\{(x,T_0(x)):x\in\spt\rho\}$. Taylor's formula and
$|T_t(x)-T_0(x)|\asymp|t|$ therefore give
$R_0(x,T_t(x))\asymp|t|^2$. This proves the first estimate
in \eqref{eq:smooth-defect-rates}; the second follows from
\eqref{eq:smooth-response-weight}.
Their quotient is
$\calE(\pi_t)^p/\calD(\pi_t)\asymp|t|^{2p/s-2}$,
which is unbounded if $p<s$.
\end{proof}

\begin{corollary}[Distances after projection onto quotient groups]
\label{cor:quotient-response}
Let $G$ be a Carnot group with stratified Lie algebra
$V_1\oplus\cdots\oplus V_s$ and $\dim V_1\geq2$.
For $1\leq j\leq s$, let
\[
 G_j=G/\exp(V_{j+1}\oplus\cdots\oplus V_s),
 \qquad \pi_j:G\longrightarrow G_j.
\]
Here $\pi_j$ is the canonical projection, $G_s=G$, and
$\pi_s=\Id$. The subgroup $\exp(V_{j+1}\oplus\cdots\oplus V_s)$
is normal, and $G_j$ has stratified Lie algebra
$V_1\oplus\cdots\oplus V_j$. Let $\dist_j$ be its
Carnot--Carath\'eodory distance for the given inner product on $V_1$. Fix $j$ and a nonempty open set $O\subseteq G$. There exist
$t_0>0$, a source probability measure $\rho=f\mm$ with
$0\leq f\in C_c^\infty(O)$, and probability measures $\mu_t=g_t\mm$
for $|t|<t_0$. The function $(t,y)\mapsto g_t(y)$ is smooth.
All $g_t$ have support in one compact subset of $O$. The optimal maps $T_t$
from $\rho$ to $\mu_t$ are unique and smooth, and satisfy
\begin{equation}\label{eq:quotient-response}
 \norm{\dist_j(\pi_jT_t,\pi_jT_0)}_{L^2(\rho)}
 \asymp |t|^{1/j}
 \asymp W_1^{1/j}(\mu_t,\mu_0).
\end{equation}
Here $W_1$ uses the distance on $G$, and the family may depend on $j$.
\end{corollary}

\begin{proof}
In Steps~1 and~2 above, choose
$v_*=(D L_{y_*})_e v_j$ for $0\ne v_j\in V_j$, where
$L_y$ is left translation by $y$ and $e$ is the identity.
The projected velocity has a nonzero component in the top layer
of the step-$j$ group $G_j$. This remains true near $x_*$.
Apply the ball--box argument of Step~3 in $G_j$ and use the
same $W_1$ estimate from Step~2.
\end{proof}

For $j<s$, the projection in \eqref{eq:quotient-response} discards
the layers $V_{j+1},\ldots,V_s$. Prescribing a velocity in $V_j$
at one point does not exclude components in these layers at
nearby source points,
so it does not give the same exponent for the full distance on $G$.

\subsection{Open questions}\label{sec:boundary}
Does the $W_1^{1/4}$ map estimate in
Corollary~\ref{cor:map-stability} hold on every step-two Carnot
group, including nonfat groups, under the same source assumptions?
More generally, does it extend to complete smooth two-generating
manifolds? The cusp in Remark~\ref{rem:nonfat-obstruction}
obstructs our semiconcavity argument, but does not disprove stability.
Our proof would need uniform estimates near pairs joined by
abnormal minimizers. Uniqueness of optimal maps alone gives
no such estimate.
A related direction is extension to ideal manifolds, namely
complete structures without nonconstant abnormal minimizers.
Their cost is locally semiconcave off the diagonal
\cite[Theorem~5.9]{FigalliRifford10}, but
Lemma~\ref{lem:varying-controls} also requires regular minimizers
in the rescaled limits.

\medskip
Does Theorem~\ref{thm:potential-stability}\textup{(ii)} hold with
$\alpha=1/2$, with the same dependence of the constant as in that
theorem?
Our approximation constants are not controlled uniformly as
$\alpha\uparrow1/2$. The estimates for $\alpha<1/2$ and
Theorem~\ref{thm:simultaneous-sharpness} identify $1/2$ as the
supremum for sources on John domains when the horizontal rank is at least two.
Part~\textup{(iv)} restricts finite atomic counterexamples:
for a fixed source and compact target region, with $W_1(\mu_0^n,\mu_1^n)>0$,
\[
\begin{gathered}
\frac{
 \norm{\phi_{\mu_1^n}-\phi_{\mu_0^n}}_{L^2(\rho)}}
 {W_1^{1/2}(\mu_0^n,\mu_1^n)}\longrightarrow+\infty
 \quad\Longrightarrow\quad
 \min\{\#\spt\mu_0^n,\#\spt\mu_1^n\}\longrightarrow+\infty.
\end{gathered}
\]

\medskip
The two local constructions give restrictions arising from changes
in Laguerre cells and from the anisotropic scaling of the distance.
It remains to determine which map estimates hold under both
restrictions in higher step.
In particular, on equiregular structures of steps three and four
where existence and uniqueness of optimal maps are available,
does the map exponent $1/4$ persist even though
\eqref{eq:intro-coercivity} fails? For unrestricted higher-step
structures, establishing existence for every compactly supported
absolutely continuous source and arbitrary compact target is
already an open problem.
More generally, what are the best exponents when all transported
pairs stay in a fixed region where the cost is smooth, with quantitative bounds
on the inverse exponential differential?

\medskip
For potentials, Theorem~\ref{thm:potential-stability}\textup{(i)}
gives a logarithmic modulus on every equiregular structure.
Can this be improved to a positive power of $W_1$ on every
fixed higher-step Carnot group? The failure of horizontal
semiconcavity on the Engel group does not decide this question.
The logarithmic rate arises when selecting a scale of affine
approximation, and we have no matching lower bound.

\medskip
For maps on a Carnot group of step greater than two, is there any
positive uniform H\"older exponent under only an upper density
bound on the source? The question remains open even if we consider
only targets with unique optimal maps. Our examples give upper
bounds on the exponent, but do not exclude all positive exponents.

\section*{Use of generative AI}
The authors used ChatGPT, Codex, and Kimi extensively in this work.
They used these tools to explore mathematical questions, test proof
strategies, check arguments and references, and draft and revise
the text. The authors chose the research questions and directed
the development of the proofs. They reviewed the final manuscript
and take responsibility for its arguments, references, and conclusions.

\enlargethispage{\baselineskip}

\enlargethispage{3\baselineskip}

\end{document}